\documentclass[12pt]{amsart}
\usepackage{a4wide}
\usepackage{times}
\usepackage{bbm}
\usepackage{mathtools, amssymb}
\usepackage{graphicx,xspace}
\usepackage{epsfig}
\usepackage{dsfont}
\usepackage[usenames,dvipsnames]{xcolor}
\usepackage{tikz}
\usepackage[T1]{fontenc}
\usepackage[utf8]{inputenc}
\usepackage{array,multirow} %
\usepackage{bm}
\usepackage{kpfonts} 
\usepackage{dsfont} 
\usepackage{setspace} 
\usepackage{cancel}
\usepackage{tipa}
\usepackage{stmaryrd}
\usepackage{csquotes}
\usepackage{footmisc}

\usepackage{hyperref,pifont}
\usepackage{todonotes}
\usepackage{dsfont}
 \usepackage[capitalize]{cleveref}

\newcommand{\ignorer}[1]{}
\def\and{\ \wedge\ }

\newcommand{\revision}[1]{{#1}}

\theoremstyle{plain}

\newtheorem{lemma}{Lemma}
\newtheorem{theorem}[lemma]{Theorem}
\newtheorem{corollary}[lemma]{Corollary}
\newtheorem{proposition}[lemma]{Proposition}
\newtheorem{definition}[lemma]{Definition}

\theoremstyle{remark}
\newtheorem{remark}[lemma]{Remark}
\newtheorem{example}[lemma]{Example}

\renewcommand\Re{\operatorname{Re}}
\def\eps{\varepsilon}
\renewcommand\epsilon{\varepsilon}

\def\R{\mathbb{R}}
\def\Q{\mathbb{Q}}

\def\P{\mathbb{P}}

\def\cL{\mathcal{L}}

\def\Z{\mathbb{Z}}
\def\Q{\mathbb{Q}}

\def\One{\bm{1}}

\DeclareMathOperator{\Right}{Right}
\DeclareMathOperator{\Left}{Left}

\DeclareMathOperator{\Cat}{Cat}

\newcommand{\Old}[1]{}

\def\cA{\mathcal{A}}
\def\cB{\mathcal{B}}

\newcommand\restr[2]{{%
		\left.\kern-\nulldelimiterspace %
		#1 %
		\right|_{#2} %
	}}

\title{Tree-indexed sums of Catalan numbers}

\author[A. Bostan]{Alin Bostan}
   \address[AB]{Inria, Sorbonne Université, LIP6, F-75252 Paris, France}
          \email{alin.bostan@inria.fr}

 \author[V. Féray]{Valentin Féray}
       \address[VF]{Université de Lorraine, CNRS, IECL, F-54000 Nancy, France}
       \email{valentin.feray@univ-lorraine.fr}

 \author[P. Thévenin]{Paul Thévenin}
 \address[PT]{Université d'Angers, CNRS, LAREMA, SFR MATHSTIC, F-49000 Angers, France}
       \email{paul.thevenin@univ-angers.fr}
 

\makeatletter
\@namedef{subjclassname@2020}{\textup{2020} Mathematics Subject Classification}
\makeatother


\begin{document}

\maketitle 

\begin{abstract}
We consider a family of infinite sums of products of Catalan numbers, indexed by trees. 
We show that these sums are polynomials in $1/\pi$ with rational coefficients; the proof is effective and provides an algorithm to explicitly compute these sums. 
Along the way we introduce parametric liftings of our sums, and show that they are polynomials in the complete elliptic integrals of the first and second kind. Moreover, the degrees of these polynomials are at most half of the number of vertices of the tree.  
The computation of these tree-indexed sums is motivated by the study of large meandric systems, which are non-crossing configurations of loops in the plane.
\end{abstract}

\section{Introduction}

\subsection{Our main result}
Let $T$ be an 
(unrooted) unlabeled tree, 
where we potentially allow one single half-edge 
(i.e.~an edge having only one extremity, see \cref{fig:first_examples}, left).
Denoting its vertex-set by $V(T)$ and its edge-set by
$E(T)$, we consider the following formal sum
\begin{equation}\label{eq:def_ST}
 S(T)(t) \coloneqq
 \sum_{(x_e) \in \mathbb \Z_+^{E(T)}}  \left( \prod_{v \in V(T)}
  \Cat_{X_v} t^{X_v} \right),
  \end{equation}
where $\Cat_n \coloneqq \frac{1}{n+1} \binom{2n}{n}$ is the $n$-th Catalan number, and where for any vertex $v$ of $T$ we set $X_v=\sum_{e \ni v} x_e$ (the sum is taken over edges $e$ incident to $v$). To avoid heavy notation, we write $S(T)$ instead of $S(T)(t)$.

Let us observe immediately that, since the map sending $(x_e)_{e \in E(T)}$ to $(X_v)_{v \in V(T)}$ is injective\footnote{\revision{Given $(X_v)_{v \in V(T)}$, we can compute $(x_e)_{e \in E(T)}$ starting from the edges adjacent to the leaves (such an edge $e$ satisfies $x_e=X_\ell$, where $\ell$ is the leaf endpoint of $e$), and then proceeding inductively along the tree. If the tree contains a half-edge, the corresponding weight $w_e$ is the last one to be computed.}} for all $T$, we have for $t>0$:
\begin{align*}
S(T) \leq \sum_{(X_v) \in \mathbb \Z_+^{V(T)}}  \left( \prod_{v \in V(T)}
  \Cat_{X_v} t^{X_v} \right) =  \left( \sum_{n \geq 0} \Cat_n t^n \right)^{|V(T)|},
\end{align*}
which is finite as soon as $t \le 1/4$.
These sums, and in particular their evaluation at $t=1/4$, arise naturally in the study
of a percolation model called the {\em infinite noodle}, to which we return later in the 
introduction (\cref{ssec:motivation}). 
For the moment, let us have a close look at three basic examples (\cref{fig:first_examples}). 

If $P_1^h$ is a tree with one vertex and a single half-edge, then
\begin{align*}
S(P_1^h) = \sum_{x \in \Z_+} \Cat_x t^{x}.
\end{align*}

If $P_2$ is the path with two vertices and one edge, we get
\begin{align*}
S(P_2) = \sum_{x \in \Z_+} \left( \Cat_x t^{x}\right)^2.
\end{align*}

Finally, if $P^h_2$ is the path with two vertices, one normal edge,
and an extra half-edge on one of the vertices, then we have
\begin{align*}
S(P^h_2) = \sum_{(x_1,x_2) \in \Z_+^2} \Cat_{x_1} t^{x_1} \Cat_{x_1+x_2} t^{x_1+x_2}.
\end{align*}
In this last example, $x_1$ is the variable associated to the normal edge and $x_2$ is the variable associated to the half-edge.
\begin{figure}
  \begin{center}
    \includegraphics[scale=2]{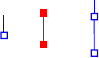}
  \end{center}
  \caption{The three trees $P_1^h, P_2, P^h_2$. Trees without half-edge are drawn in red 
  with filled vertices, while trees with one half-edge are drawn in blue with empty vertices.}
  \label{fig:first_examples}
\end{figure}

The first sum is nothing but the standard generating series of the Catalan numbers,
\[S(P_1^h) =\frac{1-\sqrt{1-4t}}{2t},\]
and in particular 
\[S(P_1^h)(1/4)=2.\]
 The values at $t=1/4$ of $S(P_2)$ and $S(P^h_2)$ can be found by standard computer algebra 
algorithms, implemented in computer algebra systems such as \texttt{Maple} or \texttt{Mathematica}: \[ S(P_2)(1/4)=\frac{16}{\pi}-4 \quad \text{and}  \quad S(P_2^h)(1/4)=\frac{8}{\pi} \, .\]
See \cref{ssec:intro_computer_algebra} for details on such algorithms, and  \cref{sec:first_examples} for a human-readable proof of these two identities.

Our main result (\cref{thm:sum-pi}) is a structural result on $S(T)$, which implies that the special value $S(T)(1/4)$
always is a polynomial with rational coefficients in $1/\pi$, as can be observed on these first examples.
To state this result more precisely, we first need to recall 
the definition of the {\em Gauss hypergeometric functions} ${}_{2}^{}{{}{{}{F_{1}^{}}}}$.
For parameters $a,b,c \in \R$ with $-c \notin \mathbb{N}$, we set
\begin{equation}\label{eq:def_hyper}
  {}_{2}^{}{{}{{}{F_{1}^{}}}}(a,b;c;z)\coloneqq  \sum_{n \ge 0} \frac{a^{\, \underline n} b^{\, \underline n}}{c^{\, \underline n}} \frac{z^n}{n!},
\end{equation} 
where $x^{\, \underline n}\coloneqq x(x+1) \cdots (x+n-1)$ is the $n$-th raising power of $x$. 
Define the two algebras
\[
\cA \coloneqq  \Q\Big[t,t^{-1}, \, {}_{2}^{}{{}{{}{F_{1}^{}}}}\big(-\tfrac12,-\tfrac12; 1;16t^2\big)\, , {}_{2}^{}{{}{{}{F_{1}^{}}}}\big(-\tfrac12,\tfrac12; 2;16t^2\big) \Big], \text{ and }\cB=\cA \big[\sqrt{1-4t}\big]
\]
i.e.~$\cA$ (resp.~$\cB)$ is the set of polynomials in the two particular hypergeometric functions
\[
H_1(t) \coloneqq {}_{2}^{}{{}{{}{F_{1}^{}}}}\big(-\tfrac12,-\tfrac12; 1;16t^2\big)
\quad \text{and} \quad
H_2(t) \coloneqq {}_{2}^{}{{}{{}{F_{1}^{}}}}\big(-\tfrac12,\tfrac12; 2;16t^2\big)
\]
 (resp.~in $H_1(t)$, $H_2(t)$ and in $\sqrt{1-4t}$), whose coefficients are Laurent polynomials in $t$.

Note that $H_1(t)$ and $H_2(t)$ are closely related to elliptic integrals,
namely to the complete elliptic integrals of first and second kinds,
\begin{align*}
K(t) = \int_0^{\pi/2} \frac{1}{\sqrt{1-t^2 \sin^2\theta}} \ d\theta &=
  \frac\pi2 \cdot  {}_{2}^{}{{}{{}{F_{1}^{}}}}\big(\tfrac12, \tfrac12; 1; t^2\big),
    \\
E(t) = \int_0^{\pi/2} {\sqrt{1-t^2 \sin^2\theta}} \ d\theta  &=
  \frac\pi2 \cdot {}_{2}^{}{{}{{}{F_{1}^{}}}}\big(-\tfrac12, \tfrac12; 1; t^2\big) .
\end{align*}
Indeed, using contiguous relations for hypergeometric functions (see e.g. \cite[Section $2.5$]{AAR99}),
we can write
\begin{equation}\label{eq:H12_EK}
H_1(t) = 
\frac{2 \,  \big(16 t^{2}\,-1 \big)\, K \left(4 t \right) + 4 E \left(4 t \right)}{\pi},
\quad
H_2(t)  = 
\frac{\left(16 t^{2}-1\right) \, K \left(4 t \right) + \left(16 t^{2}+1\right) \, E \left(4 t \right)}{12 \pi  \,t^{2}},
\end{equation}
which imply that 
\[
\cA = \Q\Big[t,t^{-1}, \, H_1(t), H_2(t) \Big] = \Q\Big[t,t^{-1}, \, \tfrac1{\pi} K \left(4 t \right), \tfrac1{\pi} E \left(4 t \right) \Big] .
\]

Moreover, it follows from a careful asymptotic analysis that
the complete elliptic integrals of first and second kinds are 
algebraically independent\footnote{\label{fn1} In fact, a stronger result is known: namely, the evaluations of the functions $K$ and $E$
at some algebraic argument (namely, $1/\sqrt{2}$) are algebraically independent over $\Q$.
Indeed, it is known that $\pi$ and $\Gamma(1/4)$ are algebraically independent, see~\cite{Chudnovsky76} or~\cite[Thm.~14]{Waldschmidt06}. 
Hence the same holds for the numbers $K(1/\sqrt{2}) = \frac{\Gamma (1/4)^2}{4\pi^{1/2}}$ and $E(1/\sqrt{2}) = \frac{\pi^{3/2}}{\Gamma(1/4)^2} + \frac{\Gamma(1/4)^2}{8\pi^{1/2}}$. Note that this algebraic independence is false over $\Q(\pi)$, as 
$2 \, K(1/\sqrt{2}) \, E(1/\sqrt{2}) - K(1/\sqrt{2})^2 = \pi/2$ (by Legendre's relation, or by the previous equalities). Nonetheless, the algebraic independence over $\mathbb Q$ of these values at algebraic arguments implies the algebraic independence of the series $E(t)$ and $K(t)$ over $\mathbb Q(t)$, and thus over $\mathbb C(t)$ (by \cite[Lemma 7.2]{ABD19}).} over $\mathbb C(t)$, see Appendix~\ref{appendix:alg_independence} for details.
The same holds for $H_1(t)$ and $H_2(t)$, and in particular 
one cannot remove any of the four functions $H_1(t)$, $H_2(t)$, $\tfrac1{\pi}K \left(4 t \right)$, $\tfrac1{\pi}E \left(4 t \right)$ from our two lists of generators of~$\cA$.

We will see $\cB$ (and hence also its subalgebra $\cA$) as a filtered algebra by 
assigning degree~$2$ to $H_1(t)$ and $H_2(t)$,
and degree 1 to $\sqrt{1-4t}$.
The degree of an element $b$ in~$\cB$ is then defined as the minimal degree of a polynomial
representation of $b$ in terms of these generators, and is denoted $d_b$.

At $t=1/4$, the values of the above hypergeometric functions are given by the Gauss summation identity (writing $\Gamma$ for the usual gamma function): for $K \in \{0,1\}$, we have
\begin{equation} \label{eq:values_elliptic}
{}_{2}^{}{{}{{}{F_{1}^{}}}}\big(-\tfrac12,K-\tfrac12; K+1;1\big)
= \frac{\Gamma(K+1)\Gamma(2)}{\Gamma(K+\tfrac32) \Gamma(\tfrac32)}
= \begin{cases}
4/\pi & \text{ for }K=0;\\
8/(3\pi)  & \text{ for }K=1.
\end{cases}
\end{equation}
Consequently, if $F$ is an element \revision{of degree at most $d$} in $\cB$, then $F(1/4)$ is a polynomial in $1/\pi$ \revision{of degree at most $\lfloor d/2 \rfloor$}.
\revision{Indeed, both hypergeometric generators of $\cB$ have been assigned degree $2$ and specialize at $t=1/4$ to polynomials in $1/\pi$ of degree $1$.} 
We can now state our main result.
 \begin{theorem}\label{thm:sum-pi}
If $T$ is a tree without half-edges, then $S ( T ) \in \cA$.
If $T$ is a tree with exactly one half-edge, then $S ( T ) \in \cB$.
In both cases, denoting by $V_T$ the number of vertices in $T$, we have
$d_{S(T)} \leq V_T$.
As a consequence, $S ( T )(\frac{1}{4})$ is a polynomial in $1/\pi$
with rational coefficients, of degree at most $\left\lfloor {V_T}/{2} \right\rfloor$.
 \end{theorem}
\revision{Compared to the results obtained using standard computer algebra methods, the strength of this result is that it is a uniform structural theorem.  
Indeed, for a fixed tree $T$, one might hope to use creative telescoping or differential equations (see Section~\ref{ssec:intro_computer_algebra} for details), but this would not explain why $S(T)$ (resp.~$S ( T )(\frac{1}{4})$)
 always lies in $\cA$ or $\cB$ (resp.~$\mathbb Q[\tfrac1{\pi}]$). 
 In some sense, Theorem~\ref{thm:sum-pi} gives a uniform explanation of a phenomenon that computer algebra alone would only suggest case by case.}
In addition, \revision{our proof is constructive and yields} an algorithm that allows us to compute exactly all these sums.
This algorithm is described in pseudo-code language in Appendix~\ref{appendix:algo}.

We also obtain a compact expression of the evaluation $S ( T )(\frac{1}{4})$ when $T$ is a \emph{star tree}: if $T=T_s$ is the tree
with a central vertex connected to $s$ leaves, then we have, for $s\geq 0$,
\begin{equation}
\label{eq:stars}
S(T_{s+3})(1/4)  = 
\frac{64}{\pi} \cdot \left(
\sum_{k=0}^s   {\binom{s}{k}} \cdot \frac{1}{\left(2 k +1\right) \left(2 k +3\right) \left(2 k +5\right)}\right) .
\end{equation}
See Section~\ref{sec:stars} for the proof, and for related results. Interestingly, this formula does not follow from
our general approach, but rather relies on hypergeometric series identities.

In~\cref{appendix:explicit_sums}, we provide explicit expressions for the sums $S(T)$ and their evaluations $S(T)\left(
\frac14 \right)$ when $T$ is one of the 24 trees with $V(T) \leq 7$ vertices (without half-edges).
By inspecting these expressions, we make the following observations and conjectures.
\begin{itemize}
\item The degree bounds in \cref{thm:sum-pi} are rarely reached: e.g. among the 11 trees with 7 vertices, only 3 trees exhibit a polynomial in $1/\pi$ of degree $\left\lfloor \frac{7}{2} \right\rfloor=3$.
\item However, the bound $\left\lfloor \frac{V_{T}}{2} \right\rfloor$ seems to be tight in the following sense. For every~$n$, there exists a tree $T_n$ with $n$ vertices such that $S(T_n)\left(
\frac14 \right)$ is a polynomial in $1/\pi$ of degree $\left\lfloor \frac{n}{2} \right\rfloor$: when $n$ is even, then a plausible candidate for $T_n$ is the ``line tree''~$T_{n,a}$ (generalizing $T_{2,a}, T_{4,a}$ and $T_{6,a}$ in~\cref{fig:trees2-6}), while when $n$ is odd a candidate for $T_n$ is the ``long fork'' $T_{n,b}$ (generalizing $T_{3,b}=T_{3,a}, T_{5,b}, T_{7,b}$ in~\cref{fig:trees2-6,fig:trees7}). We could not prove this.
However, our proof involves a generalization of \cref{thm:sum-pi}
(given in~\cref{thm:sum-pi-generalized}), for which we could prove an analogous statement for a model of decorated tree, see~\cref{prop:long stars maximize the degree}.
\item The constant term of $S(T)\left(\frac14 \right)$, when written as a polynomial in $1/\pi$, seems to always be an \emph{integer} number; also, the numerators and denominators of the other coefficients have \emph{surprisingly small} prime factors.
\item The degree of  $S(T)\left(\frac14 \right)$ in $1/\pi$ is never $0$; in particular, $S(T)$ is a transcendental power series and  $S(T)\left(\frac14 \right)$ is a transcendental number.
\item The value of $S(T)\left(\frac14 \right)$ is increasing when $T$ goes through the lists in~\cref{fig:trees2-6,fig:trees7}, which are ordered 
 according to their degree sequences sorted in nondecreasing order (using the lexicographic order on such sequences), breaking ties in an appropriate way.
\end{itemize}
We leave the proofs of these observations to further work.
	
\begin{remark}\label{rem:triangle}
One can define~$S(G)$ for general graphs $G$, and not only for trees,
in the exact same fashion.
However, it turns out that $S ( G )(\frac{1}{4})$ does not always belong to $\Q[\tfrac1{\pi}]$.
For example, if $G$ is the cycle graph $C_3$ with three vertices, then $S ( C_3 )(\frac{1}{4})$ lies in $\Q[\sqrt 2] \setminus \Q$ (see \cref{rem:residue}). 
More precisely, we can prove that
\begin{equation}\label{eq:SC3}
S(C_3)
=
\frac{ 3- 3 \cdot {}_{2}^{}{{}{{}{F_{1}^{}}}}\big(-\tfrac14, \tfrac14; \tfrac32;16t^2\big)}{2 t^{2}}
=
\frac{(1-4t)^{3/2} - (1+4t)^{3/2} + 12t}{8t^3},
\end{equation}
yielding $S ( C_3 )(\frac{1}{4}) = 24-16\sqrt{2}$.
The series $S(G)$ might also diverge at $t=1/4$ for some graphs; e.g., for the complete graph with 4 vertices, we have 
\[S(G)(t)= (3 H_1^2 - 12t^2 H_2^2 - 4H_1 + 2(32t^2 - 1)H_2 - 48t^2 + 3)/(128t^4 - 8t^2) . \]
Investigating $S(G)$ for general graphs~$G$ (even for trees with
more than one half-edge) is outside the scope of this paper. We plan to investigate such sums in a subsequent work.
\end{remark}
	
\begin{remark}\label{rem:residue}
	Using the methods of~\cite{BoLaSa17}, the sum $S(T)$ can be written\footnote{Technically, to use the machinery from~\cite{BoLaSa17}, one first needs to rewrite $\Cat(n)$ as $\binom{2n+1}{n+1} - 2 \binom{2n}{n-1}$.\label{footnote:BLS}} as the (formal) residue w.r.t. $(z_v)_{v \in V(T)}$ of the multivariate rational function
\begin{equation} \label{eq:residue}
\frac{\displaystyle{\prod_{v \in V(T)}} \frac{(2 z_{v}+1) (1-z_v)}{z_v^2}}{\displaystyle{\prod_{e=\{v_1,v_2\} \in E(T)}} \left( 1-t^2 \frac{(1+z_{v_1})^2}{z_{v_1}} \frac{(1+z_{v_2})^2}{z_{v_2}}\right)} ,
\end{equation}	
thus it is equal to the multiple integral of this rational function over a cycle in~$\mathbb{C}^{V(T)}$, see Corollary 2.7 and Proposition 2.9 in~\cite{BoLaSa17}.
Therefore, the properties of $S(T)$ are mainly governed by the geometry of the denominator of~\eqref{eq:residue}. Each curve of the form $xy = t^2 {(1+x)^2} {(1+y)^2}$ is an elliptic curve for $|t|<\frac14$, which degenerates into a genus-0 curve for $t=\pm\frac14$.
The emergence in~\cref{thm:sum-pi} of (polynomials in) the complete elliptic integrals could be linked to the fact that $S(T)$ can be written as an integral of a rational function whose denominator is a product of elliptic curves. 
However, although it provides intuition, a geometric proof of \cref{thm:sum-pi} based on \eqref{eq:residue} remains to be found. 
\end{remark}

\subsection{Motivation: shape probabilities in the infinite noodle} \label{ssec:motivation}
Our initial interest in the sums $S(T)$ comes from discrete probability theory,
and more precisely from a percolation model called the {\em infinite noodle}.
It was introduced in \cite{InfiniteNoodle}, and further studied in \cite{feray2023meandric,borga2023geometry_meandric}.
Informally, we consider the bi-infinite discrete line $\mathbb Z$ and choose for each $i$ in $\mathbb Z$
two directions, one of which is drawn as an arrow above the line and the other as an arrow below the line.
All directions are chosen independently of each other, and uniformly among $\Left$ and $\Right$.
Then we connect arrows pointing to the right with arrows pointing to the left, in the unique non-crossing way 
(independently above and below the line).
An example is shown in \cref{fig:infinite_noodle}, left.
\begin{figure}[!ht]
  \begin{center}
    \includegraphics{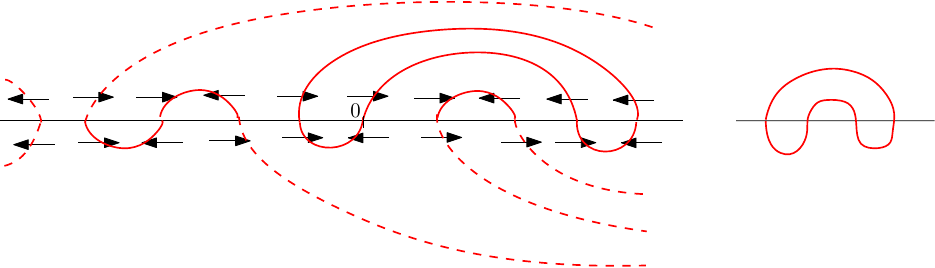}
  \end{center}
  \caption{A realization of the infinite noodle, and the corresponding shape of the component of $0$.}
  \label{fig:infinite_noodle}
\end{figure}

We are interested in the component of $0$. In particular, it is conjectured in \cite{InfiniteNoodle}
that it is finite almost surely~and in \cite[Conjecture 1.10]{borga2023geometry_meandric} that the probability that it has size at least $k$
decays as $k^{-(2\sqrt 2 -1)/7+o(1)}$.
The sums $S(T)$ come into the picture when computing the probability that the component of~$0$ has a given {\em shape}.
Being able to compute these sums does not solve the above conjectures,
but might give insight on the problem.

Informally, the shape of the component of $0$ is obtained by considering the relative order of its intersection points
with the real line, but forgetting the actual values of these points; see \cref{fig:infinite_noodle}, right.
The shape is what is called {\em a meander}, i.e.~a simple curve crossing the real line in multiple points (but only with simple crossings),
up to continuous deformation. 
We refer to \cite{feray2023meandric} for precise definitions of meanders
and of shape of components of the infinite noodle.

Writing $S_0$ for the shape of the component of $0$ and $M$ for a meander,
the probability $\mathbb P[S_0=M]$ can be expressed as a product $\prod_{i=1}^d S(T_i)$ evaluated at $t=1/4$, 
where $(T_i)_{i \le d}$ is a collection of trees, exactly one of them without a half-edge, associated with the meander $M$
(this follows from \cite[Section 4.1]{feray2023meandric}; details are given in \cref{sec:infinite_noodle}).
Hence, \cref{thm:sum-pi} implies the following statement.
\begin{corollary} \label{coro:proba_meanders}
  For any meander $M$, the probability $\mathbb P[S_0=M]$ belongs to $\mathbb Q[\tfrac1{\pi}]$.
\end{corollary}

In addition, our algorithm provides a way to compute these probabilities exactly.

\subsection{What computer algebra can (and cannot) do}
\label{ssec:intro_computer_algebra}
For a fixed tree~$T$, one can obtain different types of information about the sum $S(T)(t)$ in~\eqref{eq:def_ST}, using various computer algebra tools. Ideally, a user equipped with a computer algebra system such as \href{https://www.maplesoft.com/}{\texttt{Maple}}, \href{https://www.wolfram.com/mathematica/}{\texttt{Mathematica}}, \href{https://www.sagemath.org/}{\texttt{SageMath}}, etc.
might hope to get directly a ``formula'' for the power series $S(T)(t)$ and for the value $S(T)(1/4)$, by simply appealing to a one-line command. 

While relatively simple sums, such as those in \cref{sec:first_examples}, can be computed this way, the situation is far less simple in general. We now briefly explain what computer algebra can (and cannot) do, with respect to sums such as $S(T)(t)$ and $S(T)(1/4)$, relying on a number of advanced functionalities for symbolic summation and integration. More detailed accounts dedicated to this topic are~\cite{Schneider07,Salvy19,KoWo21}.

The sum $S(T)(t)$ in~\eqref{eq:def_ST} belongs to the class of \emph{definite multiple sums with parameters}; the parameter here is~$t$. Such sums can be tackled algorithmically using the so-called ``holonomic systems approach'' initiated by Zeilberger in the 1990s~\cite{Zeilberger90}. Its central algorithmic paradigm is called ``creative telescoping''; it is an important and very active research area in computer algebra. The central concept is that of \emph{D-finite functions} (power series satisfying linear differential equations with polynomial coefficients) and of \emph{P-recursive sequences} (sequences satisfying linear recurrence equations with polynomial coefficients, or equivalently, coefficient sequences of D-finite functions), to which the recent book~\cite{Kauers2023} is dedicated.

In fact, the coefficient $u_n = u_n(T)$ of $t^n$ in $S(T)(t)$ in~\eqref{eq:def_ST} is of a very particular form: it is a \emph{multiple binomial sum} in the sense of~\cite{BoLaSa17} (see Footnote~\ref{footnote:BLS}). As such, $S(T)(t) = \sum_{n \geq 0} u_n t^n$ is a generating function of a very special type: it is the diagonal of a multivariate rational function~\cite[\S3]{BoLaSa17}. It can be written as a multiple definite integral, and also as an algebraic residue, of some explicit multivariate rational functions. Due to these general theories, we know beforehand that the sequence $(u_n(T))_{n \geq 0}$
is P-recursive and that the power series $S(T)(t)$ is D-finite. Moreover, there are (several generations of) \emph{creative telescoping algorithms} for computing recurrence relations satisfied by $(u_n)_{n \geq 0}$ and linear differential equations for $S(T)(t)$, starting either from the expression~\eqref{eq:def_ST}, or from a diagonal or residue expression of it~\cite{Chyzak00,Koutschan10,BoLaSa13,Lairez16}. Creative telescoping algorithms do not always deliver the minimal-order differential equation satisfied by the integral/residue/diagonal; however, there exist efficient algorithms for performing this \emph{minimization} step~\cite{BoRiSa24}. 

Now, due to our structural result~(\cref{thm:sum-pi}), we moreover know that the minimal-order linear differential
equation satisfied by $S(T)$ decomposes as the least common left multiple (LCLM) of smaller-order
equations, each of them being homomorphic with symmetric powers of Gauss hypergeometric equations; the
corresponding computations (LCLM factorizations, symmetric powers, homomorphisms) can be performed
using algorithms due to van Hoeij \cite{Hoeij96} (and implemented for instance in the \texttt{DEtools} package in \texttt{Maple}).

All in all, the closed form expression predicted by \cref{thm:sum-pi} can be computed from a basis of solutions of the minimal-order differential equation: $S(T)$ is written as a polynomial in 
$H_1(t) = {}_{2}^{}{{}{{}{F_{1}^{}}}}\big(-\tfrac12,-\tfrac12; 1; 16 t^2\big)$
and
$H_2(t) = {}_{2}^{}{{}{{}{F_{1}^{}}}}\big(-\tfrac12,\tfrac12; 2; 16 t^2\big)$.
Finally $S(T)(1/4)$ is deduced as an explicit polynomial in $1/\pi$ using the classical identities~\eqref{eq:values_elliptic}, 
$H_1(1/4) = 4/\pi$ 
and
$H_2(1/4) = 8/(3\pi)$.
Below, we illustrate this algorithmic chain on a non-trivial example. A similar type of
hypergeometric expressions for generating functions appeared in a different combinatorial context
(for counting lattice walks in the quarter plane) in~\cite{BoChHoKaPe17}, where tools similar to
the ones described here were applied.

Note that our \cref{thm:sum-pi} was crucially used in the algorithmic chain sketched above. Without it, it would be by no means clear how to solve in explicit terms the minimal linear differential equation for $S(T)$. 
Finally, we insist on the fact that the algorithm sketched above only works for a fixed tree~$T$ (it does not prove~\cref{thm:sum-pi}), and that its (time and space) complexity depends on the size of the tree~$T$. It would be interesting to estimate its complexity and to compare it to the one of the algorithm described in this paper (and summarized in Appendix~\ref{appendix:algo}).

\begin{remark} \label{rem:algo_variant}
A variant of the algorithm described above, which is more efficient in practice, is based on our structural result (\cref{thm:sum-pi}) and finds the representation of $S(T)$ as an element in the algebra $\cA$ by using the degree bounds and Hermite-Padé approximation algorithms~\cite{BeLa94}, before specializing the result at $t=1/4$.

Another route, more symbolic-numeric, to find an expression of $S(T)(1/4)$ would be the following.
There exist algorithms able to evaluate efficiently a D-finite power series (here, $S(T)$) at an
algebraic point (here, $t=1/4$); the result is however not given by a formula, but rather by an
approximation of the value (here, $S(T)(1/4)$) whose computed digits are guaranteed to be correct
up to the required precision~\cite{MeSa10,Mezzarobba10}. This approach enables approximating the
value $S(T)(1/4)$ with high precision, typically hundreds or thousands of decimal digits. In many
cases, this is sufficient for algorithms performing constant recognition (based on integer
relation algorithms such as PSLQ~\cite{FeBaAr99,BaBr01} and LLL~\cite{LLL82,NoStVi11}) to identify
the value.
But here the user should be aware of two important aspects regarding this powerful symbolic-numeric
approach: first, even with our \cref{thm:sum-pi} at hand, constant recognition algorithms need
quite a high numerical precision as input; second, even when they are successful, their output
(here, the expression of $S(T)(1/4)$ as a polynomial in $1/\pi$) is only \emph{guessed}, and
\emph{not proved}. 
\end{remark}

\begin{example}
Consider the tree 
$T = $
\begin{small}
	\begin{tikzpicture}[scale=.66]
\draw[white] (-2,0) -- (2,0);
\draw[white] (0,2) -- (0,2.8);
\draw (0,0) -- (0,1) -- (0,2) -- (-1,1) (-1,1) -- (-1,0) (0,2) -- (1,1) (-1,1) -- (-2,0) (0,1) -- (1,0);
\draw[fill] (0,0) circle(.1);
\draw[fill] (0,1) circle(.1);
\draw[fill] (0,2) circle(.1);
\draw[fill] (1,1) circle(.1);
\draw[fill] (-1,1) circle(.1);
\draw[fill] (-1,0) circle(.1);
\draw[fill] (-2,0) circle(.1);
\draw[fill] (1,0) circle(.1);
\draw (-.8,1.5) node[red]{$a$};
\draw (-.2,1.5) node[red]{$b$};
\draw (.8,1.5) node[red]{$c$};
\draw (-.7,.5) node[red]{$e$};
\draw (.8,.5) node[red]{$g$};
\draw (.18,.5) node[red]{$f$};
\draw (-1.8,.5) node[red]{$d$};
\end{tikzpicture}
\end{small}
with its corresponding sum:
\[S(T) = \sum_{\Z_+^7} \Cat_{a+b+c} \Cat_{a+d+e} \Cat_{b+f+g}  \Cat_c \Cat_d \Cat_e \Cat_f \Cat_g t^{2(a+b+c+d+e+f+g)} .\]

The algorithm sketched above first computes that 
\[ S(T)(\sqrt{t}) = 1 + 7 \, t + 58 \, t^{2} + 542 \, t^{3} + 5508 \, t^{4}
	+ 59508 \, t^{5} +\cdots \]
satisfies a linear differential equation of order 10 (with coefficients in $\Q[t]$ of degree~33), 
then decomposes it as the LCLM of four linear differential equations $L_1, \ldots, L_4$, with $L_k$ of order~$k$. Solving each differential equation $L_k(y(t)) = 0$ as explained above, and finding the appropriate linear combination, it then deduces that
\begin{equation*}
	\begin{split}
S(T) = 
\frac{1}{1680 \, t^{8}}
\Big(
105 \, t^{2} \, H_1 \, H_2^{2} + 
210 \, t^{2} \, H_1 H_2  + 
105 \, t^{2} \, H_2^{2} +
3 \, \left(67 t^{2}+2\right) \, H_1 +
 \\ 
- 3 \, \left(232 t^{4}+114 t^{2}+2\right)\, H_2 
- (840 \, t^{4} + 315 t^{2})
\Big) \, ,
\end{split}
\end{equation*}
where, as before,
\[
H_1 = {}_{2}^{}{{}{{}{F_{1}^{}}}}\big(-\tfrac12,-\tfrac12; 1; 16t^2 \big)
\quad \text{and} \quad
H_2 = {}_{2}^{}{{}{{}{F_{1}^{}}}}\big(-\tfrac12,\tfrac12; 2; 16t^2 \big) \ .
\]
From there, and using the identities $H_1(1/4) = 4/\pi$ and $H_2(1/4) = 8/(3\pi)$, it concludes 
\[
S(T)(1/4) = \frac{2^{16}}{9} \cdot \left( \frac{1}{\pi^3} + \frac{1}{\pi^2} \right) - \frac{2^{13}}{35 \cdot \pi} - 7 \cdot 2^7.
\]
\end{example}

\subsection{Notation and terminology}
In the whole paper, we will always consider rooted trees, although our sums $S(T)$ are a priori defined for unrooted trees, and only depend on isomorphism classes of such unrooted trees.
The root of the tree will always be denoted by $\rho$. The height of a vertex $v \in T$ is defined as $h(v) \coloneqq  d_T(\rho,v)$, where $d_T$ denotes the usual graph distance on $T$ (all edges have length $1$). For $v,w \in T$, denote by $\llbracket v,w \rrbracket$ the unique geodesic in $T$ between the vertices $v$ and $w$. We say that $v$ is an ancestor of~$w$ if $v \in \llbracket \rho, w \rrbracket$. If this holds, we also say that $w$ is a descendant of $v$. We define the \textit{fringe subtree} rooted at a vertex $v$ as the tree formed by $v$ and all its descendants, keeping all edges existing in $T$ between these vertices.

\subsection{Outline of the paper}

We start in \cref{sec:infinite_noodle} by explaining in more detail how the sums $S(T)$ appear in the context of the infinite noodle. We then consider a few trees of small size in \cref{sec:first_examples}: we compute the associated sums and introduce some intuition about how to handle the general case. In Section~\ref{sec:change_variables}, we introduce a more general family of series, associated to \textit{decorated trees}, see \cref{ssec:generalization}. These decorated trees are actually the ones on which our algorithm will be performed. 

The proof of our main result is based on an induction on the size of the tree; \cref{sec:base_case} is devoted to the base case, namely trees with two vertices, while \cref{sec:relations} gathers relations between sums associated to different trees, which are needed for the induction step. We show in \cref{sec:more_examples}, on some examples, how these relations are used in practice. 

The last result needed for our induction to work is the study of special decorated trees called \textit{long stars}, which is performed in \cref{sec:long_stars}. This allows us to prove our main result in \cref{sec:proof_main_theorem}. The last section, \cref{sec:stars}, is devoted to the proof of~\eqref{eq:stars}. 

Finally, we provide four appendices: \cref{appendix:alg_independence}, where we prove the algebraic independence of $H_1(t)$ and $H_2(t)$; \cref{appendix:algo}, which contains a pseudo-code language representation of our algorithm; \cref{appendix:running_algo}, which runs our algorithm on a non-trivial example; and \cref{appendix:explicit_sums}, which provides an extensive list of all sums associated to trees with at most $7$ vertices (without half-edge).

\section{Shape probabilities in the infinite noodle}
\label{sec:infinite_noodle}

\subsection{The formula from \cite{feray2023meandric}}

By definition, a meander $M$ of size $k$ is a simple curve in the plane intersecting the real line
exactly at points $0, \ldots, 2k-1$, with simple intersection points.
Meanders are always considered up to continuous deformation fixing the horizontal axis.
Combinatorially, a meander can be represented 
as a pair of non-crossing pair-partitions, being drawn respectively above and below the real line
and chosen such that the resulting graph consists of a single loop, see \cref{fig:meander_faces}.

Given a meander $M$, we consider its union with the real line;
this divides the plane in regions, two unbounded ones and several bounded ones.
We call the bounded regions {\em faces} of the meander, 
and denote their set by $\mathcal F(M)$. 
Some faces are {\em interior faces}, in the sense that they are contained in the interior
of the simple curve defining the meander, the others are called {\em exterior faces}.
A face $F$ in $\mathcal F(M)$ is incident to a number of segments $[i;i+1]$;
we let $I(F)$ be the set of indices $i$ such that $F$ is incident to $[i;i+1]$.
It is easy to see that
 if $F$ is an interior face (resp.~exterior face) of the meander, then $I(F)$
consists only of even (resp.~odd) integers.
A meander with ten faces and some of the corresponding index sets
are shown on \cref{fig:meander_faces}.
\begin{figure}
  \[
  \includegraphics[scale=1.2]{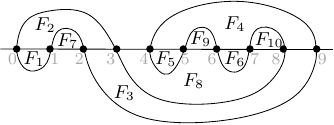}
  \]
  \caption{A meander of size $k=5$ and its faces.
  Faces $F_1$ to $F_6$ are interior faces, while faces $F_7$ to $F_{10}$ are exterior faces.
In this example, we have
$I(F_3)=\{2,8\}$, $I(F_4)=\{4,6,8\}$, $I(F_8)=\{3,5,7\}$.}
  \label{fig:meander_faces}
\end{figure}

With this notation, we have the following statement (see \cite[Proposition 10]{feray2023meandric}), where $S_0$ denotes the shape of the component of $0$ in the infinite noodle.
\begin{proposition}
\label{lem:lowercomputation}
For any meander $M$ of size $k$, we have
\begin{equation}\label{eq:proba_meandre}
  \P (S_0=M) = 2^{-4k+1} k\, \sum_{\ell_0,\dots,\ell_{2k-2} \ge 0}  \left( \prod_{F \!\in\ \mathcal F(M)}
 \Cat_{\ell_{I(F)}} 4^{-\ell_{I(F)}} \right),
 \end{equation}
 where, for a set $I$ of indices, we use the notation $\ell_I=\sum_{i\in I} \ell_i$.
\end{proposition}

\subsection{Meanders and trees}
\label{ssec:meanders_and_trees}
Let $M$ be a meander. We will associate a forest $G_M$ with $M$.
To this end, put a vertex $v_F$ in each face $F$ of $M$. Then we connect vertices corresponding to faces sharing
a segment $[i;i+1]$ as boundary, i.e.~$v_F$ is connected to $v_{F'}$ if $F$ and $F'$ are connected.
Additionally, we add a half-edge to vertices sharing a segment $[i;i+1]$ with an unbounded region.
We call $G_M$ the resulting graph, $V(G_M)$ and $E(G_M)$ being respectively its vertex- and edge-sets.
Since a meander is made of a single curve, $G_M$ does not have cycles,
i.e.~is a forest.
More precisely, the vertices associated with interior faces form a single tree without half-edge,
while vertices associated with exterior faces might form one or several trees, each with exactly one half-edge.
An example is shown on \cref{fig:meander-and-tree}.
\begin{figure}
  \begin{center}
    \includegraphics[scale=1.2]{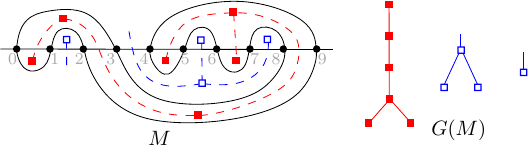}
  \end{center}
  \caption{The forest associated with the meander of \cref{fig:meander_faces}.
  On the left, we show the forest with dashed edges superimposed on the meander,
  on the right, the same forest is drawn as an abstract graph.
  Vertices corresponding to interior faces are red full squares, while those corresponding
  to exterior faces are blue empty squares.}
  \label{fig:meander-and-tree}
\end{figure}

The sum in \eqref{eq:proba_meandre} is naturally written in terms of the forest $G_M$.
Indeed edges of the forest are in one-to-one correspondence with the segments $[i;i+1]$ (for $0\le i \le 2k-2$),
so that we have one variable $x_e$ for each edge in the tree.
Also, vertices of the forest correspond to faces of the meander, so we have one (renormalized) Catalan factor for
each vertex of the forest.
Each of these Catalan factors is indexed by the sum of variables corresponding to edges incident to that vertex.
In formula,
\begin{equation}
  \P (S_0=M) = 2^{-4k+1} k\, \sum_{(x_e) \in \mathbb R_+^{E(G_M)}}  \left( \prod_{v \!\in\ V(G_M)}
  \Cat_{X_v} 4^{- X_v} \right),
 \end{equation}
 where, for each $v$ in $V(G_M)$, we set $X_v=\sum_{e:\, v\in e} x_e$.
We note that the sum factorizes over connected components.
Namely, if $G_M= \biguplus_{i=1}^d T_i$, we have
\begin{equation}
  \P (S_0=M) = 2^{-4k+1} k\, \prod_{i=1}^d \left[ \sum_{(x_e) \in \mathbb R_+^{E(T_i)}}  \left( \prod_{v \!\in\ V(T_i)}
  \Cat_{X_v} 4^{- X_v} \right) \right] = 2^{-4k+1} k\, \prod_{i=1}^d S(T_i)\big(\tfrac14\big).
\end{equation}
Thus the computation of $\P (S_0=M)$ can be reduced to that of $S(T)\big(\tfrac14\big)$,
as claimed in the introduction.

\section{Some first examples}
\label{sec:first_examples}
In this section, we explain how to compute the sum $S(T)$ in \eqref{eq:def_ST} for small trees $T$.
This serves as an introduction of some basic manipulations used in the proof of \cref{thm:sum-pi}.

\subsection{The path with two vertices}
\label{ssec:ex_P2}
As in the introduction, we denote by $P_2$ the tree with two vertices connected by an edge.
The associated sum is 
\begin{align*}
S(P_2) = \sum_{x \in \Z_+} \left( \Cat_x t^{x}\right)^2.
\end{align*}
Setting $u_{n}=4 t^{2n} \Cat_{n-1}^2$ for $n \ge 1$ and $u_0=1$, it holds that, for all $n\ge 0$,
\[16 t^2 (n-\tfrac12)^2\, u_n = (n+1)^2 u_{n+1}. \]
This recursion implies that 
\[ \sum_{n \ge 0} u_n= \sum_{n \ge 0} (16t^2)^n \frac{\big((-\tfrac12)^{\frac{n}{\,} }\big)^2}{(n!)^2} = {}_{2}^{}{{}{{}{F_{1}^{}}}}\big(-\tfrac12,-\tfrac12; 1; 16t^2\big),\]
where we recall that  ${}_{2}^{}{{}{{}{F_{1}^{}}}}$
is the hypergeometric function defined in \eqref{eq:def_hyper}.
We thus get that
\begin{align*}
  S(P_2) = \frac{1}{4t^2}\left( \sum_{n \ge 1} u_n \right) =\frac{1}{4t^2}\left( -1+ {}_{2}^{}{{}{{}{F_{1}^{}}}}\big(-\tfrac12,-\tfrac12; 1; 16t^2\big)\right).
\end{align*}

The evaluation at $t=1/4$ is given by the Gauss summation identity \eqref{eq:values_elliptic}, which yields
\[{}_{2}^{}{{}{{}{F_{1}^{}}}}\big(-\tfrac12,-\tfrac12; 1; 1\big)= \frac{\Gamma(1)\, \Gamma(2)}{\Gamma(\tfrac32)^2} = \frac4{\pi},\]
and, therefore,
\[S(P_2)(1/4) = \frac{1}{4 \cdot (1/4)^2} \left(  -1 + {}_{2}^{}{{}{{}{F_{1}^{}}}}\big(-\tfrac12,-\tfrac12; 1; 1\big) \right) = 4 \left(\frac{4}{\pi}-1\right) = \frac{16}{\pi} - 4.\]
This computation, or to be more precise, a generalization of it which can be treated
with the same method (see~\cref{ssec:base_case}),
will serve as base case in the proof of our main theorem.

\subsection{The path with two vertices, and an extra half-edge}
\label{ssec:ex_P2h}
We now consider a tree with two vertices connected by an edge, and an extra half-edge on one of them.
As in the introduction, we denote this tree by $P^h_2$. The associated sum is 
\begin{align*}
S(P^h_2) = \sum_{x_1,x_2 \in \Z_+} \Cat_{x_1} t^{x_1} \Cat_{x_1+x_2} t^{x_1+x_2}.
\end{align*}
We perform a change of variable and set $m=x_1$ and $\ell=x_1+x_2$.
Note that $m$ and $\ell$ are nonnegative integers,
and that we necessarily have $\ell \ge m$.
We get 
\[S(P^h_2)= \sum_{\ell, m \ge 0} \Cat_{\ell} \Cat_m t^{\ell+m} \One[\ell \ge m].\]
The idea is now to use that 
\begin{equation}\label{eq:reversing_inequality_example}
  \One[\ell \ge m] + \One[m \ge \ell] = 1 + \One[m=\ell].
\end{equation}
We multiply by $\Cat_{\ell} \Cat_m t^{\ell+m}$ and sum over $\ell, m \ge 0$.
By symmetry, both $\One[\ell \ge m]$ and $\One[m \ge \ell]$ yield $S(P^h_2)$. We get
\begin{align*}
2 S(P^h_2) = \left(\sum_{\ell} \Cat_{\ell} t^{\ell}\right)^2 + \sum_{\ell \ge 0} \Cat_\ell^2 t^{2\ell}
= \frac{1}{4 t^2}\left( 1-\sqrt{1-4t} \right)^2 + \frac{1}{4t^2} \left( -1+  {}_{2}^{}{{}{{}{F_{1}^{}}}}\big(-\tfrac12,-\tfrac12; 1; 16t^2\big)\right).
\end{align*}
Specializing to $t=1/4$, we get
\begin{align*}
2 S(P^h_2)(1/4) = 4 + \big( \tfrac{16}{\pi} - 4 \big)= \tfrac{16}{\pi}.
\end{align*}
We conclude that $S(P^h_2)(1/4)=\frac{8}{\pi}$.

Some ideas from this computation will be useful in the general case.
In particular,
we will always start by a change of variables to get a sum of products of Catalan numbers
indexed by single summation variables subject to some inequalities.
We will also make use of relations like \eqref{eq:reversing_inequality_example} to \enquote{manipulate} inequalities in a convenient way.

\subsection{The path with three vertices}
\label{ssec:ex_P3}
Let us now consider the path with three vertices (without half-edge),
which we denote by $P_3$.
The associated sum is 
\begin{align*}
  S(P_3) = \sum_{x_1,x_2 \in \Z_+} \Cat_{x_1} t^{x_1} \Cat_{x_1+x_2} t^{x_1+x_2} \Cat_{x_2} t^{x_2}.
\end{align*}
We split the sum according to the value of $X=x_1+x_2$:
\[ S(P_3) =\sum_{X \ge 0} \Cat_X t^{2X} \left( \sum_{x_1,x_2 \ge 0 \atop x_1+x_2=X} \Cat_{x_1} \Cat_{x_2} \right).\]
Using the standard quadratic recurrence relation for Catalan numbers, we recognize the inner sum to be equal to $\Cat_{X+1}$, so that
\[ S(P_3) =\sum_{X \ge 0} \Cat_X \Cat_{X+1} t^{2X}.\]
This is a variant of the sum $S(P_2)$, which can be handled in a similar way.
Defining $v_n$ by $v_n=-2\Cat_{n-1} \Cat_n t^{2n}$ for $n \ge 1$ and $v_0=1$, we have the recursion 
\[16 t^2\big(n-\tfrac12\big)\, \big(n+\tfrac12\big) \, v_n = (n+1)(n+2) \, v_{n+1}, \text{ for $n \ge 0$.}\]
Using the definition of hypergeometric functions, we obtain

\[\sum_{ n \ge 0} v_n = {}_{2}^{}{{}{{}{F_{1}^{}}}}\big(-\tfrac12,\tfrac12; 2;16t^2\big), \]
so that
\begin{align*}
S(P_3) = \frac{1}{2t^2}\left(1-{}_{2}^{}{{}{{}{F_{1}^{}}}}\big(-\tfrac12,\tfrac12; 2;16t^2\big)\right).
\end{align*}
Again, the value at $t=1/4$ can be determined through the Gauss summation identity~\eqref{eq:values_elliptic}:
\[ {}_{2}^{}{{}{{}{F_{1}^{}}}}\big(-\tfrac12,\tfrac12; 2;1\big) = \frac{8}{3\pi},\]
and, consequently,
\[  S(P_3)(1/4)= \frac{1}{2 \cdot (1/4)^2}\left(1-\frac{8}{3\pi} \right) = 8- \frac{64}{3\pi}.\]

This example illustrates the use of the quadratic Catalan recurrence to reduce the number of summation indices.
This technique will also be used in the general case.
This however induces a shift in the summation indices, which prompts us
to consider more general sums than $S(T)$ to be able to prove \cref{thm:sum-pi} 
by induction. This generalization is introduced in \cref{ssec:generalization} below.

\subsection{The path with four vertices}
\label{ssec:ex_P4}
We consider a last example, namely the path $P_4$ with four vertices,
 showcasing that, in general, we need to combine the changes of variables and inequality manipulations from \cref{ssec:ex_P2h}
with the reduction of number of variables via the quadratic Catalan recurrence from \cref{ssec:ex_P3}.
By definition, we have
\[S(P_4)=  \sum_{x_1,x_2,x_3 \in \Z_+} \Cat_{x_1} t^{x_1} \Cat_{x_1+x_2} t^{x_1+x_2} \Cat_{x_2+x_3} t^{x_2+x_3}
\Cat_{x_3} t^{x_3}.\]
We set $\ell_1=x_1$, $m_1=x_1+x_2$, $\ell_2=x_2+x_3$ and $m_2=x_3$. This defines a bijection from $\Z_+^3$ to the set of quadruples $(\ell_1,m_1,\ell_2,m_2)$ satisfying the relations
\[\ell_1+\ell_2=m_1+m_2 \text{ and }\ell_1 \le m_1.\]
Using $\One[\ell_1 \ge m_1] + \One[m_1 \ge \ell_1] = 1 + \One[m_1=\ell_1]$ and the symmetry between $\ell$'s and $m$'s,
we get
\begin{multline*}
2 S(P_4) = \sum_{\ell_1,m_1,\ell_2,m_2 \atop
 \ell_1+\ell_2=m_1+m_2} \Cat_{\ell_1} \Cat_{m_1} \Cat_{\ell_2} \Cat_{m_2} t^{\ell_1+m_1+\ell_2+m_2} \\ + \sum_{\ell_1,m_1,\ell_2,m_2 \atop
 \ell_1=m_1, \ell_2=m_2} \Cat_{\ell_1} \Cat_{m_1} \Cat_{\ell_2} \Cat_{m_2} t^{\ell_1+m_1+\ell_2+m_2}. 
\end{multline*}
The second sum can be factored out and computed as follows:
\[\sum_{\ell_1,m_1,\ell_2,m_2 \atop
 \ell_1=m_1, \ell_2=m_2} \Cat_{\ell_1} \Cat_{m_1} \Cat_{\ell_2} \Cat_{m_2} t^{\ell_1+m_1+\ell_2+m_2} = \left( \sum_{\ell} \Cat_{\ell}^2 t^{2\ell} \right)^2 = \left( S(P_2) \right)^2.\]
To compute the first sum, we split the sum according to the value $L$ of 
 $\ell_1+\ell_2$ and use the quadratic recurrence:
 \begin{multline*}
\sum_{\ell_1,m_1,\ell_2,m_2 \atop
 \ell_1+\ell_2=m_1+m_2} \Cat_{\ell_1} \Cat_{m_1} \Cat_{\ell_2} \Cat_{m_2} t^{\ell_1+m_1+\ell_2+m_2}=
 \sum_{L \ge 0} t^{2L} \Cat_{L+1}^2\\
 = t^{-2} \sum_{L' \ge 1} t^{2L'} \Cat_{L'}^2
 = t^{-2} \cdot S(P_2)  - t^{-2}.
 \end{multline*}
Altogether, we get
\begin{align*}
S(P_4) &= \frac12  \left( S(P_2) \right)^2
+ \frac{1}{2t^2} \left( S(P_2) \right) - \frac{1}{2t^2}\\
&=\frac{1}{32t^4}\left(-3 - 16 t^2 + 
  2{}_{2}^{}{{}{{}{F_{1}^{}}}}\big(-\tfrac12, -\tfrac12 ; 1 ; 16 t^2\big) + 
  {}_{2}^{}{{}{{}{F_{1}^{}}}}\big(-\tfrac12, -\tfrac12 ; 1 ; 16 t^2\big)^2\right) .
\end{align*}
Specializing to $t=1/4$, we get
\[S(P_4)(1/4)  = \frac12  \left( \frac{16}{\pi}-4 \right)^2
+ 8 \left( \frac{16}{\pi}-4 \right) - 8=-32 + \frac{64}{\pi} + \frac{128}{\pi^2} .\]

\section{Changing variables and generalizing}
\label{sec:change_variables}
\subsection{Vertex variables}
We will rewrite the sum $S(T)$ using one variable per vertex.
For this, it is convenient to work with rooted trees, i.e.~trees with a distinguished vertex~$\rho$,
called the root.
If $T$ has a half-edge, we will always consider the extremity of the half-edge as the root
of the tree.
Otherwise, we choose arbitrarily a root in $T$.
Then, we color vertices at even (resp. odd) distance from the root
in white (resp.~black). Let $V_{\circ}(T)$ and $V_{\bullet}(T)$ be the sets
of white and black vertices of $T$ respectively.
Also, we use the notation $v' \le_T v$ (or sometimes $v \geq_T v'$) to say that $v'$ is a descendant of $v$, 
i.e.~that $v$ lies on the path from $\rho$ to $v'$. 
Finally, we recall that if $(x_e)_{e \in E(T)}$ are edge-weights,
we let $X_v=\sum_{e \ni v} x_e$ be the sum of the weights of the edges incident to $v$.
\begin{lemma}\label{lem:changing_variables}
If $T$ is a tree without half-edge,
the map from $\mathbb Z_+^{E(T)}$ to $\mathbb Z_+^{V(T)}$
which maps $(x_e)_{e \in E(T)}$ to $(X_v)_{v \in V(T)}$ is injective and its range
is the set of $(X_v)_{v \in V(T)}$ in $\mathbb Z_+^{V(T)}$ satisfying the following conditions
\begin{equation}
\label{eq:changeofvariable}
 \begin{cases}
\quad\bullet\quad \sum_{w  \in V_{\circ}(T)} X_w = \sum_{b \in V_{\bullet}(T)} X_b, \\
\quad \bullet\quad \text{for all $v \in V_{\circ}(T)$, }  \sum_{w  \in V_{\circ}(T), w \le_T v} X_{w} \ge 
\sum_{b \in V_{\bullet}(T), b \le_T v} X_b, \\
\quad \bullet\quad \text{for all $v \in V_{\bullet}(T)$, } 
\sum_{w \in V_{\circ}(T), w \le_T v} X_w \le 
\sum_{b \in V_{\bullet}(T),  b \le_T v} X_{b}.
\end{cases}
\end{equation}
The same holds for a tree with a half-edge replacing the first condition 
with $$\textstyle \sum_{w  \in V_{\circ}(T)} X_w  \ge \sum_{b \in V_{\bullet}(T)} X_b.$$ 
\end{lemma}

\begin{proof}
We consider first the case of a tree $T$ without half-edge. 
Let us consider a vector $(x_e, e \in E(T))$ mapped to $(X_v, v \in V(T))$,
i.e., for any $v$ we have  $X_v=\sum_{e \ni v} x_e$.
Since each edge has exactly one white and one black endpoint, the first condition of \eqref{eq:changeofvariable} is fulfilled. Furthermore, we have that 
\begin{align}
\label{eq:changeofvariable2}
\quad\bullet\quad \text{for all } v \in V_{\circ}(T), \sum_{w  \in V_{\circ}(T), w \le_T v} X_{w} = 
\sum_{b \in V_{\bullet}(T), b \le_T v} X_b + x_{e(v)}, \nonumber\\
\quad\bullet\quad \text{for all } v \in V_{\bullet}(T),
\sum_{w \in V_{\circ}(T), w \le_T v} X_w =
\sum_{b \in V_{\bullet}(T),  b \le_T v} X_{b} - x_{e(v)},
\end{align}
where $e(v)$ is the edge between $v$ and its parent in $T$ (and $x_{e(\rho)}=0$ for convenience).
Since all $x_{e}$'s are nonnegative, this proves the inequalities in \eqref{eq:changeofvariable}.

Conversely, being given a family of vertex-weights $(X_v, v \in V(T))$ satisfying \eqref{eq:changeofvariable}, it is clear that \eqref{eq:changeofvariable2} characterize the corresponding family of edge-weights $(x_e, e \in E(T))$.
Hence the map $(x_e, e \in E(T)) \mapsto (X_v, v \in V(T))$ is indeed injective.
 Moreover, the signs in~\eqref{eq:changeofvariable2}, together with the inequalities \eqref{eq:changeofvariable},
 ensure that all edge-weights defined this way are indeed nonnegative.
 This proves the lemma for trees without half-edges.
 
We now consider the case where $T$ is a tree with a half-edge $\overline{e}$. Then, since each other edge connects a black vertex to a white vertex, we get that
\begin{align*}
x_{\overline{e}}=\sum_{w  \in V_{\circ}(T)} X_w - \sum_{b \in V_{\bullet}(T)} X_b.
\end{align*}
Using this identity together with the same strategy as in the case without half-edge proves the result.
\end{proof}

Since variables indexed by white and black vertices play different roles,
from now on, we denote the former by $(\ell_w)_{w \in V_{\circ}(T)}$ and the latter by
$(m_b)_{b \in V_{\bullet}(T)}$.
Then using Lemma~\ref{lem:changing_variables}, $S(T)$ can be rewritten as
\begin{equation}
\label{ST_with_vertex_variables}
S ( T ) = \sum_{(\ell_w), (m_b) \geq 0}\left( \prod_{w  \in V_{\circ}(T)} \Cat_{\ell_w} t^{\ell_w} \right) 
\left( \prod_{b \in V_{\bullet}(T)} \Cat_{m_b} t^{m_b} \right) \left( \prod_{v \in V(T)} \One[C_v] \right),
\end{equation}
where $C_v$ is the following condition
\[\begin{cases}
\quad\bullet\quad  \sum_{w  \in V_{\circ}(T)} \ell_w = \sum_{b \in V_{\bullet}(T)} m_b&  \text{ if $v =\rho$,} \\
\quad\bullet\quad  \sum_{w  \in V_{\circ}(T), w \le_T v} \ell_w \ge 
\sum_{b \in V_{\bullet}(T), b \le_T v} m_b
 & \text{ if $v \in V_{\circ}(T)$,} \\
\quad\bullet\quad  \sum_{w  \in V_{\circ}(T), w \le_T v} \ell_w \le 
\sum_{b \in V_{\bullet}(T), b \le_T v} m_b
 & \text{ if $v \in V_{\bullet}(T)$.}
\end{cases}\]
Again, if the tree $T$ has a half-edge with extremity $\rho$,
then the equality in the case $v=\rho$ should be replaced by an inequality.

\subsection{A generalization}
\label{ssec:generalization}
We now consider a general framework,
which (strictly) contains sums associated with trees with exactly one or without half-edge,
and in addition, will be sufficiently flexible to allow one to compute the sums recursively.

Let us consider a rooted tree $T$,
where each vertex is decorated with a triple $(\eps_v,\bowtie_v,K_v)$
living in the product space $\{-1,0,1\} \times \{=,\le,\ge,\varnothing\} \times \mathbb Z$.
Vertices $v$ with $\eps_v=-1$ (resp.~$\eps_v=0$, $\eps_v=1$)
will be referred to as black (resp.~gray, white).
We do not impose that the coloring is proper, 
i.e.~black vertices can be connected to black
vertices and similarly for white and gray vertices.

To define $S(T)$ in this general setting, we associate with each black vertex $b$ (resp.~white vertex $w$) a variable $m_b$ (resp.~$\ell_w$). No variable is associated with a gray vertex. Furthermore, with each vertex, we associate a condition (inequality or equality)
\begin{equation}\label{eq:Cv}
  \tag{$C^T_v$} \left(\sum_{w \le_T v \atop \eps_w=1} \ell_w -  \sum_{b \le_T v \atop \eps_b=-1} m_b \right) \bowtie_v \left( \sum_{v' \le_T v} K_{v'} \right),
\end{equation}
where the condition is void (and thus always satisfied) if $\bowtie_v=\varnothing$.
When there is no ambiguity regarding the tree under consideration, 
we shall denote this condition simply by $(C_v)$.

With this in hand, we define
\begin{equation}
\label{eq:def_ST_2}
S ( T ) = \sum_{(\ell_w), (m_b)}\left( \prod_{w: \, \eps_w=1} \Cat_{\ell_w} t^{\ell_w} \right) 
\left( \prod_{b: \, \eps_b=-1} \Cat_{m_b} t^{m_b} \right) \left( \prod_v \One[C_v] \right),
\end{equation}
 where the sum is taken over tuples $(\ell_w)$ and $(m_b)$ of nonnegative integers indexed by white and black vertices of $T$, respectively.
 Also, in the inner products, $w$ runs over white vertices (indicated by the condition $\eps_w=1$), 
 $b$ over black ones (indicated by the condition $\eps_b=-1$), and $v$ over all vertices.
In later discussions, the condition $\eps_w=1$ (resp.~$\eps_b=-1$) will be implicit when using the variable $w$, resp.~$b$.

We make here a small abuse of notation compared to the definition of \eqref{eq:def_ST}. Indeed, letting $T$ be a non-decorated tree with possibly a half-edge.
\revision{We associate to $T$ a rooted decorated version $T'$ of $T$ as follows:}
\begin{itemize}
\item \revision{If $T$ contains a half-edge with extremity $\rho$, then $\rho$ is chosen as the root of $T'$. Otherwise, choose an arbitrary vertex $\rho$ as the root of $T'$; }
\item \revision{The root of $T'$ is white, decorated with $(=, 0)$ if $T$ has no half-edge, 
or with $(\ge,0)$ in case of a tree $T$ with a half-edge;}
\item \revision{$T'$ has only black and white vertices and they alternate;}
\item \revision{all the black vertices are decorated with $(\le, 0)$ and, 
finally, all the white vertices are decorated with $(\ge, 0)$.}
\end{itemize}
With this construction, the sum $S(T)$ in the sense of  \eqref{eq:def_ST} corresponds to  $S(T')$  in the sense of \eqref{eq:def_ST_2}; see \cref{ST_with_vertex_variables}.
 We will sometimes refer to $T'$ as the canonical coloration and decoration of $T$.

We will prove the following theorem, 
containing \cref{thm:sum-pi} as a special case.
 \begin{theorem}\label{thm:sum-pi-generalized}
   For any decorated tree $T$ as above, one has $S ( T ) \in \cB$. 
   Furthermore, if the root of~$T$ is decorated with an equality symbol,
   then $S ( T ) \in \cA$.

   In both cases, denoting by $\tilde{V}_T$ the number of vertices that are not gray (i.e. either black or white vertices) in $T$, we have $d_{S(T)} \le \tilde{V}_T$.
   As a consequence, for any decorated tree $T$, the evaluation 
   $S ( T )(\frac{1}{4})$ is a polynomial in $1/\pi$
   with rational coefficients and degree at most $\lfloor \frac{\tilde{V}_T}{2} \rfloor$.  
 \end{theorem}

\section{The base case: trees with two vertices}
\label{sec:base_case}

\subsection{Preliminary: contiguous relations between hypergeometric functions}
\label{ssec:relations}

We prove here a result that will be useful in \cref{ssec:base_case}.
For any $K \geq 0$, consider the hypergeometric function
\begin{align*}
H^{(K)}: z \mapsto {}_{2}^{}{{}{{}{F_{1}^{}}}}\big(-\tfrac12, K-\tfrac12 ; K+1 ; z\big).
\end{align*}

\begin{proposition}
\label{prop:contiguous}
For all $K \geq 0$, the series $H^{(K)}(16t^2)$ is an element of $\cA$ of degree $2$.
\end{proposition}

\begin{proof}
By definition, $t \mapsto H^{(0)}(16t^2)$ and $t \mapsto H^{(1)}(16t^2)$ are elements of $\cA$, and their degree is 2. Hence, it suffices to show that, for all $K \geq 2$, $H^{(K)}(16t^2)$ is a linear combination of $H^{(0)}(16t^2)$ and $H^{(1)}(16t^2)$. In fact, we will prove the following identity: for all $K \geq 0$ all $z$ such that $|z|<1$, we have
\begin{equation}
\label{eq:contiguous}
\frac{(K+1/2)(K+5/2)}{K+2} z \, H^{(K+2)}(z) = (K+1)(1+z)H^{(K+1)}(z)-(K+1)H^{(K)}(z).
\end{equation}
In order to prove \eqref{eq:contiguous}, we rely on the so-called contiguous relations between hypergeometric functions (see e.g. \cite[Section $2.5$]{AAR99}). We use in particular the following three relations, for all $a,b,c$:
\begin{itemize}
\item[(i)] $c(1-z)\,{}_{2}^{}{{}{{}{F_{1}^{}}}}\left(a,b;c;z\right)-c\,{}_{2}^{}{{}{{}{F_{1}^{}}}}\left(a,b-1;c;z\right)+(c-a)z\,{}_{2}^{}{{}{{}{F_{1}^{}}}}\left(a,b;c+1;z%
\right)=0,$
\item[(ii)] $(c-b)\,{}_{2}^{}{{}{{}{F_{1}^{}}}}\left(a,b-1;c;z\right)+\left(2b-c+(a-b)z\right)\,{}_{2}^{}{{}{{}{F_{1}^{}}}}\left(a,b;c;z\right)+b(z%
-1)\,{}_{2}^{}{{}{{}{F_{1}^{}}}}\left(a,b+1;c;z\right)=0,$
\item[(iii)] $(c-b-1)\,{}_{2}^{}{{}{{}{F_{1}^{}}}}\left(a,b;c;z\right)+b\,{}_{2}^{}{{}{{}{F_{1}^{}}}}\left(a,b+1;c;z\right)-(c-1)\,{}_{2}^{}{{}{{}{F_{1}^{}}}}\left(a,b;c-1;z%
\right)=0$.
\end{itemize}

Applying (i) to $(a',b',c')=(a,b+2,c+1)$, we obtain:
\begin{align*}
(c+1-a)z \,{}_{2}^{}{{}{{}{F_{1}^{}}}}(a,b+2;c+2;z) =  (c+1) \,{}_{2}^{}{{}{{}{F_{1}^{}}}}(a,b+1;c+1;z) + (c+1)(z-1)\, {}_{2}^{}{{}{{}{F_{1}^{}}}}(a,b+2;c+1;z).
\end{align*}
Applying now (ii) to $(a',b',c) = (a,b+1,c+1)$, we get
\begin{multline*}
(b+1)(z-1)\,{}_{2}^{}{{}{{}{F_{1}^{}}}}(a,b+2;c+1;z) = -(c-b)\,{}_{2}^{}{{}{{}{F_{1}^{}}}}(a,b;c+1;z)\\
- (2b+1-c+(a-b-1)z)\,{}_{2}^{}{{}{{}{F_{1}^{}}}}(a,b+1;c+1;z).
\end{multline*}

Finally, applying (iii) to $(a',b',c')=(a,b,c+1)$, we get:
\begin{align*}
(c-b) \,{}_{2}^{}{{}{{}{F_{1}^{}}}}(a,b;c+1;z) = b\,{}_{2}^{}{{}{{}{F_{1}^{}}}}(a,b+1;c+1;z)-c\,{}_{2}^{}{{}{{}{F_{1}^{}}}}(a,b;c;z).
\end{align*}

Using these three identities, we get
\begin{align*}
(c+1-a)z &\,{}_{2}^{}{{}{{}{F_{1}^{}}}}(a,b+2;c+2;z) =  (c+1)\, {}_{2}^{}{{}{{}{F_{1}^{}}}}(a,b+1;c+1;z) + (c+1)(z-1) \,{}_{2}^{}{{}{{}{F_{1}^{}}}}(a,b+2;c+1;z)\\
&= (c+1)\, {}_{2}^{}{{}{{}{F_{1}^{}}}}(a,b+1;c+1;z) \\ &\qquad - \frac{c+1}{b+1} \big[ (c-b)\,{}_{2}^{}{{}{{}{F_{1}^{}}}}(a,b;c+1;z) + (2b+1-c+(a-b-1)z)\,{}_{2}^{}{{}{{}{F_{1}^{}}}}(a,b+1;c+1;z) \big]\\
&=(c+1) \,{}_{2}^{}{{}{{}{F_{1}^{}}}}(a,b+1;c+1;z) - \frac{c+1}{b+1} (2b+1-c+(a-b-1)z)\,{}_{2}^{}{{}{{}{F_{1}^{}}}}(a,b+1;c+1;z) \\ & \qquad +\frac{c+1}{b+1} \big[ b\,{}_{2}^{}{{}{{}{F_{1}^{}}}}(a,b+1;c+1;z)-c\,{}_{2}^{}{{}{{}{F_{1}^{}}}}(a,b;c;z)\big].
\end{align*}
This implies:
\begin{align*}
\frac{b+1}{c+1}(c+1-a)z\,{}_{2}^{}{{}{{}{F_{1}^{}}}}(a,b+2;c+2;z) &= \left( b+1 - (2b+1-c + (a-b-1)z +b \right)\, {}_{2}^{}{{}{{}{F_{1}^{}}}}(a,b+1;c+1;z) \\ & \qquad -c\,{}_{2}^{}{{}{{}{F_{1}^{}}}}(a,b;c;z)\\
&= \left(c+(b+1-a)z\right)\,{}_{2}^{}{{}{{}{F_{1}^{}}}}(a,b+1;c+1;z) -c\,{}_{2}^{}{{}{{}{F_{1}^{}}}}(a,b;c;z).
\end{align*}

In particular, setting $a=-1/2,b=K-1/2,c=K+1$, we get \eqref{eq:contiguous}.
\end{proof}

\subsection{Decorated trees with two vertices}
\label{ssec:base_case}
We consider here simple decorated trees $T$ with two vertices, one black, one white, and where the nonroot
vertex is decorated with $(\varnothing,0)$ (we will see in the next section
that this is a harmless hypothesis). The sum $S(T)$ depends only on the decoration $(\bowtie,K)$
of the root vertex and is given by 
\[S(T)=S_{\bowtie,K} = \sum_{a, b \geq 0} \Cat_a \Cat_b t^{a+b} \One[a \bowtie b+K] \]
The following proposition
 will serve as the initialization case in the proof of \cref{thm:sum-pi}.
\begin{proposition}
\label{prop:basecase}
For any decoration $(\bowtie,K)$, the sum
$
S_{\bowtie,K} 
$ is in $\cB$. Moreover, if $\bowtie$ is the equality symbol, then 
$
S_{\bowtie,K} 
$ is in $\cA$. In all cases, $
S_{\bowtie,K} 
$ has degree at most 2.
\end{proposition}

\begin{proof}
If $\bowtie \, = \varnothing$, the condition $a \bowtie b+K$ is void (always satisfied) and
\[S_{\varnothing,K} = \sum_{a, b \geq 0} \Cat_a \Cat_b t^{a+b} = 
\left( \sum_{a \geq 0} \Cat_a t^a\right)^2=  \frac{1}{4 t^2}\left( 1-\sqrt{1-4t} \right)^2=\frac{1}{4 t^2}\left(2-4t -2\sqrt{1-4t}\right).\]
In particular, it lies in $\cB$ and has degree $1$.
We consider now the case where $\bowtie$ is the equality symbol.
Since $S_{=,K} = S_{=,-K}$, we focus on the case $K \ge 0$,
for which we have
 \[S_{=,K}=\sum_{X \ge 0} \Cat_X \Cat_{X+K} t^{2X+K}.\]
The following argument is a slight generalization of that of \cref{ssec:ex_P2,ssec:ex_P3}.
For $n \ge 1$, define $v_n$ by $v_n= \alpha_K \Cat_{n-1} \Cat_{n+K-1} t^{2n+K-2}$, where
$\alpha_K$ will be specified later.
Then, for $n \ge 1$, we have the recursion 
\[16 t^2\big(n-\tfrac12\big)\, \big(n+K-\tfrac12\big) \, v_n = (n+1)(n+K+1) \, v_{n+1}.\]
Setting $v_0=1$ and choosing $\alpha_K= \frac{-4\, (2K-1) t^{2-K}}{(K+1) \Cat_K}$, this recursion also holds for $n=0$.
Using the definition of hypergeometric functions, we obtain
\[\sum_{ n \ge 0} v_n = {}_{2}^{}{{}{{}{F_{1}^{}}}}\big(-\tfrac12,K-\tfrac12; K+1;16t^2\big), \]
so that
\[
S_{=,K}  = \frac{1}{\alpha_K} \sum_{n \ge 1} v_n
= \frac{1}{\alpha_K} \, \Big({}_{2}^{}{{}{{}{F_{1}^{}}}}\big(-\tfrac12,K-\tfrac12; K+1;16t^2\big) -1 \Big).
\]
In particular, by~\cref{prop:contiguous}, $S_{=,K}$ is in $\cA$ and has degree $2$. It remains to consider the sums $S_{\ge,K}$ and $S_{\le,K}$. Since one has $S_{\le,K}=S_{\ge,-K}$, we only consider the former. For $K=0$, the sum  $S_{\ge,0}$ coincides with the quantity $S(P_2^h)$ computed in Section~\ref{ssec:ex_P2h}. The formula obtained there shows that, indeed,
$S_{\ge,0}$ is in $\cB$ and has degree $2$.
Furthermore, for $K \geq 0$, we have
\begin{align*}
S_{\ge,K} = S_{\ge,0}-\sum_{r=0}^{K-1} S_{=,r}  \ \text{and}\ S_{\ge,-K} = S_{\ge,0}+\sum_{r=1}^{K} S_{=,-r}.
\end{align*}
Hence for all $K$ in $\Z$, the sum $S_{\ge,K}$ is in $\cB$ and has degree at most $2$, concluding the proof.
\end{proof}

\section{Relations between our sums}
\label{sec:relations}
\subsection{Pictorial representations}
We present some relations between the quantities $S( T )$.
These relations are helpful to prove \cref{thm:sum-pi} and are generalizations
of the manipulations made in computing our first examples (\cref{sec:first_examples}).
As much as possible, we represent these relations pictorially.
Here is a list of conventions for such pictures.
\begin{itemize}
\item As said above, the value of $\eps_v$ is encoded in the color of the vertex
(black for $\eps_v=-1$, gray for $\eps_v=0$, white for $\eps_v=1$).
Hence vertices in the pictures are only decorated by pairs in $\{=,\le,\ge,\varnothing\} \times \mathbb Z$.
We use vertices with several colors to indicate that the relation 
is valid regardless of the color of the vertex.
\item A Delta symbol $\Delta$ represents a generic pair in $\{=,\le,\ge,\varnothing\} \times \mathbb Z$. We write $\Delta +K$ to indicate that the second coordinate
of $\Delta$ should be incremented by $K$.
\item We write $\le_0$, $\ge_0$, $=_0$ and $\varnothing_0$ for $(\le,0)$, $(\ge,0)$,
$(=,0)$ and $(\varnothing,0)$.
  \item 
A polygon with a capital letter $T$, $U$, \ldots, represents a generic portion of a tree. 
When there is no ambiguity about which portion should be attached where in the right-hand side, we may also simply use dashed lines.
\item Finally, we extend $S$ by linearity and multiplicativity to linear combinations of forests
of rooted trees, and simply write $F_1 \equiv F_2$ for $S(F_1)=S(F_2)$.
\end{itemize}

\subsection{Relations using the vertex conditions}

\subsubsection*{Relations at the root.}
The first lemma allows to factorize sums, when the root $\rho$ satisfies $\bowtie_\rho=\varnothing$.

\begin{lemma}
\label{lem:root_varnothing}
If the root vertex satisfies $\bowtie_\rho=\varnothing$, then one of the following factorization formulas holds, depending on its color: 
 $$\begin{array}{c} \includegraphics{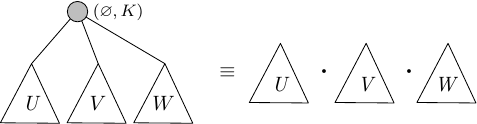} \end{array};$$
 $$\begin{array}{c} \includegraphics{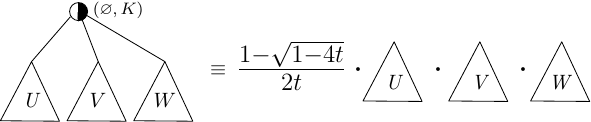} \end{array}.$$
\end{lemma}

\begin{proof}
  Call $\rho$ the root vertex of $T$. 
  If it is decorated with the symbol $\varnothing$, 
  then the variable associated with $\rho$, if any, does not appear in any of the conditions $C_v$'s for $v \in T$.
Furthermore, if $\rho$ is gray, there is no variable associated with the root,
 and the factorization of $S(T)$ is obvious. 
 If $\rho$ is white, then one can factorize $S(T)$ as 
\begin{align*}
S(T)=\sum_{\ell_\rho \geq 0} \Cat_{\ell_\rho} t^{\ell_\rho} \cdot \prod_{T'} S(T'),
\end{align*}
where the product is taken on all fringe subtrees rooted at a child of $\rho$.
The result follows using the identity $\sum_{\ell_\rho \geq 0} \Cat_{\ell_\rho} t^{\ell_\rho} = \frac{1-\sqrt{1-4t}}{2t}$. The same holds if $\rho$ is black.
\end{proof}

\subsubsection*{Relations at generic vertices.}
The next two lemmas deal with equality and empty conditions on nonroot vertices and
nonroot internal vertices, respectively.
We start with a simple factorization identity.
The quadrangular shape of $U$ below indicates that the parent of $v$
may have some descendants in $U$.
\begin{lemma}
\label{lem:equalitylabel}
Given a nonroot vertex $v$ of a tree $T$ such that $\bowtie_v$ is the equality symbol, one has:
  $$\begin{array}{c} \includegraphics{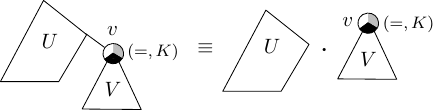} \end{array}.$$
\end{lemma}

\begin{proof}
  The product $S(U) \cdot S(V)$ in the right-hand side can be written
  as a large sum over variables associated with black/white vertices in $U$ and in $V$.
  Comparing with $S(T)$ (where $T$ is the tree on the left), the variables involved are the same on both sides.
  Hence, one only needs to check that the sets of conditions are equivalent.
  Notice if $v'$ is not a strict ancestor of $v$, then $C^T_{v'}= C^U_{v'}$ or $C^T_{v'}= C^V_{v'}$ (depending on where $v'$ lies).
  Now take $z$ a strict ancestor of $v$. The condition $C^T_z$ writes as
\begin{align*}
\left(\sum_{w \le_T z} \ell_w -  \sum_{b \le_T z} m_b \right) \bowtie_z \left( \sum_{v' \le_T z} K_{v'} \right)
\end{align*}
Note that $w \le_T z$ is realized either when $w$ is in $V$ (i.e.~$w \le_T v$) or when $w$ is in $U$ and $w \le_{U} z$.
Assuming that $C^T_v$ holds, that is,
\begin{align*}
\left(\sum_{w \le_T v} \ell_w -  \sum_{b \le_T v} m_b \right) = \left( \sum_{v' \le_T v} K_{v'} \right),
\end{align*}
the condition $C^T_z$ is equivalent to 
\begin{align*}
  \left(\sum_{w \le_{U} z} \ell_w -  \sum_{b \le_{U} z} m_b \right) \bowtie_z \left( \sum_{v' \le_{U} z} K_{v'} \right).
\end{align*}
The latter is nothing else than $C^U_z$.
This implies
\begin{align*}
\One[C^T_z] \cdot \One[C_v] = \One[C^U_z] \cdot \One[C_v].
\end{align*}
The result follows.
\end{proof}
 
\begin{lemma}
\label{lem:varnothinglabel}
Given a nonroot vertex $v$ such that $\bowtie_v$ is the empty symbol, one has:
  $$\begin{array}{c} \includegraphics{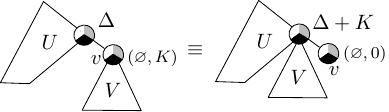} \end{array}.$$
When the parent of $v$ is gray, this simplifies further as follows:
  $$\begin{array}{c} \includegraphics{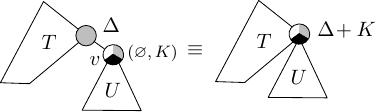} \end{array}.$$
\end{lemma}

\begin{proof}
  As above, we simply check that the conditions involved in the sums on both sides are equivalent.
In the first part, all vertices $v'$ except $v$ have the same set of descendants in both trees
and the same sum $\sum_{v'' \le_T v'} K_{v''}$,
so conditions $C_{v'}$, for $v' \ne v$, are the same on both sides.
Furthermore, $v$ is decorated with $\varnothing$, so the condition $C_v$ trivially holds on both sides.
The first part of the lemma follows.
The second part is proved similarly.
\end{proof}

The third identity allows to change the value of $K$ for vertices decorated with inequality symbols.
\begin{lemma}
\label{lem:shift_k}
  Given a nonroot vertex with decoration $(\ge,K)$ or $(\le,K)$, one has the following identity: 
  \begin{align*}&\begin{array}{c} \includegraphics{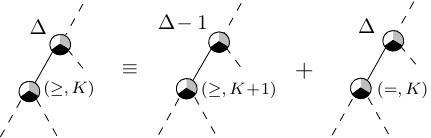} \end{array}\\
  \text{ and }&\begin{array}{c} \includegraphics{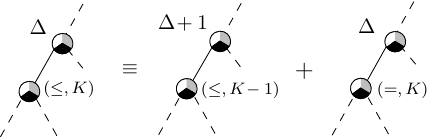} \end{array}.\end{align*}
  This can be also applied at the root vertex disregarding the modification of the decoration
  of the parent vertex.
\end{lemma}

\begin{proof}
We only prove the first one ; the second one can be proved the same way. Let $v$ be a nonroot vertex with decoration $(\geq,K)$. We use the fact that $\One[C_v]=\One[C'_v]+\One[C''_v]$, where:
\begin{align*}
 \tag{$C'_v$} \left(\sum_{w \le_{T} v} \ell_w -  \sum_{b \le_{T} v} m_b \right) \geq \left( \sum_{v' \le_{T} v, v' \neq v} K_{v'} + K + 1\right),
\end{align*} and

\begin{align*}
 \tag{$C''_v$} \left(\sum_{w \le_{T} v} \ell_w -  \sum_{b \le_{T} v} m_b \right) = \left( \sum_{v' \le_{T} v} K_{v'} \right).
\end{align*}
What is left to check is that the other conditions $C_{v'}$ for $v' \neq v$ are the same in the three trees. This is clear if $v'$ is not an ancestor of $v$. Now, take $z$ an ancestor of $v$. Since the set of descendants of $z$ in all three trees is the same, we only need to check that $\sum_{v' \leq z} K_{v'}$ is the same in all three trees. This is clear by construction. 
\end{proof}

The next formula will be used to reverse some inequalities.
\begin{lemma}\label{lem:reverting}
  $\begin{array}{c} \includegraphics{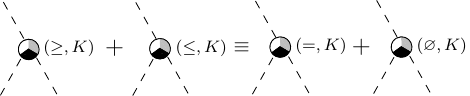} \end{array}$.   
\end{lemma}

\begin{proof}
Let $v$ be the vertex considered in the transformation. Observe that none of the conditions $C_{v'}$ for $v' \neq v$ is altered by this transformation. The result then follows from the fact that, for all $a,b \in \Z$, we have
\[
\One[a \geq b] + \One[a \leq b] = 1 + \One[a = b].\qedhere
\]
\end{proof}

\subsubsection*{Relation at leaves}
We now consider simplifications related to leaves. 
\begin{lemma}\label{lem:leaf}
The following relations hold:
 $$\begin{array}{c} \includegraphics{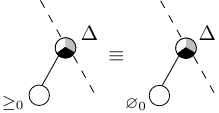} \end{array},\quad  \begin{array}{c} \includegraphics{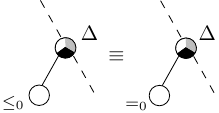} \end{array},\quad \begin{array}{c} \includegraphics{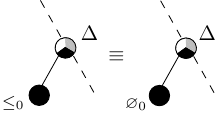} \end{array},$$
 $$\begin{array}{c} \includegraphics{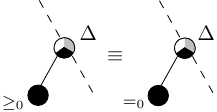} \end{array}, \quad \begin{array}{c} \includegraphics{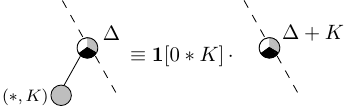} \end{array}.$$
 In the last equation, $\ast$ is an arbitrary element in $\{=,\le,\ge,\varnothing\}$
and $0 \ast K$ is the corresponding condition
 ($0=K$, $0 \le K$, $0 \ge K$ or $\mathrm{True}$).
  \end{lemma}

\begin{proof}
Let us prove the first one. Let $v$ be the leaf in question. Then, the condition $C_v$ can be written $\ell_v \geq 0$, which is a void condition. 
To prove the second one, observe that, this time, one can write $C_v$ as $\ell_v \leq 0$. Since the sum $S(T)$ is only on nonnegative values of $\ell_v$, this is equivalent to conditioning on $\ell_v=0$. The third and fourth equalities are proved the same way. 

 In order to prove the last one, observe that the condition $C_v$ is simply $0 \ast K$.
 Hence we only need to check that the conditions $C_{v'}$ for $v' \neq v$ are the same on both sides. 
 Take $v' \neq v$. Since $v$ is gray, no variable in $S(T)$ is indexed by $v$
 and the variable sum/difference in the left-hand side of $C_{v'}$ is the same on both sides.
 One can also check easily that the quantity $\sum_{v'' \le_T v'} K_{v''}$ is the same on both sides. This shows the last equality. 
\end{proof}

\subsection{Relations using the Catalan recurrence}
Our next set of relations takes advantage of the relation $\sum_{x+y=z} \Cat_x \Cat_y= \Cat_{z+1}$ satisfied by Catalan numbers.
To use this,
there needs to be two variables $\ell_{w_1}$ and $\ell_{w_2}$ (or $m_{b_1}$ and $m_{b_2}$),
which always appear together in the conditions $(C_v)$.

A first situation where this happens is when a node has two children of the same color (white or black) both decorated with a void symbol.
\begin{lemma}
\label{lem:twin_varnothing}
  $\begin{array}{c} \includegraphics{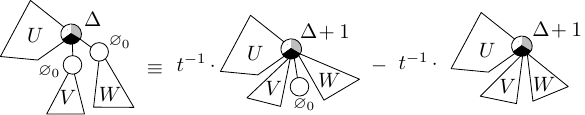} \end{array}$.   \\
    The same holds replacing white vertices with black vertices and the shift $+1$ with $-1$.\end{lemma}

\begin{proof}
Let $T$ be the tree on the left-hand side of Lemma~\ref{lem:twin_varnothing},
$w_1, w_2$ be the two white siblings with label $\varnothing$, 
and $p$ their common parent. 
Since $w_1$ and $w_2$ are decorated with $\varnothing$, the conditions $C_{w_1}$ and $C_{w_2}$ are void.
Furthermore, if $v$ is an ancestor of $p$, possibly $p$ itself, then the condition $C_{v}$ 
involves the sum $\ell_{w_1}+\ell_{w_2}$. Finally, if $v \ne w_1, w_2$ and $v$ is not an ancestor of $p$ (which we write $v\not\ge_T p$),
then the condition $C_{v}$ does not involve any of $\ell_{w_1}$, $\ell_{w_2}$.
Therefore the sum $S(T)$ can be written as follows:
\begin{multline*}
  S ( T ) = \sum_{(\ell_w)_{w\ne w_1,w_2} \atop (m_b)} \left[ \, \left( \prod_{w: w \ne w_1,w_2} \Cat_{\ell_w} t^{\ell_w} \right) 
\, \left( \prod_{b} \Cat_{m_b} t^{m_b} \right) \, \left( \prod_{v \not\ge_T p} \One[C_v] \right)\right. \\
\left. \cdot \sum_{L \ge 1}\,  \left( \prod_{v \ge_T p} \One[\tilde C_v] \right)
\, \left( \sum_{\ell_{w_1}+\ell_{w_2}=L-1} \Cat_{\ell_{w_1}} t^{\ell_{w_1}}\Cat_{\ell_{w_2}} t^{\ell_{w_2}}\right)\right],
\end{multline*}
where $\tilde C_v$ is the condition $C_v$ with the sum $\ell_{w_1}+\ell_{w_2}$ replaced by $L-1$.
The last sum in the above display simplifies as $t^{-L+1} \Cat_{L}$.
Having a sum over $L \ge 1$ is inconvenient, so we rewrite it as a sum over $L \ge 0$,
from which we remove the term corresponding to $L=0$. This gives
\begin{multline*}
  \hspace{-.5cm} S( T ) = t^{-1} \sum_{L,(\ell_w)_{w\ne w_1,w_2} \atop (m_b)}
 \Cat_{L} t^{L} \left( \prod_{w: w \ne w_1,w_2} \Cat_{\ell_w} t^{\ell_w} \right) 
\, \left( \prod_{b} \Cat_{m_b} t^{m_b} \right) \left( \prod_{v \not\ge_T p} \One[C_v] \right)\left( \prod_{v \ge_T p} \One[\tilde C_v] \right) \\
- t^{-1} \sum_{(\ell_w)_{w\ne w_1,w_2} \atop (m_b)}
\left( \prod_{w: w \ne w_1,w_2} \Cat_{\ell_w} t^{\ell_w} \right) 
\, \left( \prod_{b} \Cat_{m_b} t^{m_b} \right) \left( \prod_{v \not\ge_T p} \One[C_v] \right) \left( \prod_{v \ge_T p} \One[\tilde C_v] \right) ,
\end{multline*}
where all summation variables run over nonnegative integers.
The first (resp.~second) sum can be identified to $S(T_1)$,
where $T_1$ is the middle tree in Lemma~\ref{lem:twin_varnothing}
(resp.~$S(T_2)$,
where $T_2$ is the last tree in Lemma~\ref{lem:twin_varnothing}).
This proves the equation in the lemma.
The proof of the analogous statement with black vertices and shift $-1$
is similar.
\end{proof}

A second situation is when a black or white vertex has a child of the same color
decorated with a void symbol.
\begin{lemma}
\label{lem:consecutive_varnothing}
  $\begin{array}{c} \includegraphics{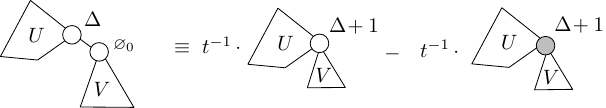} \end{array}$.   \\
    The same holds replacing white vertices with black vertices and the shift $+1$ with $-1$.\end{lemma}
\begin{proof}
The proof is similar to that of Lemma~\ref{lem:twin_varnothing}.
\end{proof}

\revision{
To help the reader, we provide in Table~\ref{table} a summary the relations obtained in this section, along with the hypotheses under which we apply each of them during the induction process. Lemma \ref{lem:reverting} does not appear, as it is only used specifically in the proof of Lemma \ref{lem:linear_system}.}

\begin{table}
\begin{center}
\begin{tabular}{|c|c|c|}
\hline
Relation & Hypothesis & Picture\\
\hline
Lemma \ref{lem:root_varnothing} & Gray root $(\varnothing,K)$ & \includegraphics[scale=.5]{Root_Varnothing}\\
\hline
Lemma \ref{lem:root_varnothing} & Black/white root $(\varnothing,K)$ & \includegraphics[scale=.5]{Root_Varnothing_2}\\
\hline
Lemma \ref{lem:equalitylabel} & Non-root vertex $(=,K)$ & \includegraphics[scale=.5]{Factorization_Equality}\\
\hline
Lemma \ref{lem:varnothinglabel} & Non-root vertex $(\varnothing,K)$ & \includegraphics[scale=.5]{Bringing_Up}\\
\hline
Lemma \ref{lem:varnothinglabel} & Non-root vertex $(\varnothing,K)$ with gray parent & \includegraphics[scale=.5]{Bringing_Up_bis}\\
\hline
Lemma \ref{lem:shift_k} & Vertex $(\geq,K)$, $K \neq 0$ & \includegraphics[scale=.5]{Shifting_K}\\
\hline
Lemma \ref{lem:shift_k} & Vertex $(\leq,K)$, $K \neq 0$ & \includegraphics[scale=.5]{Shifting_K_bis}\\
\hline
Lemma \ref{lem:leaf} & Leaf $(\geq,0)$ or $(\leq,0)$ & \includegraphics[scale=.5]{Leaf1} \tiny{and similar}	 \\
\hline
Lemma \ref{lem:leaf} & Gray leaf & \includegraphics[scale=.5]{Leaf5}\\
\hline
Lemma \ref{lem:twin_varnothing} & \begin{tabular}{c} Two siblings of the same color\\  with labels $(\varnothing,K)$,$(\varnothing,K')$ \end{tabular} & \includegraphics[scale=.5]{Twin_Varnothing}\\
\hline
Lemma \ref{lem:consecutive_varnothing} & \begin{tabular}{c} Black/white vertex with a child of the \\ same color labeled $(\varnothing,K)$ \end{tabular} & \includegraphics[scale=.5]{Consecutive_Varnothing}\\
\hline
\end{tabular}\smallskip
\end{center}
\caption{Summary of our reduction operations\label{table}}
\end{table}

\section{Further examples}
\label{sec:more_examples}
We present here further examples,
using now the tools from \cref{sec:relations}. 
\revision{In the first two examples (Sections~\ref{ssec:ex_P4+1} and~\ref{ssec:ex_tree_P5}), 
we are able to express the sum associated to the tree of interest in terms
of sums which we have already computed.
The example of \cref{ssec:ex_tree_6} is more subtle:
we will need to consider a whole family of trees simultaneously,
and compute the associated sums as solutions of a linear system.
This is a special case of more general relations which will be presented in Section~\ref{sec:long_stars}.}

In all these examples, we do not aim at
computing explicitly the sum $S(T)$, but at proving that  $S(T)$ belongs to~$\cA$
(we only consider here trees decorated by an equality symbol at the root).
We will therefore use the symbol $F \cong G$ to mean that $F-G$ is in $\cA$.
The purpose is to simplify the discussion; it should be nevertheless
clear that everything can be computed explicitly.

\subsection{A decorated path with four vertices}
\label{ssec:ex_P4+1}

We consider the tree of \cref{fig:more_examples}, left. 
This is a decorated variant of the path with four vertices that we denote $P_4^{+1}$.
We will see in the next section that this term appears in the recursive decomposition of $S(P_5)$
($P_5$ being the  path with five vertices), and thus deserves our attention.
By definition, we have
\[S(P_4^{+1})  = \sum_{\substack{\ell_1, \ell_2, m_1, m_2 \\ \ell_1 + \ell_2= m_1 + m_2 +1}}
 \Cat_{\ell_1} \Cat_{\ell_2}\Cat_{m_1}\Cat_{m_2}t^{\ell_1 + \ell_2+m_1 + m_2} \One[\ell_1 \le m_1].\]
  We could, as in \cref{ssec:ex_P4},
 express $S(P_4^{+1})$ in terms of similar expressions with $\ell_1 \le m_1$
 replaced by no condition/an equality/the reverse inequality $\ell_1 \ge m_1$.
 However, the shift by $+1$ breaks the symmetry between $\ell$'s and $m$'s
 and we do not have an obvious relation between $S(P_4^{+1}) $ 
 and the similar expression with the condition $\ell_1 \ge m_1$.
 Thus, we need to do something different.

\begin{figure}
  \begin{center}
    \includegraphics{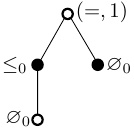} \qquad \qquad \includegraphics{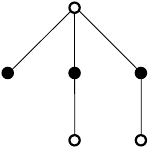} 
  \end{center}
  \caption{Examples of trees considered in \cref{sec:more_examples}. Left: A decorated tree with four vertices (\cref{ssec:ex_P4+1}).
  Right: a \enquote{long star} with six vertices (\cref{ssec:ex_tree_6}).}
  \label{fig:more_examples}
\end{figure}

 The idea is to introduce artificially comparisons between $\ell_2$ and $m_2$.
 Locally in this section, for $\bowtie_1,\bowtie_2$ in $\{=,\le,\ge,\varnothing\}$, let us define $A[\bowtie_1,\bowtie_2]$ as
\[A[\bowtie_1,\bowtie_2] = \sum_{\substack{\ell_1, \ell_2, m_1, m_2 \\ \ell_1 + \ell_2= m_1 + m_2 +1 \\ \ell_1 \bowtie_1 m_1,\ \ell_2 \bowtie_2 m_2}}
 \Cat_{\ell_1} \Cat_{\ell_2}\Cat_{m_1}\Cat_{m_2}t^{\ell_1 + \ell_2+m_1 + m_2}.\]
 Note that $A$ is symmetric in its two arguments. Also, by definition, $S(P_4^{+1})  = A[\le,\varnothing]$. 
Combining the identities
$  \One[\ell_i \ge m_i] = 1 + \One[\ell_i = m_i] -\One[\ell_i \le m_i]$
for $i$ in $\{1,2\}$, we get
\begin{equation}\label{eq:expansion_9_terms}
A[\ge,\ge]\, =\, A[\varnothing,\varnothing]+A[=,=]+A[\le,\le]-2A[\le,\varnothing]+2A[=,\varnothing]-2A[=,\le].
\end{equation}
Recall that our goal is to compute $A[\le,\varnothing]$ (or, at least prove that it is in $\cA$).
All terms involving an equality symbol factorize trivially (recall Lemma~\ref{lem:equalitylabel}) and
can be proved to be in $\cA$. The quantity $A[\le,\le]$ is equal to $0$,
since there is no tuples $(\ell_1, \ell_2, m_1, m_2)$ with $\ell_1 + \ell_2= m_1 + m_2 +1$ 
and $\ell_i \le m_i$ for $i$ in $\{1,2\}$ (empty sum).
For $A[\ge,\ge]$ we observe that $\ell_1 + \ell_2= m_1 + m_2 +1$ 
and $\ell_i \ge m_i$ for $i$ in $\{1,2\}$ implies that either $\ell_1 =m_1+1$ and $\ell_2=m_2$,
or the opposite. In both cases, the sum factorizes,
proving that $A[\ge,\ge]$ lies in $\cA$. We then consider $A[\varnothing,\varnothing]$.
The following computation corresponds to a double application of Lemma~\ref{lem:twin_varnothing}:
\begin{align*} 
  A[\varnothing,\varnothing] &= \sum_{\ell_1, \ell_2, m_1, m_2 \atop \ell_1 + \ell_2= m_1 + m_2 +1} \Cat_{\ell_1} \Cat_{\ell_2}\Cat_{m_1}\Cat_{m_2}t^{\ell_1 + \ell_2+m_1 + m_2} \\
  &=\sum_{L,M: L=M+1} t^{L+M} \left( \sum_{\ell_1 + \ell_2=L} \Cat_{\ell_1} \Cat_{\ell_2} \right)
  \left( \sum_{m_1 + m_2=M} \Cat_{m_1} \Cat_{m_2} \right)\\
  &= \sum_{L,M: L=M+1} \Cat_{L+1} \Cat_{M+1} t^{L+M}.
 \end{align*}
 Up to some boundary terms which are easily proved to be in $\cA$, the latter is $t^{-2} S_{=,1}$ with the notation 
 of \cref{ssec:base_case}, which is in $\cA$ (\cref{prop:basecase}).

 To sum up, we have proved that all terms in \cref{eq:expansion_9_terms}, except for $A[\le,\varnothing]$,
 belong to $\cA$. We can also verify easily that they have degree at most $4$.
  Therefore $S(P_4^{+1})  = A[\le,\varnothing]$ belongs to $\cA$, and has
degree at most 4, as predicted by \cref{thm:sum-pi-generalized}.
\medskip

 An interesting feature in this example is that we did not simply start form $S(P_4^{+1})$ and apply
 reduction operations from the previous section, but $S(P_4^{+1})$ appears as \enquote{unknown} in some equation,
 where we can prove that all other terms belong to $\cA$.
 For more complicated trees, we will need to construct linear systems of equations to compute $S(T)$
 (see \cref{ssec:ex_tree_6} for an example and \cref{sec:long_stars} for the general case).

 \subsection{The path with five vertices}
 \label{ssec:ex_tree_P5}
 We consider here the path $P_5$ with five vertices. Rooting it in its central vertex,
 we can write
 \[S(P_5)=\sum_{\substack{\ell_1, \ell_2, \ell_3, m_1, m_2 \\  \ell_1 + \ell_2+\ell_3= m_1 + m_2 \\ \ell_1 \le m_1, \ell_2 \le m_2}} \Cat_{\ell_1} \Cat_{\ell_2}\Cat_{\ell_3}\Cat_{m_1}\Cat_{m_2}t^{\ell_1 + \ell_2+\ell_3+m_1 + m_2}.\]
 As before, we can replace the inequality $\ell_1 \le m_1$ by no condition, $\ell_1=m_1$ and $\ell_1 \ge m_1$ successively
 (the last one appearing with a minus sign).
 The term with $\ell_1=m_1$ is easily factored out and proved to be in $\cA$.
 In the term with no condition between $\ell_1$ and $m_1$, we set $L=\ell_1+\ell_3$ and
 use the quadratic Catalan recurrence (see Lemma~\ref{lem:consecutive_varnothing}).
 Proving that this term is in $\cA$ reduces to proving that $S(P_4^{+1})$ is in $\cA$,
 which has been done in the previous section. We therefore have
 \[ S(P_5) \cong - \sum_{\substack{\ell_1, \ell_2, \ell_3, m_1, m_2 \\  \ell_1 + \ell_2+\ell_3= m_1 + m_2 \\ \bm{ \ell_1 \ge m_1}, \ell_2 \le m_2}} \Cat_{\ell_1} \Cat_{\ell_2}\Cat_{\ell_3}\Cat_{m_1}\Cat_{m_2}t^{\ell_1 + \ell_2+\ell_3+m_1 + m_2}.\]
 We do the same operation on $\ell_2 \le m_2$. With a similar justification as above, we get
 \[ S(P_5) \cong \sum_{\substack{\ell_1, \ell_2, \ell_3, m_1, m_2 \\  \ell_1 + \ell_2+\ell_3= m_1 + m_2 \\ \bm{\ell_1 \ge m_1}, \bm{\ell_2 \ge m_2}}} \Cat_{\ell_1} \Cat_{\ell_2}\Cat_{\ell_3}\Cat_{m_1}\Cat_{m_2}t^{\ell_1 + \ell_2+\ell_3+m_1 + m_2}.\]
The latter sum is however trivial since the condition in the index
 imply $\ell_3=0$, $\ell_1=m_1$ and $\ell_2=m_2$.
This proves that $S(P_5)$ is in $\cA$. Again, one easily checks that it has degree at most $5$.

\subsection{A tree with six vertices}
\label{ssec:ex_tree_6}
We consider here the tree $T_6$ from \cref{fig:more_examples}, right.
Such trees will be referred to as long stars, as they have a central vertex, to which branches of length at most 2 are attached, and will be considered more in detail in \cref{sec:long_stars}.
We have
\[ S(T_6) =\sum_{\substack{\ell_1, \ell_2, \ell_3, m_1, m_2,m_3 \\ 
\ell_1 + \ell_2+\ell_3= m_1 + m_2 +m_3\\
\ell_1 \le m_1, \ell_2 \le m_2}} 
\Cat_{\ell_1} \Cat_{\ell_2}\Cat_{\ell_3}\Cat_{m_1}\Cat_{m_2}\Cat_{m_3}t^{\ell_1 + \ell_2+\ell_3+m_1 + m_2+m_3} . \]
Again we need efficient notation: for $i+j \le 3$ we let $B_{i,j}$ the above sum, where the condition $\ell_1 \le m_1, \ell_2 \le m_2$
is replaced by
\[ \ell_h \le m_h \text{ for }h \le i, \text{ and } \ell_h \ge m_h \text{ for }i < h \le i+j.\]
In words, $i$ is the number of inequalities in the direction $\ell_h \le m_h$, and 
$j$ the number in the direction $\ell_h \ge m_h$ (we exploit here the symmetry, 
which ensures that $B_{i,j}$ does not depend on which indices are compared,
provided that the same index is not involved in several comparisons).
In particular, the quantity we are interested in is $S(T_6)=B_{2,0}$.
When $i+j \le 1$, we have (at least) two variables $\ell_h$ which are not involved in comparison,
and we can use the quadratic Catalan recurrence as usual (i.e.~Lemma~\ref{lem:twin_varnothing}) to reduce
our sum to sums with fewer variables, already proved to be in $\cA$.
Also the sums $B_{3,0}$ and $B_{0,3}$ are trivial, since, in these cases, the conditions
force $\ell_h=m_h$ for all $h$. Furthermore, using
the usual transformations
 $1+\One[\ell_h=m_h]=\One[\ell_h \ge m_h] + \One[\ell_h \le m_h]$ (Lemma~\ref{lem:reverting})
and observing that a condition $\ell_h=m_h$ in the index set leads to a factorization of the sum
into smaller sums, already proved to be in $\cA$, we get:
\[
  B_{3,0}+B_{2,1} \cong B_{2,0}, \quad B_{2,1} + B_{1,2} \cong B_{1,1}, \quad B_{1,2}+B_{0,3} \cong B_{0,2}.
  \]\[
B_{2,0}+B_{1,1} \cong B_{1,0}, \quad B_{1,1}+B_{0,2} \cong B_{0,1}.
\]
Recalling that $B_{3,0}$, $B_{0,3}$, $B_{1,0}$ and $B_{0,1}$ are in $\cA$, we can for example write
\[2B_{2,1} + B_{1,2} \cong (B_{3,0}+B_{2,1}) + (B_{2,1} + B_{1,2}) \cong B_{2,0}+B_{1,1} \cong B_{1,0} \cong 0,\]
and similarly, $B_{2,1} +2B_{1,2} \cong 0$.
Since the matrix $(\begin{smallmatrix} 2&1 \\ 1&2 \end{smallmatrix})$ is invertible, we conclude that both $B_{1,2}$ and $B_{2,1}$
  are in $\cA$. This implies in particular that $S(T_6) = B_{2,0} \cong B_{3,0}+B_{2,1}$ is in $\cA$ (in fact, this implies
  that all $B_{i,j}$ are in $\cA$).
  Again, a careful thinking about the above argument convinces oneself
  that the degree bounds $d_{S(T_6)} \le 6$ holds.

\section{Long stars}
\label{sec:long_stars}

As illustrated in the previous section, the relations of \cref{sec:relations}
are not sufficient to reduce our trees and compute $S(T)$ directly by induction.
In this section, we prove new relations,
in the form of linear systems involving some $S(T)$.
\revision{These relations will be needed in our algorithm to handle fringe subtrees of height $2$ that cannot be directly reduced by the local moves of \cref{sec:relations}, allowing us to overcome the last remaining obstacle in the proof of \cref{thm:sum-pi-generalized}.}

Let $R$ (for root or {\em rootstock}) be a decorated tree
with a special non-decorated leaf, which we call blossom.
We say that $R$ is trivial if it is reduced to that single blossom.
Given a decorated tree $T$, we write $R\mid T$ for the tree 
obtained by grafting $T$ on $R$, 
i.e.~replacing the blossom of $R$ by $T$.
The following families of trees will be of particular interest in this section: let
\[U^{\Delta}_{i,j,k}=\hspace{-3mm}\begin{array}{c}
\includegraphics{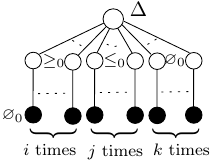}
\end{array},\quad
\overline{U}^{\Delta}_{i,j,k}=\hspace{-3mm}\begin{array}{c}
\includegraphics{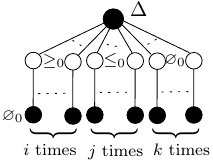}
\end{array},\quad
V^{\Delta}_{i,j,k}=\hspace{-3mm}\begin{array}{c}
\includegraphics{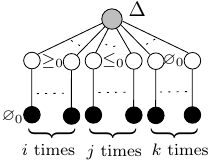}
\end{array}.\]
Only the root color changes between the pictures, and the black leaves all have the decoration $\varnothing_0$. We call all these trees \textit{long stars}.
As above, $\Delta$ represents a generic pair in $\{=,\le,\ge,\varnothing\} \times \mathbb Z$.
We note that these trees have many equivalent forms (equivalent in the sense that
the associated sums are the same). 
In particular, switching the color of a black leaf and of its white parent
does not change the sum (as it does not change the set of summation variables,
nor the conditions). 
Observe that, for any $k \geq 1$, in a tree $V^{\Delta}_{i,j,k}$ we can bring up a pair black leaf-white parent with $\varnothing$ decoration to the root (using Lemma~\ref{lem:varnothinglabel}). Hence, we always have $S(R \mid V^{\Delta}_{i,j,k})=S(R \mid T)$, where $T$ has strictly smaller size than $V^{\Delta}_{i,j,k}$. We will use it repeatedly in what follows.

\subsection{The $U$-family}
Our goal here is to establish some relation involving trees $U^{\Delta}_{i,j,k}$,
which will allow later for recurrences.
By symmetry, we have similar relations for $\overline{U}^{\Delta}_{i,j,k}$,
which we will not write down explicitly.

We first note that, as soon as $k \ge 1$, Lemma~\ref{lem:consecutive_varnothing} along with Lemma~\ref{lem:varnothinglabel} implies 
that
\begin{equation}\label{eq:reduc_U_k}
S\big(R|U^{\Delta}_{i,j,k}\big) = t^{-1} S\big(R|V^{\Delta+1}_{i,j,k}\big) - t^{-1} S\big(R|\overline{U}^{\Delta+1}_{i,j,k-1}\big).
\end{equation}
Moreover, from Lemmas~\ref{lem:equalitylabel} and \ref{lem:reverting}, one has
\begin{equation}\label{eq:reduc_U_ij}
S\big(R|U^{\Delta}_{i+1,j,0}\big) + S\big(R|U^{\Delta}_{i,j+1,0}\big) = S\big(R|U^{\Delta}_{i,j,1}\big) + S\big(R|U^{\Delta}_{i,j,0}\big) \cdot S_{=,0},
\end{equation}
where we recall from \cref{ssec:base_case} that $S_{=,0}\coloneqq \sum_{\ell \ge 0} \Cat_{\ell}^2 t^{2\ell}$.
Using \cref{eq:reduc_U_ij}, together with \cref{eq:reduc_U_k},
we can express any $S\big(R|U^{\Delta}_{i,j,0}\big)$
in terms of sums associated with smaller trees, $S\big(R|V^{\Delta+1}_{i',j',1}\big)$ where $i'+j'+1=i+j$, and $S\big(R|U^{\Delta}_{d,0,0}\big)$ where $d=i+j$.
To compute the latter, let us recall that in the definition of $S\big(R|U^{\Delta}_{d,0,0}\big)$,
we have variables $\ell_v, \ell_{w_1}, \dots, \ell_{w_d}, m_{b_1}, \dots, m_{b_d}$ corresponding to the nodes of the tree $U^{\Delta}_{d,0,0}$ ($v$ being the root of this tree),
and conditions $(C_{w_i}): \ell_{w_i} \ge m_{b_i}$.
We now consider two possible decorations of $U^{\Delta}_{d,0,0}$:
\begin{itemize}
\item If $\Delta\ =(\geq,0)$, then the condition $(C_v)$ is
\[\ell_v+ \ell_{w_1}+ \dots+ \ell_{w_d} \ge  m_{b_1}+ \dots+ m_{b_d}.\]
It is trivially implied by the conjunction of conditions $(C_w, \, w\text{ child of }v)$.
Thus we have
\begin{equation}
\label{eq:U1}
S\big(R|U^{\geq,0}_{d,0,0}\big)=S\big(R|U^{\varnothing,0}_{d,0,0}\big),
\end{equation}
and we can use Lemma~\ref{lem:varnothinglabel} (or Lemma~\ref{lem:root_varnothing} if $R$ is trivial) to simplify the right-hand side further (see \cref{sec:proof_main_theorem} for more details).
\item If $\Delta\ =(\leq,0)$, then the condition $(C_v)$ is
\[\ell_v+ \ell_{w_1}+ \dots+ \ell_{w_d} \le  m_{b_1}+ \dots+ m_{b_d}.\]
Assuming the conditions $(C_w, \, w\text{ child of }v)$, this can only happen
if $\ell_v=0$ and $\ell_{w_i}=m_{b_i}$ for all $i \ge 0$.
Hence we have
\begin{equation}
\label{eq:U2}
S\big(R|U^{\leq,0}_{d,0,0}\big)=S(R) \cdot (S_{=,0})^d,
\end{equation}
 where the blossom should be simply ignored when computing $S(R)$.
\end{itemize}

\subsection{The $V$-family}
The case of $S(R|V^{\Delta}_{i,j,k})$ is slightly more delicate.
When $k \ge 2$, applying twice Lemma~\ref{lem:twin_varnothing} (to a pair of white vertices with decoration $\varnothing$, and then to the pair of the corresponding black leaves) yields
\begin{equation}
\label{eq:V1}
S(R|V^{\Delta}_{i,j,k}) = t^{-2} \, S(R|V^{\Delta}_{i,j,k-1}) + t^{-2} \, S(R|V^{\Delta}_{i,j,k-2}) - t^{-2} \, S(R|U^{\Delta}_{i,j,k-2}) - t^{-2} \, S(R|\overline{U}^{\Delta}_{i,j,k-2}).
\end{equation}

Again, $S(R|V^{\Delta}_{d,0,0})$ can be simplified depending on the decoration $\Delta$:
\begin{itemize}
\item If $\Delta\ =(\geq,0)$, then the condition $(C_v)$ is implied by 
$(C_w, \, w\text{ child of }v)$, and
\begin{equation}
\label{eq:V2}
 S\big(R|V^{\geq,0}_{d,0,0}\big)=S\big(R|V^{\varnothing,0}_{d,0,0}\big),
\end{equation}
which can be simplified further using Lemma~\ref{lem:varnothinglabel} (or Lemma~\ref{lem:root_varnothing} if $R$ is trivial).
\item If $\Delta\ =(\leq,0)$, then the condition 
$(C_v) \wedge (C_w, \, w\text{ child of }v)$ forces $\ell_{w_i}=m_{b_i}$ for all $i \ge 0$,
and we have
\begin{equation}
\label{eq:V3}
S\big(R|V^{\leq,0}_{d,0,0}\big)=S(R) \cdot (S_{=,0})^d.
\end{equation}
\end{itemize}
The same holds by symmetry for $S(R|V^{\Delta}_{0,d,0})$, exchanging the two cases above.
The following lemma \revision{expresses} $S(R|V^{\Delta}_{i,j,0})$ for $i,j>0$
\revision{in terms of}
 $S(R|V^{\Delta}_{d,0,0})$, $S(R|V^{\Delta}_{0,d,0})$
and $S(R|V^{\Delta}_{i',j',k})$ for either $k = 2$ and $i'+j'=i+j-2$, or $i'+j'+k < i+j$.
\begin{lemma}
\label{lem:linear_system}
  For any $R$, any $\Delta$ and any $d \ge 3$, we have
\[  \,  \begin{pmatrix}
    S( R \mid V^{\Delta}_{1,d-1,0} )\\
    \\
    \vdots \\
    \\
     S( R \mid V^{\Delta}_{d-1,1,0} )
  \end{pmatrix} 
  \, = \,  \revision{\begin{pmatrix}
    2 & 1 & 0 & \ldots &0 \\
    1 & 2 & 1 & \ddots & \vdots \\
    0 & \ddots  & \ddots & \ddots & 0 \\
    \vdots & \ddots & 1 & 2 & 1  \\
    0 & \ldots &0 & 1 & 2
  \end{pmatrix}^{-1}} \begin{pmatrix}
    X_1 - S(R|V^{\Delta}_{0,d,0}) \\ X_2 \\ \vdots \\ X_{d-2} \\ X_{d-1} -S(R|V^{\Delta}_{d,0,0}) 
  \end{pmatrix}, \]
  where, for $1 \le i \le d-1$,
 
  $$X_i=S(R \mid V^{\Delta}_{i-1,d-1-i,2}) + 2S(R \mid V^{\Delta}_{i-1,d-1-i,1}) \cdot S_{=,0} + S(R \mid V^{\Delta}_{i-1,d-1-i,0}) \cdot (S_{=,0})^2.
  $$
  For $d=2$, we have a single equation
  \[2S(R \mid V^{\Delta}_{1,1,0})=X_1-S(R \mid V^{\Delta}_{0,2,0})-S(R \mid V^{\Delta}_{2,0,0}).\]
\end{lemma}

\begin{proof}[Proof of Lemma~\ref{lem:linear_system}]
  Let $i,j>0$ with $i+j=d$. Using twice Lemma~\ref{lem:reverting} (together with Lemma~\ref{lem:equalitylabel}),
  we have
  \begin{multline*} S(R \mid V^{\Delta}_{i+1,j-1,0})+ 2 S(R \mid V^{\Delta}_{i,j,0}) + S(R \mid V^{\Delta}_{i-1,j+1,0})  \\
  = S (R \mid V^{\Delta}_{i,j-1,1}) + S (R \mid V^{\Delta}_{i-1,j,1}) + (S (R \mid V^{\Delta}_{i,j-1,0}) + S (R \mid V^{\Delta}_{i-1,j,0}) )\cdot S_{=,0}.
\end{multline*}
Using again Lemma~\ref{lem:reverting} (and Lemma~\ref{lem:equalitylabel}), this simplifies as
 \begin{multline*} S(R \mid V^{\Delta}_{i+1,j-1,0})+ 2 S(R \mid V^{\Delta}_{i,j,0}) + S(R \mid V^{\Delta}_{i-1,j+1,0})  \\
   = S (R \mid V^{\Delta}_{i-1,j-1,2}) + 2 S (R \mid V^{\Delta}_{i-1,j-1,1}) \cdot S_{=,0} + S (R \mid V^{\Delta}_{i-1,j-1,0}) \cdot S_{=,0}^2.
 \end{multline*}
\revision{These relations can be written in matrix form as done in the lemma, provided that the tridiagonal $(d-1) \times (d-1)$ matrix
 \[ \begin{pmatrix}
    2 & 1 & 0 & \ldots &0 \\
    1 & 2 & 1 & \ddots & \vdots \\
    0 & \ddots  & \ddots & \ddots & 0 \\
    \vdots & \ddots & 1 & 2 & 1  \\
    0 & \ldots &0 & 1 & 2
  \end{pmatrix}\]
  is invertible. To prove the latter,
  let $\delta_{d}$ be its determinant. We have 
  $\delta_2=2$ and $\delta_3=3$, and 
  for $d \ge 4$, expanding along the first row gives $\delta_{d}=2 \delta_{d-1} - \delta_{d-2}$. By induction, we get $\delta_d=d$, proving that the matrix is invertible. This ends the proof of the lemma.}
  \end{proof}

To conclude this section, let us note that Lemma~\ref{lem:reverting} (together with Lemma~\ref{lem:equalitylabel}) yields
\begin{equation}
\label{eq:V4}
S (R \mid V^{\Delta}_{i,j,1})=  S(R \mid V^{\Delta}_{i+1,j,0}) + S(R \mid V^{\Delta}_{i,j+1,0}) -
S (R \mid V^{\Delta}_{i,j,0}) \cdot S_{=,0},
\end{equation}
which expresses $ S (R \mid V^{\Delta}_{i,j,1})$ in terms of sums that we have already computed or associated with smaller trees.

\subsection{Long stars with an equality decoration}
We consider here the case where $R$ is trivial,
and $\Delta$ is of the form $(=,K)$ for some $K$ in $\Z$,
and we will use the above relations to prove \cref{thm:sum-pi-generalized}
in the case of long stars.
\begin{proposition}\label{prop:UV}
For any $i,j,k \ge 0$ and any $K$,
all of $S\big(U^{(=,K)}_{i,j,k}\big)$, $S\big(\overline{U}^{(=,K)}_{i,j,k}\big)$ and $S\big(V^{(=,K)}_{i,j,k}\big)$
belong to $\cA$. Moreover, we have the degree bounds
\[d_{S\big(U^{(=,K)}_{i,j,k}\big)} \le 2(i+j+k)+1, \quad d_{S\big(\overline{U}^{(=,K)}_{i,j,k}\big)} \le 2(i+j+k)+1,
\quad d_{S\big(V^{(=,K)}_{i,j,k}\big)} \le 2(i+j+k).\]
\end{proposition}
Note that the right-hand side in the above degree bounds is the number of non-gray vertices
in the corresponding tree, which is consistent with~\cref{thm:sum-pi-generalized}.
\begin{proof}
We prove it by induction on the size $|T|$ of a long star $T$, defined as
\begin{align*}
\left|U^{\Delta}_{i,j,k}\right|=\left|\overline{U}^{\Delta}_{i,j,k}\right|=\left|V^{\Delta}_{i,j,k}\right|=i+j+k.
\end{align*}
Long stars of size $0$ are restricted to a single vertex and the proposition holds trivially in this case.
Let $s\ge 0$.
Assume that the result holds for all long stars of size at most $s$,
and let $T$ be a long star of size $s+1$.
\smallskip

\emph{The $V$ family.}
Consider first the case $T=V^{(=,K)}_{i,j,k}$ with $i+j+k=s+1$. If $k \geq 2$, then \eqref{eq:V1} and the induction hypothesis imply that $S(T)$ is in $\cA$
and that the degree bound holds.

We now consider the tree $T=V^{(=,K)}_{s+1,0,0}$.
Denoting by $\rho$ the root of $T$,
the condition $(C^T_\rho)$ writes $\ell_{w_1}+\ldots+\ell_{w_{s+1}} = m_{b_1}+\ldots+m_{b_{s+1}}+K$. Hence, under the conditions $\ell_{w_i} \geq m_{b_i}$ for all $1 \leq i \leq s+1$, we have that $$\left(\ell_{w_1}-m_{b_1}, \ldots, \ell_{w_{s+1}}-m_{b_{s+1}}\right) \in E_{s+1,K} \coloneqq  \left\{ x_1, \ldots, x_{s+1} \in \Z_+, \sum_{j=1}^{s+1} x_j = K \right\}.$$
Note that $E_{s+1,K}$ is finite for all $K$. We can therefore write $S(T)$ as a finite sum over $\mathbf{x} \coloneqq  (x_1, \ldots, x_{s+1})$ in $E_{s+1,K}$:
\begin{align*}
S(T) = \sum_{\mathbf{x} \in E_{s+1,K}} S\left(  T, \bigwedge_{j=1}^{s+1} \left[\ell_{w_j}=m_{b_j}+x_j\right]\right),  
\end{align*}
where $S(T,C)$ is the same as $S(T)$ with the additional condition $C$ on the
vertices. Each of the (finitely many) terms $S(T,C)$ above 
is a product of factors of the form $S(T')$, where $T'$ is a tree with two vertices and an equality decoration at the root.
Hence, using \cref{prop:basecase}, we conclude that $S(T)=S(V^{(=,K)}_{s+1,0,0})$ is in $\cA$ and has degree at most $2(s+1)$.
Swapping the roles of $(\ell_{w_j})$ and $(m_{b_j})$, the same
holds for $V^{\Delta}_{0,s+1,0}$.

We now consider the case where $T=V^{(=,K)}_{i,j,0}$ where $i,j \geq 1$ and $i+j=s+1$. Here we use  Lemma~\ref{lem:linear_system} along with \eqref{eq:V1}
and the induction hypothesis, which proves that $S(T)$ is in $\cA$
with the appropriate degree bound.
This completes the case of trees $V^{(=,K)}_{i,j,k}$ with $k=0$.

Finally, if $k=1$, we see from \eqref{eq:V4} and the results just above 
that $S(T)$ is also in $\cA$
with the appropriate degree bound.
In the end we have proved the desired statement for any long star of the $V$ family of size $s+1$.
\medskip

\emph{The $U$ family.}
We now consider a tree $T=U^{(=,K)}_{i,j,k}$ of size $|T|=s+1$. 
First, observe that if $k \geq 1$, we get directly from \eqref{eq:reduc_U_k} 
that $S(T)$ is in $\cA$
with the appropriate degree bound (using the proved statement for trees 
of the $V$ family of size $s+1$).

Now, consider $T=U^{(=,K)}_{s+1,0,0}$. 
The same argument as for the $V$ family above shows 
that $S(T)$ is in $\cA$ with the appropriate degree bound. 
Finally, if $T=U^{(=,K)}_{i,s+1-i,0}$ for $0 \leq i \leq s$, observe that \eqref{eq:reduc_U_ij} allows to write $S(T)$ as a linear combination of $S(U^{\Delta}_{s+1,0,0})$ and of some $S(T')$ where $T'$ is a long star in the $U$-family, either of size at most $s$,
or indexed by a triple $(i,j,k)$ with $k \ge 1$. Such expressions have already be proved to be in $\cA$ with appropriate degree bounds,
so this also holds for $S(T)$.
This shows that the proposition holds for all long stars of the $U$ family of size $s+1$.
\medskip

\emph{The $\overline{U}$ family.}
Finally, the family $\overline{U}$ is treated exactly the same way as the $U$ family.
This concludes our induction step, and the proposition is proved.
\end{proof}

\section{Proof of the main theorem}
\label{sec:proof_main_theorem}

We can now prove the generalized version of our main result, \cref{thm:sum-pi-generalized}. \revision{To help the reader, we provide in Table~\ref{table:notation} a
summary of our main notation.}
Additionally, we will write $S(T) \in \cA/\cB$ to mean that $S(T)$ is in $\cA$ if its root is decorated by an equality symbol, and in $\cB$ otherwise.
\begin{table}
\begin{center}
\begin{tabular}{|c|c|}
\hline
$H_1(t)$ & ${}_{2}^{}{{}{{}{F_{1}^{}}}}\big(-\tfrac12,-\tfrac12; 1;16t^2\big)$\\
\hline
$H_2(t)$ & ${}_{2}^{}{{}{{}{F_{1}^{}}}}\big(-\tfrac12,\tfrac12; 2;16t^2\big)$\\
\hline
$\cA$ & $\Q\Big[t,t^{-1}, \, H_1(t), H_2(t) \Big]$\\
\hline
$\cB$ & $\cA\Big[\sqrt{1-4t} \Big]$\\
\hline
$d_b$ & filtered degree of an element $b$ of the algebra $\cB$\\
\hline
$\Delta_v=(\bowtie_v,K_v)$ & decoration of a vertex $v$, corresponding to the condition\\
&$\left(\sum_{w \le_T v \atop \eps_w=1} \ell_w -  \sum_{b \le_T v \atop \eps_b=-1} m_b \right) \bowtie_v \left( \sum_{v' \le_T v} K_{v'} \right)$\\
\hline
$U^\Delta_{i,j,k}$ & long star with white root decorated $\Delta$\\
\hline
$\overline{U}^\Delta_{i,j,k}$ & long star with black root decorated $\Delta$\\
\hline
$V^\Delta_{i,j,k}$ & long star with gray root decorated $\Delta$\\
\hline
\end{tabular}\smallskip
\end{center}
\caption{Table of notation\label{table:notation}}
\end{table}

\cref{thm:sum-pi-generalized} will be proved by induction on the so-called \emph{total path length} of the tree $T$, defined as 
\begin{align*}
\cL(T) \coloneqq  \sum_{v \in T} d(\rho, v),
\end{align*}
where we recall that $\rho$ is the root of $T$ and $d$ the graph distance on $T$.

In the induction step, we shall decompose our tree thanks to the relations obtained in Lemmas \ref{lem:root_varnothing}-\ref{lem:consecutive_varnothing}. Let us start with the following reduction. For clarity, we postpone all degree discussions to the end of the proof.

\begin{definition}\label{def:goodtrees}
We say that a decorated tree $T$ is \textit{good} if it has the following properties:
\begin{itemize}
\item[(i)] No vertex has label $=$, except possibly the root $\rho$;
\item[(ii)] If a vertex $v$ is such that $\bowtie_v \in \{ \geq, \leq \}$, then $K_v=0$;
\item[(iii)] There are no gray leaves, and all leaves have decoration $\varnothing_0$;
\item[(iv)] There are no pairs of leaves with the same parent and the same color;
\item[(v)] No leaf has a parent of its color;
\item[(vi)] No leaf has a gray parent.
\end{itemize}
\end{definition}

The interest of this definition lies in the following result.

\begin{proposition}
\label{prop:goodtrees}
Fix $k \ge 0$. Assume that $S(T') \in \cA/\cB$ for
\begin{itemize}
  \item any tree $T'$ with $\cL(T') \le k$;
  \item any good tree $T'$ with $\cL(T') \le k+1$.
\end{itemize}
Then $S(T) \in  \cA/\cB$
for all trees $T$ with $\cL(T) \le k+1$.
\end{proposition}

\begin{proof}
Let $T$ be a tree with $\cL(T) = k+1$, we want to prove that $S(T) \in  \cA/\cB$.
Since we assume that this holds for any tree $T'$ with $\cL(T') \le k$, and for any good tree with $\cL(T') \le k+1$, it suffices to consider the case when $T$ is not a good tree and $\cL(T) = k+1$. Not being a good tree means that $T$ violates one of the items (i) to (vi) of \cref{def:goodtrees}. Let us consider these items one by one.

By Lemma~\ref{lem:equalitylabel}, if $T$ violates item (i) (i.e.~a nonroot vertex has label $=$), then $S(T)$ can be written as $S(T)=S(U)\cdot S(V)$ with $\cL(U) \leq k, \, \cL(V) \leq k$, and $S(T)$ is in $\cA/\cB$.

Now assume that there is a vertex $v$ in $T$ with label $(\bowtie_v,K_v)$, where $\bowtie_v \in \{ \geq, \leq \}$ and $K_v \neq 0$. Then, using Lemma~\ref{lem:shift_k} and again Lemma~\ref{lem:equalitylabel}, we can write $S(T) = S(U) + S(V) \cdot S(W)$ or $S(T)=S(U)-S(V) \cdot S(W)$, where $V,W$ have smaller total path lengths than $T$,
$U$ has the same total path length as $T$,
but the sum in $U$ of $|K_v|$
over its $\{\le,\ge\}$-decorated vertices $v$ is smaller than that in $T$.
By an immediate induction, we can assume that all $\{\le,\ge\}$-decorated vertices $v$ in $T$ satisfy $K_v=0$, i.e.~that $T$ satisfies item (ii).

By the last equality of Lemma~\ref{lem:leaf}, if $T$ has a gray leaf then either $S(T)=0$ or $S(T)=S(T')$ for $T'$ such that $\cL(T')=k$. 
In both cases, we can conclude that $S(T)$ is in $\cA/\cB$.
So we assume that $T$ has no gray leaves. If $T$ contains a leaf with a decoration different from $\varnothing_0$,
then we can use Lemma~\ref{lem:varnothinglabel} (if the decoration is $(\varnothing,K)$, $K\ne 0$) or Lemma~\ref{lem:leaf} (if the decoration is $\le_0$ or $\ge_0$)
to write $S(T)=S(T')$, where $T'$
has the same total path length has $T$ but fewer leaves with a decoration different from $\varnothing_0$, or violates item (i). Again, by an immediate induction,
we can assume that all leaves of $T$ are decorated with $\varnothing_0$,
i.e.~that $T$ satisfies item (iii).

Now, if (iv) or (v) does not hold, then, recalling that all leaves in $T$ have decoration $\varnothing_0$ and are non-gray,
by Lemma~\ref{lem:twin_varnothing} or Lemma~\ref{lem:consecutive_varnothing} we can write $S(T)=t^{-1}S(U)-t^{-1}S(V)$ 
where $U$ and $V$ have smaller total path length than $T$.

Finally, if (vi) does not hold, by the second part of Lemma~\ref{lem:varnothinglabel} we can write $S(T)=S(U)$ where $U$ has a smaller total path length as $T$.
In all cases, we conclude by induction that $S(T)$ is in $\cA/\cB$,
proving the proposition.
\end{proof}

\cref{prop:goodtrees} shows that we only need to consider good trees. We now investigate in more detail the structure of a good tree. Using a standard terminology in the domain, we call \textit{fringe subtree} of $T$ a subtree of $T$ made of one vertex and all its descendants.

\begin{lemma}
\label{lem:fringe_subtrees}
Let $T$ be a good tree. Then:
\begin{itemize}
\item[(i)] any fringe subtree of height $1$ of $T$ is either a black vertex with a single white child, or a white vertex with a single black child.
\item[(ii)] a fringe subtree $U$ of height $2$ of $T$ is either:
\begin{itemize}
\item[(a)] a gray vertex $v$, such that all subtrees rooted at children of $v$ are as in (i);
\item[(b)] a white vertex $v$, potentially with one black leaf child
  and such that other  subtrees rooted at children of $v$ are as in (i)
  (since $U$ has height $2$, $v$ has at least one non-leaf child); 
\item[(c)] same as (ii b), swapping black and white.
\end{itemize}
\end{itemize} 
\end{lemma}
\begin{proof}
Item (i) directly follows from Definition \ref{def:goodtrees}, items (iii) to (vi). 
Let us prove (ii). Let $U$ be a fringe subtree of $T$ of height $2$. If the root of $U$ is gray, then it cannot have a leaf as a child by \cref{def:goodtrees} (vi). Hence all its children are roots of fringe subtrees of height $1$, proving (ii,a). If the root $v$ of $U$ is white, then it can have a black leaf as a child, but not more than one and it cannot have a leaf of another color as child, by \cref{def:goodtrees}, items (iii) to (v). All other subtrees rooted at children of $v$ are fringe subtrees of height $1$ and hence, are as described in (i). This proves (ii,b), and (ii,c) is proved the same way, swapping black and white.
\end{proof}

We now have all the tools needed to prove \cref{thm:sum-pi-generalized}.

\begin{proof}[Proof of \cref{thm:sum-pi-generalized}]
\revision{We first prove by induction on $k$ the following statement (we will return to the degree bound afterwards).}
\begin{quote}
\revision{For any decorated tree $T$ of path length at most $k$, the associated series $S(T)$ is in $\mathcal B$. Furthermore, if the root of $T$ is decorated with an equality symbol, then $S(T)$ is in $\mathcal A$.}
\end{quote}
\revision{We start with $k=0$.}
If $T$ has total path length $0$, i.e. is restricted to a single vertex, then $S(T)$ is in $\mathbb Q[t,C(t)]$.
Indeed, in that case, either $S(T)$ is a polynomial in $t$ (which holds in particular
when the decoration of the single vertex is an equality symbol),
 or $\sum_{\ell \ge 0} \Cat_{\ell} t^{\ell}  - S(T)=  C(t) - S(T)$ is a polynomial in $t$.
 \revision{This proves the above statement for $k=0$.}
 
 Now, let $k \ge 0$, and assume that for any decorated tree $T'$ with $\cL(T') \le k$,
 we have $S(T') \in \cA/\cB$. 
 We want to prove that for any $T$ with $\cL(T) = k+1$,
 the sum $S(T)$ lies in $\cA/\cB$. 
 From \cref{prop:goodtrees}, it is enough to prove it for good trees $T$.
  
 Let us consider a good tree $T$ with $\cL(T) = k+1$.
By Lemma~\ref{lem:fringe_subtrees} (i),
if $T$ has height $1$, then $T$ has exactly $2$ vertices, 
and the result is a consequence of~\cref{prop:basecase}.
Assume now that $T$ has height at least $2$,
and let $v$ be a vertex of $T$ such that the fringe subtree rooted at $v$ has height exactly $2$.
We can assume that $v$ is decorated either
 with $(\ge,0)$ or with $(\le,0)$. Indeed,
 \begin{itemize}
 \item Since $T$ is a good tree, an equality decoration can only be found if $v$ is the root of $T$, and this case has already been treated, see \cref{prop:UV}.
 \item If $v$ is decorated by $(\varnothing,K)$, then $S(T)$ can be expressed
 in terms of sums associated to smaller trees, either by Lemma~\ref{lem:root_varnothing} or Lemma~\ref{lem:varnothinglabel} 
(depending on whether $v$ is the root or not).
\item Recall that, in good trees, decorations $(\ge,K)$ and $(\le,K)$ with $K \ne 0$,
are forbidden (\cref{def:goodtrees}, item (ii)).
 \end{itemize}
By  Lemma~\ref{lem:fringe_subtrees} (ii), this fringe subtree can be of three different kinds
and we will treat each case separately.
\smallskip

{\em Case (a) --} $v$ is a gray vertex, 
and all subtrees attached to $v$ consist of exactly one black and one white vertex.
Without loss of generality, we assume that, in each branch the white vertex is the parent and
the black vertex is the leaf (switching the colors of the black and the white vertices in a branch does not change the associated sum, nor the total path length of the tree).
With the notation of \cref{sec:long_stars}, this means that $T$ writes as $R \mid V^{\Delta}_{i,j,k}$
for some $R$ and some $i,j,k \ge 0$ and $\Delta \in \{\le_0,\ge_0\}$.

If $k \ge 2$, \cref{eq:V1} expresses $S(T)$ in terms of $S(T')$ with $\cL(T') <  \cL(T)$,
proving that $S(T) \in \cA/\cB$. 
If $(i,j,k) = (d,0,0)$ or $(i,j,k) = (0,d,0)$, then \cref{eq:V2,eq:V3}
implies that $S(T)$ belongs to $\cA/\cB$.
Consider now the case where $k=0$, but $i,j>0$.
We use Lemma~\ref{lem:linear_system} with $d=i+j$.
 Then $S(T)=S(R \mid V^{\Delta}_{i,j,0})$ is one of the components
 of the vector on the left-hand side. But all quantities on the right-hand side
 have been assumed or proved to belong to $\cA/\cB$:
 they are either of the form $S(T')$ with $\cL(T') <  \cL(T)$, or 
 of the form $S(R \mid V^{\Delta}_{d,0,0})$, $S(R \mid V^{\Delta}_{0,d,0})$, $S(R \mid V^{\Delta}_{i',j',2})$ with $i'+j'+2=i+j$, all these trees having the same total path length as $T$
 and having already been treated.
 
 It remains to consider the case $k=1$. Then \eqref{eq:V4}
 expresses $S(T)$ in terms of some $S(T')$, with $T'$ of smaller total path length ($S(R \mid V^{\Delta}_{i,j,0})$) or already proved to be in  $\cA/\cB$
 ($S(R \mid V^{\Delta}_{i,j+1,0})$ and $S(R \mid V^{\Delta}_{i+1,j,0})$).
 \smallskip
 
 {\em Case (b1) --}
$v$ is a white vertex and all subtrees attached to $v$ 
consist of exactly one black and one white vertex.
As before, one may assume that we have only black leaves, by switching if necessary
the colors of the white and the black vertices of a given branch -- this changes neither $S(T)$, nor $\cL(T)$.
Then $T$ writes as $R \mid U^{\Delta}_{i,j,k}$
for some $R$ and some $i,j,k \ge 0$.
 Recalling that $v$ is decorated either with $(\ge,0)$ or with $(\le,0)$, \cref{eq:reduc_U_k,eq:reduc_U_ij,eq:U1,eq:U2} 
allows to express $S(T)$ in terms of $S(T')$ which are assumed or already proved to belong to  $\cA/\cB$, implying that $S(T)$ itself is in $\cA/\cB$.
\smallskip

 {\em Case (b2) --} $v$ is a white vertex, there is exactly one black leaf attached to $v$ and all other subtrees attached to $v$ 
consist of exactly one black and one white vertex.
 Again, one may assume that we have only black leaves.
 Then, using the second part of Lemma~\ref{lem:varnothinglabel} in the unusual direction
  to \enquote{pull down} $v$ and its black leaf, 
 we have that $S(T)=S(T_2)$, where $T_2$ is of the form $R \mid V^{\Delta}_{i,j,k}$ with $k \ge 1$.
 This operation slightly increases the total path length, namely we have
 $\cL(T_2)=\cL(T)+d(\rho,v)+2$.
 
  We then use the same argument as in Case (a)
 to write $S(T_2)$ in terms of trees with smaller total path length.
 One can check that in these reductions, all trees $T'$ that appear
 satisfy $\cL(T') \le \cL(T_2)-d(\rho,v)-3$
  (reducing $i+j+k$ by $1$ in a long star
 reduces the total path length by $3$).
 In particular, all trees $T'$ in the reduction satisfy $\cL(T')< \cL(T)$,
 and we can use the induction hypothesis, proving that  $S(T)=S(T_2)$ belongs to 
 $\cA/\cB$ in this case as well.
 \smallskip
 
The case (c) of Lemma~\ref{lem:fringe_subtrees} (ii) is symmetric to (b), and hence treated similarly. This completes the induction, proving that $S(T)$ belongs to 
 $\cA/\cB$ for any decorated tree $T$.
 \medskip

We finally prove the degree bound $d_{S(T)} \le \tilde V_T$ on these sums. 
Let us denote (DB) this bound.
In order to show it, it is enough to prove that all the results used throughout the proof are compatible with these bounds.

\begin{itemize}
\item \cref{prop:basecase} considers trees $T$ with $2$ vertices, and shows that $d_{S(T)} \le 2$. Hence these trees satisfy (DB).
\item In Lemmas \ref{lem:root_varnothing}, \ref{lem:equalitylabel} and \ref{lem:varnothinglabel}, using the notation of these lemmas, if $U,V,W$ satisfy (DB), then the trees on the left-hand sides also satisfy (DB).
\item In each of the relations shown in Lemmas \ref{lem:shift_k}, \ref{lem:reverting} and \ref{lem:leaf}, if all of the involved trees but one satisfy (DB), then the last tree also clearly satisfies (DB).
\item In Lemmas \ref{lem:twin_varnothing} and \ref{lem:consecutive_varnothing}, since the number of non-gray vertices in the tree $T_1$ on the left-hand side is strictly larger than that in the trees $T_2,T_3$ on the right-hand side, if $T_2,T_3$ satisfy (DB) then it is also the case for $T_1$.
\item Equations \eqref{eq:reduc_U_k}, \eqref{eq:U1} and \eqref{eq:V2} concern trees with the exact same number of non-gray vertices. Hence, if all of the involved trees but one satisfy (DB), then the last tree also satisfies it.  
\item Recalling that $d_{S_{=,0}}=2$, if all trees involved in Equations \eqref{eq:reduc_U_ij}, \eqref{eq:U2},\eqref{eq:V3} or \eqref{eq:V4} except one satisfy (DB), so does the last tree. 
\item In \eqref{eq:V1}, if the trees $V_{i,j,k-1}^{\Delta}, V_{i,j,k-2}^{\Delta}, U_{i,j,k-2}^{\Delta}, \overline{U}_{i,j,k-2}^{\Delta}$ satisfy (DB), then so does $V_{i,j,k}^{\Delta}$.
\item Finally, in Lemma~\ref{lem:linear_system}, if all trees appearing in the right-hand side satisfy (DB), then the trees on the left-hand side satisfy (DB) as well.
\end{itemize}
This allows to prove by induction that all decorated trees satisfy (DB). This ends the proof of the theorem.
\end{proof}

We finally show that the degree bound in \cref{thm:sum-pi-generalized} is tight, as announced in the introduction.
\begin{proposition}
\label{prop:long stars maximize the degree}
For any $n \geq 0$, there exists a decorated tree $T_n$ with $n$ non-gray vertices
such that $d_{S(T_n)}=n$ and $S(T_n)(\frac14)$ is a polynomial in $1/\pi$ of degree $\lfloor n/2 \rfloor$.
\end{proposition}

\revision{Observe that we get tightness of the degree bound for decorated trees. An analogous result for ordinary trees is plausible but we do not have a proof of it.}

\begin{proof}
First assume that $n$ is even, namely $n=2d$. Then, $V_{d,0,0}^{(\leq,0)}$ has $2d$ non-gray vertices, and by \eqref{eq:V3}:
\begin{align*}
S(V_{d,0,0}^{(\leq,0)}) &= (S_{=,0})^d.
\end{align*}
By Section \ref{ssec:base_case}, $S_{=,0}=\frac{1}{4t^2}\left( -1+ {}_{2}^{}{{}{{}{F_{1}^{}}}}\big(-\tfrac12,-\tfrac12; 1; 16t^2\big)\right)$, and hence has degree $2$.
We conclude that $S(T_n)$ has degree $2d$, as wanted.
Also $S_{=,0}(\frac14)=4-\frac4{\pi}$, so that $S(V_{d,0,0}^{(\leq,0)})(\frac14)$
is a polynomial in $1/\pi$ of degree $d$.

On the other hand, if $n=2d+1$, we consider the tree $U_{d,0,0}^{(\geq,0)}$ 
which has $2d+1$ non-gray vertices.
 By \eqref{eq:U1} and Lemma~\ref{lem:root_varnothing}, we have $S(U_{d,0,0}^{(\leq,0)})=\frac{1-\sqrt{1-4t}}{2t} \cdot (S_{=,0})^d$, which has degree~$2d+1$.
 At $t=1/4$, we get $S(U_{d,0,0}^{(\leq,0)})(\frac14)=2 \, S_{=,0}(\frac14)^d$,
 which is a polynomial in $1/\pi$ of degree $d$.
\end{proof}

\section{The special case of stars}
\label{sec:stars}

For all $s \geq 0$, let us consider the tree $T_s$ of height $1$ with a white root decorated with $(=,0)$
and $s$ black leaves decorated with $(\varnothing,0)$.
We consider the corresponding sum, \revision{and its evaluation at $1/4$,}
 \begin{equation}\label{eq:basic}
	\revision{S(T_s)} =
\sum_{m_1,\dots,m_s \geq 0}
\Cat_{m_1}
\cdots
\Cat_{m_s}
\Cat_{m_1 + \cdots + m_s}
\revision{{t}^{2m_1 + \cdots + 2m_s}, \qquad A_s \coloneqq S(T_s)(1/4)}.
\end{equation}
\revision{Using computer algebra software, one can compute the first values 
\[A_0 = 1, \quad A_1 = \frac{16}{\pi} - 4, \quad  A_2 = 8 - \frac{64}{3 \pi}, \quad  A_3 = \frac{64}{15 \pi} .\]
}
We will prove using hypergeometric functions that, for all $s \geq 0$,
$A_s$ is a {\em linear} polynomial in $1/\pi$, and derive compact explicit formulas
for it (see \cref{thm:stars} and \cref{rk:more_As} below).
The arguments employed in this section are mostly independent from the rest of the paper;
\revision{the key point is that, unlike for general trees, $S(T_s)$ can be expressed as a ${}_{3}^{}{{}{{}{F_{2}^{}}}}$ hypergeometric function with parameters depending on~$s$}, 
see~\cref{coro:3F2}
(our general approach would lead to a more complicated recursion than the one presented here).
We focus on the evaluation at $t=1/4$ of the sum $S(T_s)$,
since the power series $S(T_s)$ itself does not seem to admit as compact formulas as its evaluation,
see Remark~\ref{rmk:stars_series_pas_si_simple}.
\medskip 

For any two power series $A(t) = \sum_{n \geq 0} a_n t^n$ and $B(t) = \sum_{n \geq 0} b_n t^n$, we denote by $A\odot B$ their Hadamard product 
$(A\odot B) (t) \coloneqq \sum_{n \geq 0} a_n b_n t^n$. The first result directly follows from the definition.

\begin{lemma}\label{lem:Hadamard}
Let $C(t) \coloneqq  \sum_{n \geq 0} \Cat_n t^n$ be the generating function of the Catalan sequence $(\Cat_n)_n$.
Then \revision{$S(T_s)=(C\odot C^s)(t^2)$, and consequently,}
$A_s = (C\odot C^s) (1/16)$ for all $s\geq 0$.
\end{lemma}

It is well-known that $C(t)$ is an algebraic power series satisfying $C(t) =  1 +  t C(t)^2$, and that it is also a hypergeometric \revision{function}, namely 
$C(t) = {}_{2}^{}{{}{{}{F_{1}^{}}}} \left(\frac{1}{2},1;2; 4t \right) .$
The next result shows that all positive integer powers of $C(t)$ are also hypergeometric.

\begin{lemma}\label{lem:HypergeomCats}
For any \revision{$s\geq 0$}, 
\[ 
C(t)^s = {}_{2}^{}{{}{{}{F_{1}^{}}}} \left(\frac{s}{2}, \frac{s+1}{2}; s+1; 4t \right) .
\]
\end{lemma}

\begin{proof}
From $C(t) - t C(t)^2 = 1$ it follows that $t C(t)$ is equal to the compositional inverse of $(t-t^2)$. The Lagrange-Bürmann inversion theorem then implies that the $n$-th coefficient of $C^s$ is equal to $s/(n+s)$ times the $n$-th coefficient of $1/(1-t)^{n+s}$. From this we get that\footnote{A generalization of~\eqref{id:Lambert}, attributed to Lambert~\cite{Lambert1770}, can be found in~\cite[\revision{Eq.~(5.60), p.~201}]{GKP94}.} 
\begin{equation}\label{id:Lambert}
	 C(t)^s = \sum_{n\geq 0} \frac{s}{n+s} \binom{2n+s-1}{n} t^n ,
\end{equation}	 
and the conclusion follows by using the formula $4^n \cdot \left(\frac{s}{2} \right)_n \cdot \left(\frac{s+1}{2} \right)_n = (s)_{2n}$.
\end{proof}

As a consequence of Lemma~\ref{lem:HypergeomCats} and of the obvious identity 
\[
 {}_{2}^{}{{}{{}{F_{1}^{}}}} \left(a_1, 1; c_1; r_1 t \right)
 \odot
 {}_{2}^{}{{}{{}{F_{1}^{}}}} \left(a_2, b_2; c_2; r_2 t \right)
 =
 {}_{3}^{}{{}{{}{F_{2}^{}}}} \left(a_1, a_2, b_2; c_1, c_2; r_1 r_2 t \right),
\]
we deduce that
\[
(C\odot C^s) (t) = {}_{3}^{}{{}{{}{F_{2}^{}}}} \left(\frac{1}{2},\frac{s}{2},\frac{s+1}{2};2,s +1;16 t \right) ,
\]
and combining this with Lemma~\ref{lem:Hadamard}, we get the following nice hypergeometric \revision{expressions for $S(T_s)(t)$ and for $A_s$}.
\begin{corollary} \label{coro:3F2}
For any \revision{$s\geq 0$},
the following equalities hold:
\begin{equation} \label{eq:As=3F2}
\revision{S(T_s)= {}_{3}^{}{{}{{}{F_{2}^{}}}} \left(\frac{1}{2},\frac{s}{2},\frac{s+1}{2};2,s +1; 16 \, t^2\right),}\qquad A_s = {}_{3}^{}{{}{{}{F_{2}^{}}}} \left(\frac{1}{2},\frac{s}{2},\frac{s+1}{2};2,s +1;1 \right) .
\end{equation}
\end{corollary}

We now prove the main result of this section.
\begin{theorem}\label{thm:stars}
For any $s\geq 0$:
\[A_{s + 3}  = 
\frac{64}{\pi} \cdot \left(
\sum_{k=0}^s   {\binom{s}{k}} \cdot \frac{1}{\left(2 k +1\right) \left(2 k +3\right) \left(2 k +5\right)}\right) .
\]	
In particular,  $A_s \sim\frac{8}{\pi} \cdot \frac{2^s}{s^3}$ as $s$ tends to infinity.
\end{theorem}

We observe that, for $s\geq 3$, the sum~$A_{s}$ is not only a linear polynomial in $\Q[1/\pi]$, but it is equal to a rational multiple of $1/\pi$.

\begin{proof}
The starting point is the \href{https://functions.wolfram.com/HypergeometricFunctions/Hypergeometric3F2/07/01/01/0001/}{following identity}~\cite[Eq.~(4.1.2), page 108]{Slater66}, which holds as soon as $\Re(b_2) > \Re(a_3) > 0$:
\[
 {}_{3}^{}{{}{{}{F_{2}^{}}}} \left(a_1, a_2, a_3; b_1, b_2; z \right)
 =
\frac{\Gamma(b_2)}{\Gamma(a_3) \cdot \Gamma(b_2-a_3)}
\cdot \int_0^1 t^{a_3-1} \cdot (1-t)^{b_2-a_3-1} \cdot 
 {}_{2}^{}{{}{{}{F_{1}^{}}}} \left(a_1, a_2; b_1; tz \right) \, dt
. 
\]
Choosing $(a_1, a_2, a_3) = (s/2, (s+1)/2, 1/2)$, $(b_1, b_2) = (s+1, 2)$ and specializing at $z=1$ yields
\[
A_s
 =
\frac{\Gamma(2)}{\Gamma(1/2) \cdot \Gamma(3/2)}
\cdot \int_0^1 t^{-1/2} \cdot (1-t)^{1/2} \cdot 
 {}_{2}^{}{{}{{}{F_{1}^{}}}} \left(s/2, (s+1)/2; s+1; t \right) \, dt .
\]
The prefactor evaluates as $\frac{\Gamma(2)}{\Gamma(1/2) \cdot \Gamma(3/2)}=\frac2{\pi}$.
By Lemma~\ref{lem:HypergeomCats}, we have
\[{}_{2}^{}{{}{{}{F_{1}^{}}}} \left(s/2, (s+1)/2; s+1; t \right)= C(t)^s=\left( \frac2{1 + \sqrt{1-t}}\right)^{s}.\] 
From this, we get
\[
A_s
 =
\frac{2}{\pi}
\cdot \int_0^1 \sqrt{\frac{1-t}{t}} \cdot 
\left( \frac{1 + \sqrt{1-t}}{2}\right)^{-s}
 \, dt .
\]
The successive changes of variables $t=4u - 4u^2$, $u = 1/v$ and $w=v-1$ provide the alternative representations
\[
A_s
 =
\frac{4}{\pi}
\cdot \int_{\frac12}^1 \frac{(2u-1)^2}{\sqrt{u (1-u)}} \cdot u^{-s} \, du 
 =
\frac{4}{\pi}
\cdot \int_{1}^2 \frac{(2-v)^2}{\sqrt{v-1}} \cdot v^{s-3} \, dv 
 =
\frac{4}{\pi}
\cdot \int_{0}^1 \frac{(1-w)^2}{\sqrt{w}} \cdot (w+1)^{s-3} \, dw .
\]
Then the binomial formula gives us
\[A_{s+3}=\frac{4}{\pi} \cdot \left(
\sum_{k=0}^s   {\binom{s}{k}} \cdot \int_{0}^1 \frac{(1-w)^2}{\sqrt{w}} \cdot (w)^{k} \, dw \right).\]
Recalling the standard evaluation of beta integrals $\int_0^1 x^{\alpha-1} (1-x)^{\beta-1}dx=\frac{\Gamma(\alpha)\Gamma(\beta)}{\Gamma(\alpha+\beta)}$, we get
\[
\int_{0}^{1} \left(1-w \right)^{2} w^{k -\frac{1}{2}} \, dw
 = 
\frac{16}{\left(2 k +1\right) \left(2 k +3\right) \left(2 k +5\right)},
\]
which concludes the proof of the exact formula.

In order to get the asymptotic behavior, observe that, by using Stirling's formula
or standard deviation estimates, one has, as $s \rightarrow \infty$:
\begin{align*}
\sum_{k=0}^s \binom{s}{k} \cdot \frac{1}{(2k+1)(2k+3)(2k+5)} &= (1+o(1)) \sum_{|k-s/2| \leq s^{3/4}} \binom{s}{k} \cdot \frac{1}{(2k+1)(2k+3)(2k+5)}\\
&= (1+o(1)) \frac{1}{s^3} \sum_{|k-s/2| \leq s^{3/4}} \binom{s}{k}
 = (1+o(1)) \frac{2^s}{s^3}.
\end{align*}
This proves that
\begin{align*}
A_s &\underset{s \rightarrow \infty}{\sim} \frac{64}{\pi} \frac{2^{s-3}}{(s-3)^3} 
\underset{s \rightarrow \infty}{\sim} \frac{8}{\pi} \frac{2^s}{s^3}.\qedhere
\end{align*}
\end{proof}

\begin{remark}\label{rk:more_As}
From the moment representation \revision{(obtained in the proof of~\cref{thm:stars})}
\[
A_{s}
 =
\frac{4}{\pi}
\cdot \int_{1}^2 \frac{(2-v)^2}{\sqrt{v-1}} \cdot v^{s-3} \, dv ,
\]	
one can deduce \revision{(for instance, by using integration by parts)} that the sequence $(A_s)_{s \geq 0}$ satisfies the following recurrence relation, valid for \revision{all $s\geq 0$}:
\begin{equation}\label{eq:rec_As_1}
(2s+1)(s^2-s+1) A_{s+1} - 2(s-2)(s^2+s+1) A_s = \frac{16}{\pi} \cdot 2^s .
\end{equation}
\revision{From this first-order inhomogeneous recurrence, one deduces the following second-order homogeneous recurrence, also valid for $s\geq 0$: 
}
\begin{equation}\label{eq:rec_As_2}
  (2 s + 3) A_{s+2}
  - 2 (3 s - 2) A_{s+1} 
 + 4 (s - 2) A_s 
 = 0.
 \end{equation}	
\revision{
Using Petkov\v{s}ek's algorithm ``Hyper'' in~\cite[\S4 and \S5.2]{Petkovsek92}, one can prove that the recurrence equation~\eqref{eq:rec_As_1} does not admit any \emph{hypergeometric} solutions (these are sequences satisfying first-order homogeneous recurrences with polynomial coefficients); in particular, the sequence 
$(A_s)_{s}$ is not hypergeometric.
On the other hand, the same algorithm shows that the recurrence equation \eqref{eq:rec_As_2} admits a one-dimensional subspace of hypergeometric solutions, spanned by 
 $\frac{4^{s} \left(s^{2}-s +1\right)}{{\binom{s}{3}} {\binom{2 s -1}{s -1}}}$ (valid for $s \geq 3$ only).
As a consequence, this implies that our particular sequence $(A_s)$, solution of~\eqref{eq:rec_As_2}, is \emph{d'Alembertian}, that is, 
expressible as nested indefinite sums of hypergeometric sequences.
Indeed, the Abramov-Petkov\v{s}ek algorithm ``A'' in~\cite[\S3]{AbPe94} allows us to deduce the following alternative closed expression for $A_s$ ($s \geq 3$), which explicitly exhibits its d'Alembertian nature\footnote{\revision{Note that the expression in \cref{thm:stars} is \emph{not} d'Alembertian, since the summand contains the summation bound $s$ in the binomial coefficient, which renders this sum definite rather than indefinite.}}:
\begin{equation}\label{eq:closed_As_3}
A_s = \frac{16}{3 \pi} \cdot \frac{4^s (s^2-s+1)}{\binom{s}{3} \binom{2s}{s}} \cdot 
\sum_{\ell = 2}^{s-1} \frac{\binom{\ell+1}{2} \binom{2\ell-1}{\ell-2}}{2^{\ell} ( \ell^4+\ell^2+1)}.
\end{equation}
} 
\end{remark}

\begin{remark}\label{rmk:stars_series_pas_si_simple}
Computer algebra tools, notably creative telescoping \revision{techniques}, can be used to prove that
\[
u_s \coloneqq \revision{S(T_s) =}\, {}_{3}^{}{{}{{}{F_{2}^{}}}} \left(\frac{1}{2},\frac{s}{2},\frac{s+1}{2};2,s +1; z \right)
 =
\frac{2}{\pi}
\cdot \int_0^1 \sqrt{\frac{1-t}{t}} \cdot 
\left( \frac{1 + \sqrt{1-tz}}{2}\right)^{-s}
 \, dt 
\]
\revision{satisfies the homogeneous linear recurrence of order three}
\[
z \left(2 s +5\right) \left(s +2\right) u_{s +3}
-2 \left(s +1\right) \big( (s+3) z +4 s +8\big) u_{s +2}
+16 \left( s^{2}+ 2 s -2\right) u_{s +1}
-8 \left(s +3\right) \left(s -2\right) u_s = 0 
\]
\revision{and the inhomogeneous linear recurrence of order two}
\begin{equation*}
\begin{split}
4 \left(s +3\right) \left(2-s \right) u _s
+4 \left(s +1\right) \left(s +2\right) u_{s +1}
-z \left(s +1\right) \left(s +3\right) u_{s +2}
  =  \\
\revision{
{8 \left(s + 4\right)}
\cdot  {}_{4}^{}{{}{{}{F_{3}^{}}}}  
\left(
-\frac{1}{2}, \frac{s}{2}+1, \frac{s+1}{2}, - \frac{s^{2}+5 s +10} {s^{2}+5 s -2}; 2, s +2, -\frac{2 (s+1)(s+4)}{s^{2}+5 s -2}; z \right) .
}
\end{split}
\end{equation*}
Using the algorithms in~\cite{Petkovsek92,AbPe94}, \revision{it can be shown that these recurrences possess neither hypergeometric nor even d'Alembertian solutions
for parametric~$z$.}
In particular, no formula 
\revision{similar to Eq.~\eqref{eq:closed_As_3}}
can exist for \revision{$S(T_s)(t) = {}_{3}^{}{{}{{}{F_{2}^{}}}} \left(\frac{1}{2},\frac{s}{2},\frac{s+1}{2};2,s +1; 16 t^2 \right)$.
We believe as well that there is no formula for $S(T_{s+3})(t)$ that generalizes the one for $A_{s+3} = S(T_{s+3})(1/4)$ in~\cref{thm:stars}; however, we do not yet have a proof of this fact.
}
\end{remark}

\begin{remark}\label{rmk:stars_series_still_linear_in_H1_and_H2}
\revision{
By \cref{thm:stars}, $A_s=S(T_s)(\tfrac14)$ is a \emph{linear} polynomial in $\Q[1/\pi]$ for all $s \geq 3$.
A natural question is whether this result admits a functional lifting, i.e., if the power series $S(T_s)$ also is a \emph{linear} polynomial in $\Q[t,t^{-1}][H_1, H_2]$ for all $s \geq 3$,
where, as before,
\[
H_1(t) \coloneqq {}_{2}^{}{{}{{}{F_{1}^{}}}}\big(-\tfrac12,-\tfrac12; 1;16t^2\big)
\quad \text{and} \quad
H_2(t) \coloneqq {}_{2}^{}{{}{{}{F_{1}^{}}}}\big(-\tfrac12,\tfrac12; 2;16t^2\big).
\]
Indeed, using an approach similar to that in~\cref{rmk:stars_series_pas_si_simple}, one 
can prove that for all $s \geq 0$ we have:
\begin{equation}\label{eq:stars_series_still_linear_in_H1_and_H2}
S(T_{s+3}) =
B_s(t) \cdot H_1(t)
-
C_s(t) \cdot H_2(t) ,
\end{equation}
 where the cofactors $B_s(t)$ and $C_s(t)$ are Laurent polynomials in $t$, both satisfying the same linear recurrence relation of order three, also satisfied by $Y_s(t) = S(T_{s+3})$:
\begin{equation} \label{eq:Bs_Cs_Ys}
\begin{split}
2 t^{2} \left(2 s +11\right) \left(s +5\right) Y_{s+3}(t)
-\left(s +4\right) \left(4 \,t^{2} ( s+6) + s +5\right) Y_{s+2}(t) \\
+ 2 \left(s^{2}+8 s +13\right) Y_{s+1}(t) 
-\left(s +1\right) \left(s +6\right) Y_s(t)
= 0,
\end{split}
\end{equation}
with 
 \(
B_0 = \frac{1}{20 t^{4}}, 
B_1 = \frac{2 t^{2}+1}{140 t^{6}}, 
B_2 = \frac{10 t^{4}+1}{840 t^{8}}
\) and
\(
C_0 = \frac{6 t^{2}+1}{20 t^{4}}, 
C_1 = \frac{32 t^{4}+8 t^{2}+1}{140 t^{6}}, 
C_2 = \frac{160 t^{6}+30 t^{4}+6 t^{2}+1}{840 t^{8}}
\).
For $s=0, \ldots, 3$, 
expression~\eqref{eq:stars_series_still_linear_in_H1_and_H2} coincides with that in
\S\ref{sec:star_s=3}, 
\S\ref{sec:star_s=4}, 
\S\ref{sec:star_s=5}, 
and
\S\ref{sec:star_s=6}.
}

\revision
{In the absence of a closed-form expression for $S(T_{s+3})$ extending that of \cref{thm:stars} or that of~\cref{rk:more_As}, the recurrence relation~\eqref{eq:Bs_Cs_Ys} allows the cofactors of the linear expression~\eqref{eq:stars_series_still_linear_in_H1_and_H2} to be determined explicitly (and rapidly) for any $s \geq 0$.
}

\revision{
Note that for $t=\frac14$, Eq.~\eqref{eq:stars_series_still_linear_in_H1_and_H2} provides an expression 
$A_s=S(T_s)(\tfrac14) = \frac{4}{\pi} B_s(\tfrac14) -  \frac{8}{3\pi} C_s(\tfrac14)$, 
where the cofactors $B_s(\tfrac14)$ and $C_s(\tfrac14)$ are again d'Alembertian:
\[
B_s(\tfrac14)
=
\frac{ \left(5 s +11\right) \cdot 2^{s +1}}{2 s +5}
+
\frac{4 \cdot 4^s \cdot \left(s^2+5 s +7\right)}{\left(s +1\right) \left(s +2\right) \left(s +3\right) \binom{2 s +5}{s +2}} \cdot \sum_{k=0}^s \frac{\left(k +1\right) \left(4 k +9\right) \binom{2 k +2}{k +1}}{2^k}
\]
and
\[
C_s(\tfrac14) =
 \frac{3 \cdot \left(s -1\right) \cdot  2^{s}}{2 s +5}
 + \frac{6 \cdot 4^s \cdot  \left(s^2 +5 s +7\right)}{\left(s +1\right) \left(s +2\right) \left(s +3\right) \binom{2 s +5}{s +2}}
 \cdot  \sum_{k=0}^s \frac{\left(k +1\right) \left(4 k +13\right) \binom{2 k +2}{k +1}}{2^k}
\,  .
 \]
 While still linear 
in $1/\pi$, this alternative d'Alembertian expression for $A_s$ is slightly more involved  
than~\eqref{eq:closed_As_3}.
}

\revision{
As a final remark, although we do not have a nested sum expression for the sequence $ (S(T_{s+3})(t))_s$, we do have a nested integral expression for its generating function 
\(\sum_{s \geq 0} S(T_{s+3})(t) z^s \), namely
\[
\frac{\left(1-2 z\right) \sqrt{\frac{z-z^{2}-4 t^{2}}{1-z}}}{2 t^{2} z^{{5}/{2}}}
\, \left(H_1 \cdot {\int}\frac{\left(\frac{\left(1-z\right) z}{ z-z^{2}-4t^{2}}\right)^{{3}/{2}}}{\left(1-2 z\right)^{2}} {d}z + H_2 \cdot {\int}\frac{\left((1-z)^2-(8 z -6) t^{2}\right) \left(\frac{z}{z-z^{2}-4 t^{2}}\right)^{{3}/{2}}}{\left(1-2 z\right)^{2} \sqrt{1-z }}{d}z \right) .
\]
}
\end{remark}	

\appendix
\crefalias{section}{appendix}

\section{Algebraic independence of the generators}
\label{appendix:alg_independence}

We prove in an elementary way the algebraic independence of $H_1(t)$ and $H_2(t)$, or, equivalently, of the complete elliptic integrals of first and second kinds, $K(t)$ and $E(t)$.
It may well be that this result is classical; however we could not locate it in the extremely rich literature on special functions.
Our proof is based on asymptotic arguments and on the transcendence of $\pi$.
(In the footnote~\footref{fn1} we sketched another proof relying on a much harder result in transcendence theory.)

We consider the functions
\[f_1(t) \coloneqq \frac{2K(\sqrt t)}{\pi} = {}_{2}^{}{{}{{}{F_{1}^{}}}}\big(\tfrac12, \tfrac12; 1; t\big),
\qquad f_2(t) \coloneqq \frac{2E(\sqrt t)}{\pi} = {}_{2}^{}{{}{{}{F_{1}^{}}}}\big(-\tfrac12, \tfrac12; 1; t\big).\]
We first study the asymptotic behavior of the coefficients of $f_1$ and $f_2$,
and the behavior of $f_1(t)$ and $f_2(t)$ near $t=1$. 
We observe via Stirling's formula that $f_1(t)$ expands as
$f_1(t)=\sum_{n \ge 0} c_{n,1} t^n$ with $c_{n,1} = \binom{2n}{n} \cdot 16^{-n}  \sim 1/(\pi n)$.
This yields the estimate $f_1(t) \sim |\log(1-t)|/\pi $ as $t$ tends to $1$ from below.
On the other hand, the Gauss summation formula~\eqref{eq:values_elliptic} gives
\[f_2(1)=\frac{\Gamma(1)^2}{\Gamma(\tfrac12) \Gamma(\tfrac32)}=\frac{2}{\pi}.\]
Now, for the sake of contradiction, we assume that there exists a nonzero polynomial  with rational coefficients such that $P(t,f_1(t),f_2(t))=0$. 
We expand $P(t,x,y)$ in the form $\sum_{i,j \ge 0} (t-1)^i x^j P_{i,j}(y)$,
and consider the nonzero term $(t-1)^{i_0} x^{j_0} P_{i_0,j_0}(y)$
which is asymptotically the largest around $t=1$ after substituting $x=f_1(t)$ and $y=f_2(t)$.
Concretely, $(-i_0,j_0)$ is the maximum of $\{(-i,j) : P_{i,j}\ne 0\}$
 in lexicographic order.
 Looking at the asymptotic behavior of $P(t,f_1(t),f_2(t))$ as $t$ tends to $1$
from below, we have 
\[ P(t,f_1(t),f_2(t))\, = \, (P_{i_0,j_0}(\tfrac2{\pi})+o(1))\cdot (t-1)^{i_0} \cdot ( |\log(1-t)|^{j_0}/\pi^{j_0}).\]
Since this should be equal to $0$, it forces $P_{i_0,j_0}(\tfrac2{\pi})=0$.
But $\pi$ is transcendental, so that we must have $P_{i_0,j_0} = 0$.
We have reached a contradiction, proving the algebraic independence of $f_1(t)$ and $f_2(t)$ over $\mathbb Q(t)$.
By \cite[Lemma 7.2]{ABD19} (whose proof can be found in Section 88, pages 133-134, of Mahler's book~\cite{Mahler76}), they are also algebraically independent over $\mathbb C(t)$, since their coefficients lie in~$\Q$.
Thus $E(t)$ and $K(t)$ are
algebraically independent over $\mathbb C(t)$,
and, using \eqref{eq:H12_EK}, it is also the case for $H_1(t)$ and~$H_2(t)$.
\hfill $\square$

\begin{remark}
Yet a different proof, of a more arithmetic flavor, could be based on the results of~\cite{ABD19}. 
This proof still uses asymptotics, but does not rely on the transcendence of $\pi$; it uses instead properties (``à la Lucas'') modulo primes of the coefficients of 
$f_1(t) 
= 1+4 t +36 t^{2}+\cdots$ and of 
$f_2(t) 
= 1-4 t -12 t^{2}-80 t^{3}-\cdots$.
Note that while the coefficient sequence $\binom{2n}{n}^2$ of $f_1(t)$ is $p$-Lucas for all primes~$p\in\mathcal{P}$, this property does not hold for the coefficient sequence $\binom{2n}{n}^2/(1-2n)$ of $f_2(t)$, i.e. $f_2 \notin{\mathfrak{L}(\mathcal{P})}$. Hence, one cannot conclude as in~\cite[Theorem~7.7]{ABD19}. However, one can show that $f_2(t)$ belongs to the larger set~${\mathcal{L}(\mathcal{P})}$ (that is, $f_2(t)/f_2(t^p)\bmod p$ is a rational function in $\mathbb{F}_p(t)$ for all~$p\in\mathcal{P}$), so that~\cite[Proposition~7.6]{ABD19} can be applied.
\end{remark}

\section{The algorithm}
\label{appendix:algo}

\revision{
We describe here, in pseudo-code, an algorithm to compute the series $S(T)$ associated to a decorated tree $T$.
This is the algorithmic translation of our proof of Theorem~\ref{thm:sum-pi-generalized}.\\
{\sc Input:} a (rooted) decorated tree $T$;\\
{\sc Output:} the associated series $S(T)$ defined in~\eqref{ST_with_vertex_variables}.}
\medskip

We denote by $H(T)$ the height of $T$, and by $\rho$ its root. The algorithm performs the following steps one by one in order. Implicitly, every time a sum associated to a tree with strictly smaller total path length than $T$ (including trees with strictly less vertices than $T$) appears, perform the algorithm on this smaller tree, starting from step~1) below.
The results proved in the paper show that this algorithm is well-defined and always terminates.

\emph{1) Tree of height $0$}

If $H(T)=0$, then, letting $T=\{\rho\}$:
\begin{itemize}
\item If $\rho$ is gray decorated with $(\bowtie, K)$, return $\One[0 \bowtie K]$.
\item If $\rho$ is white, then:
\begin{itemize}
\item if $\Delta_\rho=(\varnothing,K)$, return $C(t)$;
\item if $\Delta_\rho=(=,K)$, return $\Cat_K t^K$ if $K \geq 0$, return $0$ otherwise;
\item if $\Delta_\rho=(\geq,K)$, return $C(t)- \sum_{\ell=0}^{K-1} \Cat_\ell t^\ell$;
\item if $\Delta_\rho=(\leq,K)$, return $\sum_{\ell=0}^{K} \Cat_\ell t^\ell$.
\end{itemize}
\item If $\rho$ is black, then:
\begin{itemize}
\item if $\Delta_\rho=(\varnothing,K)$, return $C(t)$;
\item if $\Delta_\rho=(=,K)$, return $\Cat_{-K} t^{-K}$ if $K \le 0$, return $0$ otherwise;
\item if $\Delta_\rho=(\geq,K)$, return $\sum_{\ell=0}^{-K} \Cat_\ell t^\ell$ if $K \leq 0$, return $0$ otherwise;
\item if $\Delta_\rho=(\leq,K)$, return $C(t)- \sum_{\ell=0}^{-K-1} \Cat_\ell t^\ell$.
\end{itemize}
\end{itemize} 

\emph{2) Make $T$ a good tree}

Now let $T$ be a decorated tree of height $H(T) \geq 1$.
\begin{itemize}
\item[2(a)] If there exists $v \in T$ nonroot with decoration $(=,K)$ for some $K \in \Z$, use  Lemma~\ref{lem:equalitylabel} and (with the notation of Lemma~\ref{lem:equalitylabel}), return $S(U) \cdot S(V)$. Repeat this step until there is no such vertex in $T$.
\item[2(b)] If there exists a vertex $v$ with decoration $(\bowtie,K)$ where $\bowtie \, \in \{ \leq,\geq\}$ and $K \in \Z \backslash \{0\}$, use Lemma~\ref{lem:shift_k} $|K|$ times in the direction that makes $|K|$ decrease, to write $S(T)=S(T')+A$, where $T'$ is the tree $T$ with the decoration of $v$ being now $(\bowtie,0)$, and $A$ is a sum of sums associated to trees with a leaf having decoration $(=,K)$. On the trees appearing in $A$, apply Step 2(a). 

Repeat this step on $T'$, until there is no such vertex in $T$.
\item[2(c)] If there exists a gray leaf $v$ with decoration $(\bowtie,K)$, return $\mathds{1}[0 \bowtie K] \cdot S(T')$. Here $T'$ is $T \backslash \{v\}$, in which the parent of $v$ in $T$ gets $+K$ to its own decoration (Lemma~\ref{lem:leaf}, last inequality). Repeat this step on $T'$, until there is no gray leaf in $T$.
\item[2(d)] If there exists a leaf $v$ with decoration $(\bowtie,0)$ where $\bowtie \, \in \{ \leq,\geq\}$, use Lemma~\ref{lem:leaf} to write $S(T)=S(T')$, where in $T'$, $v$ has decoration either $(=,0)$ or $(\varnothing,0)$. If it has decoration $(=,0)$, apply Step 2(a) to $T'$. Otherwise, repeat this step on $T'$ until there is no such leaf in $T$.
\item[2(e)] If there exists a leaf $v$ with decoration $(\varnothing,K)$ for some $K \in \Z$, use the first part of Lemma~\ref{lem:varnothinglabel} to modify it in a leaf with decoration $\varnothing_0$. Repeat this step on the resulting tree, until there is no such leaf in $T$.
\smallskip

Now all leaves of $T$ have decoration $\varnothing_0$.

\item[2(f)] If there exist two leaves with the same color and the same parent, then use  Lemma~\ref{lem:twin_varnothing} to write $S(T)=t^{-1} S(T')-t^{-1} S(T'')$, where $T',T''$ have strictly less vertices than $T$.
\item[2(g)] If there exists a leaf with the same color as its parent, then use  Lemma~\ref{lem:consecutive_varnothing} to write $S(T)=t^{-1} S(T')-t^{-1} S(T'')$, where $T',T''$ have strictly less vertices than $T$.
\item[2(h)] If there exists a leaf with a gray parent, then by Lemma~\ref{lem:varnothinglabel} (second part), we have $S(T)=S(T')$ where $T'$ has strictly less vertices than $T$.
\end{itemize}

We now handle the case where $T$ is a good tree.

\emph{3) Good tree of height $1$}

\begin{itemize}
\item If $H(T)=1$, then, since $T$ is a good tree, $T$ is made of two vertices, one black and one white, and the nonroot vertex is decorated with $\varnothing_0$ (see Lemma~\ref{lem:fringe_subtrees} (i)). Using  \cref{prop:basecase}, we get $S(T)$ explicitly.
\end{itemize}

Let us consider the case where $T$ is a good tree with height $\geq 2$. We let $v$ be any vertex of $T$ such that the fringe subtree $T_v$ rooted at $v$ has height \textbf{exactly} $2$.

\emph{4) Modify the fringe subtree $T_v$ and get rid of $=$/$\varnothing$ decoration at $v$.}
\begin{itemize}
\item[4(a)] if $v$ is decorated with $(\varnothing,K)$ for some $K$ in $\mathbb Z$,
use Lemma~\ref{lem:root_varnothing} or Lemma~\ref{lem:varnothinglabel} 
(depending on whether $v$ is the root or not) to compute $S(T)$.
\item[4(b)] Let $w$ be a black child of $v$ with at least one child itself. Then we have ensured that $w$ has exactly one child $z$ which is a white leaf. We have $S(T)=S(T')$, where $T'$ is obtained from $T$ by swapping the colors of $w$ and $z$. 
Perform this step on $T_v \subseteq T'$ until there is no more such a vertex $w$.
\item[4(c)] If $v$ is decorated with $(=,K)$ for some $K$, then $v$ is necessarily the root of the tree and $T$ is of the form $U_{i,j,k}^{(=,K)}$, $\bar U_{i,j,k}^{(=,K)}$ or $V_{i,j,k}^{(=,K)}$. 
  \begin{itemize}
    \item[4(c)1.] Assume $T=V_{i,j,k}^{(=,K)}$, 
      If $k \ge 2$, we use \eqref{eq:V1}, we have $S(T)=t^{-2}(S_1(t)+S_2(t)+S_3(t)+S_4(t))$, where $S_1,S_2,S_3,S_4$ are associated with trees with strictly less vertices than $T$. If $k=1$, we use \eqref{eq:V4} to express $S(T)$ in terms of $S(V_{i,j,0}^{(=,K)})$. If $k=0$ and $i,j \ne 0$, we use Lemma~\ref{lem:linear_system}. Finally if $k=0$ and $i$ or $j$ is equal to $0$, the condition on summation variables can be rewritten as a disjoint union of conjunctions of simple equalities, reducing the sum to products of sums indexed by two-vertex trees (see the proof of \cref{prop:UV} for details).
    \item[4(c)2.] Assume $T=U_{i,j,k}^{(=,K)}$ or $T=\bar U_{i,j,k}^{(=,K)}$. If $k \ge 1$, we simply apply \eqref{eq:reduc_U_k}. If $k=0$, but $j \ne 0$, we apply \eqref{eq:reduc_U_ij} to reduce this case to $k=j=0$. Finally if $k=j=0$, the condition on summation variables allows to reduce the sum as above.
  \end{itemize}
\end{itemize}

We now consider cases according to the color of $v$.

\emph{5) If $v$ is gray}

In this case, $T_v=V_{i,j,k}^{\Delta}$ for some $i,j,k \geq 0$ and some decoration $\Delta$.
\begin{itemize}
\item[5(a)] If $k\geq 2$, we can write by \eqref{eq:V1} that $S(T)=t^{-2}(S_1(t)+S_2(t)+S_3(t)+S_4(t))$, where $S_1,S_2,S_3,S_4$ are associated with trees with strictly less vertices than $T$.
\item[5(b)] If $k=1$, then \eqref{eq:V4} allows us to write $S(T)$ as a sum of sums involving either trees with less vertices than $T$, or trees $V_{i,j,0}^{\Delta}$ for some $i,j \geq 0$, on which we perform the next steps of the algorithm.
\item[5(c)] If $(i,j,k)=(d,0,0)$ or $(0,d,0)$ for some $d \geq 1$, then:
\begin{itemize}
\item[5(c).1] If $\Delta=(\geq,0)$, use \eqref{eq:V2} to write $S(T)=S(T')$ where $T'$ has an internal vertex $v$ with decoration $(\varnothing,0)$ and go to Step 4(a)).
\item[5(c).2] If $\Delta=(\leq,0)$, \eqref{eq:V3} allows us to compute $S(T)$;
\end{itemize}
\item[5(d)] If $i,j\geq 1$, then we use Lemma~\ref{lem:linear_system} to write $S(T)$ as a sum of sums associated to trees that are either with less vertices than $T$; or $V^{\Delta}_{d,0,0}$ and $V^{\Delta}_{0,d,0}$ on which we apply the algorithm started at Step 5(c); or trees $T'$ with the same number of vertices and total path length than $T$, obtained from $T$ by replacing $T_v$ by $V^{\Delta}_{a,b,2}$ where $a+b=i+j-2$. We can apply our algorithm starting at Step 5(a) to these trees.
\end{itemize}

\emph{6) If $v$ is black or white}

We describe here only the case where $v$ is white. The case where $v$ is black can be easily deduced by symmetry.
\begin{itemize}
\item[6(a)] If all subtrees rooted at children of $v$ are of size $2$ except for one black leaf, then, by the second part of Lemma~\ref{lem:varnothinglabel} in the reverse direction, $S(T)=S(T')$ where $T'_v=V^{\Delta}_{i,j,k}$ for some $i,j,k \geq 0$. Apply the algorithm to $T'$, started at Step 5(a).
\item[6(b)] If all subtrees rooted at children of $v$ are of size $2$, then $T_v=U_{i,j,k}^{\Delta}$ for some $i,j,k \geq 0$:
\begin{itemize}
\item[6(b).1] If $k \geq 1$, then \eqref{eq:reduc_U_k} allows us to write $S(T)=t^{-1}S(T')-t^{-1}S(T'')$, where~$T'$ has less vertices than $T$ and $T''=V_{i,j,k}^{\Delta+1}$. 
We apply our algorithm to $T'$ and~$T''$, starting at Step 2(e) (to get rid of the $+1$ shift).
\item[6(b).2] If $(i,j,k)=(d,0,0)$, then:
\begin{itemize}
\item If $\Delta=(\geq,0)$, use \eqref{eq:U1} to write $S(T)=S(T')$ where $T'$ has an internal vertex with decoration $\varnothing_0$ and go to step 4(a).
\item If $\Delta=(\leq,0)$, use \eqref{eq:U2} to write $S(T)$ as a product of sums associated to trees with less vertices than $T$.
\end{itemize}
\item[6(b).3] If $k=0$ and $1 \leq j \leq d$, then we use \eqref{eq:reduc_U_ij} to write $S(T)$ as the sum of sums associated to either $U_{i+j,0,0}$, on which we apply the algorithm started at Step 6(b).2; or trees with less vertices than $T$ ; or trees with $k=1$, on which we apply the algorithm started at Step 6(b).1.
\end{itemize} 
\end{itemize}

\section{Running the algorithm on a non-trivial example}
\label{appendix:running_algo}

We show here how our algorithm works on an example, namely the tree $T$ on the left of \cref{fig:octopus_tree} (this is the tree $T_{6,c}$ in \cref{fig:trees2-6} and in \cref{ssec:T6c}).

\begin{figure}[!ht]
\begin{tabular}{c c c}
\begin{tikzpicture}
\draw[white] (-2,0) -- (2,0);
\draw[white] (0,2) -- (0,2.8);
\draw (0,0) -- (0,1) -- (0,2) -- (-1,1) (0,2) -- (1,1) -- (1,0);
\draw[fill] (0,0) circle(.1);
\draw[fill] (0,1) circle(.1);
\draw[fill] (0,2) circle(.1);
\draw[fill] (1,1) circle(.1);
\draw[fill] (-1,1) circle(.1);
\draw[fill] (1,0) circle(.1);
\end{tikzpicture}
&
\begin{tikzpicture}[every node/.style={scale=.8}]
\draw[white] (-2,0) -- (2,0);
\draw[white] (0,2) -- (0,2.8);
\draw (0,0) -- (0,1) -- (0,2) -- (-1,1) (0,2) -- (1,1) -- (1,0);
\draw[fill=white] (0,0) circle(.15);
\draw[fill] (0,1) circle(.15);
\draw[fill=white] (0,2) circle(.15);
\draw[fill] (1,1) circle(.15);
\draw[fill] (-1,1) circle(.15);
\draw[fill=white] (1,0) circle(.15);
\draw (.6,2) node{$(=,0)$};
\draw (.4,1) node{$\leq_0$};
\draw (1.4,1) node{$\leq_0$};
\draw (-.6,1) node{$\varnothing_0$};
\draw (1.4,0) node{$\varnothing_0$};
\draw (.4,0) node{$\varnothing_0$};
\end{tikzpicture}
&
\begin{tikzpicture}[scale=.4]
\draw (-1,0) -- (11,0);
\draw (10,0) arc(0:180:.5);
\draw (10,0) arc(0:-180:1.5);
\draw (9,0) arc(0:-180:.5);
\draw (8,0) arc(0:180:3.5);
\draw (7,0) arc(0:180:1.5);
\draw (6,0) arc(0:-180:1.5);
\draw (6,0) arc(0:180:.5);
\draw (5,0) arc(0:-180:.5);
\draw (2,0) arc(0:-180:.5);
\draw (3,0) arc(0:180:.5);
\end{tikzpicture}
\end{tabular}
\caption{Left: a tree $T$ with $6$ vertices. Middle: the associated decorated tree. Right: a meander whose interior face corresponds to the sum $S(T)(\frac{1}{4})$.}
\label{fig:octopus_tree}
\end{figure}
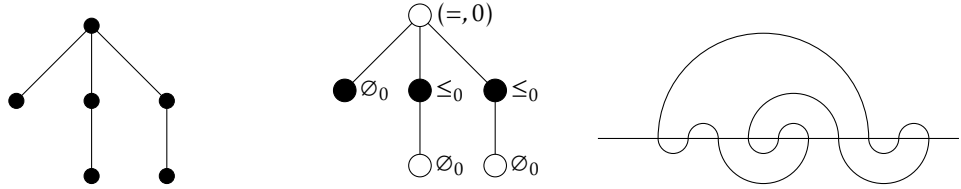

We first build the decorated tree $T'$ associated to $T$, obtained by the construction described in Section \ref{sec:change_variables} by taking as the root its topmost vertex. This decorated tree is depicted in \cref{fig:octopus_tree}, middle. The sum $S(T) \coloneqq  S(T)(\frac{1}{4})$ corresponds e.g. to the interior face of the meander drawn on \cref{fig:octopus_tree}, right, in the sense of Section \ref{ssec:meanders_and_trees}. 

Our goal is to compute the sum $S(T)=S(T')$.
Using the second part of Lemma~\ref{lem:varnothinglabel} in the unusual direction to \enquote{pull down} the black leaf of $T'$, we get that $S(T')=S(V^{(=,0)}_{2,0,1})$, where $V^{\Delta}_{i,j,k}$ is defined in \cref{sec:long_stars}. As a first step, we can write:

\begin{align*}
S\left( V_{2,0,1}^{(=,0)} \right) &= S\left( V_{3,0,0}^{(=,0)} \right) + S\left( V_{2,1,0}^{(=,0)} \right) - S\left( V_{2,0,0}^{(=,0)} \right) \cdot S_{=,0} \text{ by } \eqref{eq:V4}.\\
\end{align*}

We can immediately compute $S\left( V_{2,0,0}^{(=,0)} \right) = S_{=,0}^2$. On the other hand, we obtain from Lemma~\ref{lem:linear_system} that

\begin{align*}
\begin{pmatrix}
S(V_{1,2,0}^{(=,0)}) \\
S(V_{2,1,0}^{(=,0)})
\end{pmatrix}
&= \begin{pmatrix}
2 & 1 \\ 1 & 2
\end{pmatrix}^{-1}
\left[ \begin{pmatrix}
X_1 \\ X_2
\end{pmatrix} - \begin{pmatrix}
S(V_{0,3,0}^{(=,0)}) \\ S(V_{3,0,0}^{(=,0)})
\end{pmatrix} \right]= \frac{1}{3} \begin{pmatrix}
2 & -1 \\ -1 & 2
\end{pmatrix} \left[ \begin{pmatrix}
X_1 \\ X_2
\end{pmatrix} - \begin{pmatrix}
S(V_{0,3,0}^{(=,0)}) \\ S(V_{3,0,0}^{(=,0)})
\end{pmatrix} \right],
\end{align*}
where $X_1,X_2$ are defined in Lemma~\ref{lem:linear_system}, so that
\begin{align*}
S\left( V_{2,1,0}^{(=,0)} \right) &=  -\frac{1}{3}\left( X_1 -S\left( V_{0,3,0}^{(=,0)} \right) \right) + \frac{2}{3} \left( X_2 -S\left( V_{3,0,0}^{(=,0)} \right) \right)\\ 
&= -\frac{1}{3} \left( S\left( V_{0,1,2}^{(=,0)} \right) + 2 S\left( V_{0,1,1}^{(=,0)} \right) \cdot S_{=,0} + S\left( V_{0,1,0}^{(=,0)} \right) \cdot (S_{=,0})^2 -S\left( V_{0,3,0}^{(=,0)} \right) \right) \\
&+ \frac{2}{3} \left( S\left( V_{1,0,2}^{(=,0)} \right) + 2 S\left( V_{1,0,1}^{(=,0)} \right) \cdot S_{=,0} + S\left( V_{1,0,0}^{(=,0)} \right) \cdot (S_{=,0})^2 -S\left( V_{3,0,0}^{(=,0)} \right) \right)
\end{align*}\smallskip
 
Since $S\left( V_{i,j,k}^{(=,0)} \right) = S\left( V_{j,i,k}^{(=,0)} \right)$ for all $i,j,k$, we get
\begin{align*}
S\left( V_{3,0,0}^{(=,0)} \right) + S\left( V_{2,1,0}^{(=,0)} \right) &= \frac{1}{3} \left( S\left( V_{0,1,2}^{(=,0)} \right) + 2 S\left( V_{0,1,1}^{(=,0)} \right) \cdot S_{=,0} + S\left( V_{0,1,0}^{(=,0)} \right) \cdot (S_{=,0})^2 +2 S\left( V_{0,3,0}^{(=,0)} \right) \right).
\end{align*}\smallskip

Consider all the terms of the right-hand side separately, from right to left. 
\begin{itemize}
\item First, $S\left( V_{0,3,0}^{(=,0)} \right) = (S_{=,0})^3$ since the conditions on the summation indices force $\ell_1=m_1$, $\ell_2=m_2$ and $\ell_3=m_3$.
\item Second, unpacking the definition, we see that $S\left( V_{0,1,0}^{(=,0)} \right)=S_{=,0}$.
\item Third, we observe that $S\left( V_{0,1,1}^{(=,0)} \right) = S(P_4)$ which was computed in Section \ref{ssec:ex_P4}.
\item Finally, by \eqref{eq:V1}, we get
\begin{align*}
S\left( V_{0,1,2}^{(=,0)} \right) = t^{-2} S\left( V_{0,1,1}^{(=,0)} \right) + t^{-2} S\left( V_{0,1,0}^{(=,0)} \right) - t^{-2} S\left( U_{0,1,0}^{(=,0)} \right) - t^{-2} S\left( \overline{U}_{0,1,0}^{(=,0)} \right).
\end{align*}\smallskip
\end{itemize}
But we have $S\left( U_{0,1,0}^{(=,0)} \right) =S(P_3)$,
which has been computed in \cref{ssec:ex_P3}.
Moreover, $S\left( \overline{U}_{0,1,0}^{(=,0)} \right) = S_{=,0}$ (the conditions force here the variable indexed by the root vertex to be equal to $0$), which gives us
\begin{align*}
S\left( V_{0,1,2}^{(=,0)} \right) = t^{-2} S\left( P_4 \right) - t^{-2} S(P_3) .
\end{align*}
We can finally write:

\begin{align*}
S(T) &= S\left( V_{2,0,1}^{(=,0)} \right)\\
&= \frac{1}{3} \left[ S\left( V_{0,1,2}^{(=,0)} \right) + 2 S\left( V_{0,1,1}^{(=,0)} \right) \cdot S_{=,0} + S\left( V_{0,1,0}^{(=,0)} \right) \cdot (S_{=,0})^2 +2 S\left( V_{0,3,0}^{(=,0)} \right) \right] - S\left( V_{2,0,0}^{(=,0)} \right) \cdot S_{=,0}\\
&= \frac{1}{3} \left[ t^{-2} S\left( P_4 \right) - t^{-2} S(P_3) + 2 S\left( P_4 \right) S_{=,0} +3(S_{=,0})^3 \right] - (S_{=,0})^3\\
&= \frac{1}{3}  S\left( P_4 \right) (t^{-2}+2\, S_{=,0}) - \frac{t^{-2}} {3} S\left( P_3 \right).
\end{align*}
Using the expression for $S_{=,0}$, $S(P_3)$ and $S(P_4)$ given in  \cref{ssec:ex_P2,ssec:ex_P3,ssec:ex_P4}
respectively, we get
\[S(T)=\frac{ H_1^{3}+3 H_1^{2}- \left(16 t^2 +1\right) H_1 +32 t^2 H_2 -(48 t^2 +3)}{192 t^{6}}
, \]
where
\[
H_1 = {}_{2}^{}{{}{{}{F_{1}^{}}}}\big(-\tfrac12,-\tfrac12; 1;16t^2\big)
\quad \text{and} \quad
H_2 = {}_{2}^{}{{}{{}{F_{1}^{}}}}\big(-\tfrac12,\tfrac12; 2;16t^2\big) .
\]
Using $H_1(1/4)=\frac4{\pi}$ and $H_2(1/4)=\frac8{3\pi}$, this specializes to
\begin{align*}
S(T)(1/4) = -128 - \frac{512}{9 \pi} + \frac{1024}{\pi^2} + \frac{4096}{3 \pi^3}.
\end{align*}

\section{Explicit sums for trees with at most 7 vertices}\label{appendix:explicit_sums}

Up to graph isomorphism, there are  24 graphs with at most 7 vertices:
1 tree with 2 vertices,
1 tree with 3 vertices,
2 trees with 4 vertices,
3 trees with 5 vertices,
6 trees with 6 vertices
and
11 trees with 7 vertices.
They are depicted in \cref{fig:trees2-6,fig:trees7}, themselves borrowed from Chapter 1 of Steinbach's ``Field Guide to Simple Graphs''~\cite{Steinbach99}.

We present explicit expressions for the sums $S(T)$ and their evaluations $S(T)\left( \frac14\right)$ for these 24 graphs. As predicted by~\cref{thm:sum-pi}, the sums $S(T)$ are expressed in terms of the two hypergeometric functions
\[
h_1 \coloneqq H_1(\sqrt{t}) = {}_{2}^{}{{}{{}{F_{1}^{}}}}\big(-\tfrac12,-\tfrac12; 1;16t\big)
\quad \text{and} \quad
h_2 \coloneqq H_2(\sqrt{t}) = {}_{2}^{}{{}{{}{F_{1}^{}}}}\big(-\tfrac12,\tfrac12; 2;16t\big),
\]
while their evaluation $S(T)\left( \frac14\right)$ are given as polynomials in $1/\pi$.
They were obtained using the (algebraic) algorithm mentioned in the first paragraph of~\cref{rem:algo_variant}.

\medskip 
\subsection{Trees with 2 vertices}
\subsubsection{Tree $T_{2,a}$}

\begin{eqnarray*}
S(T)(\sqrt{t}) 
& = & 
\frac{h_1 - 1}{4 t}
=
1+t +4 t^{2}+25 t^{3}+196 t^{4}+1764 t^{5}+17424 t^{6}+\cdots, \\
S(T)\left( \frac14\right)
& =& 
\frac{16}{\pi}-4 
\approx 1.092958
.
\end{eqnarray*}

\begin{figure}[t]
  \begin{center}
	  \begin{tabular}{lllll}
    \includegraphics[scale=0.3,angle=90]{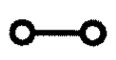} &  
    \qquad \includegraphics[scale=0.3,angle=90]{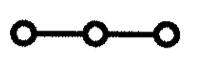} &    
	\qquad \includegraphics[scale=0.3,angle=90]{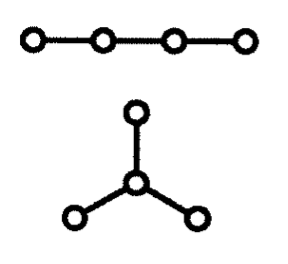}  & 
    \qquad \includegraphics[scale=0.3,angle=90]{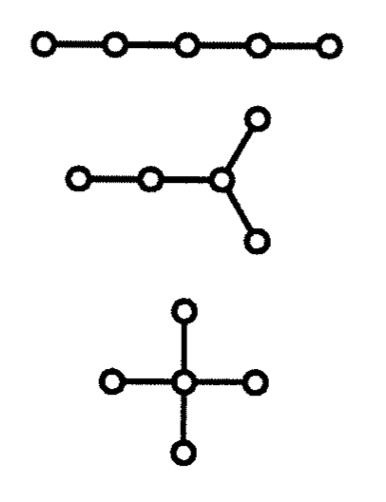} & 
	\qquad \includegraphics[scale=0.3,angle=90]{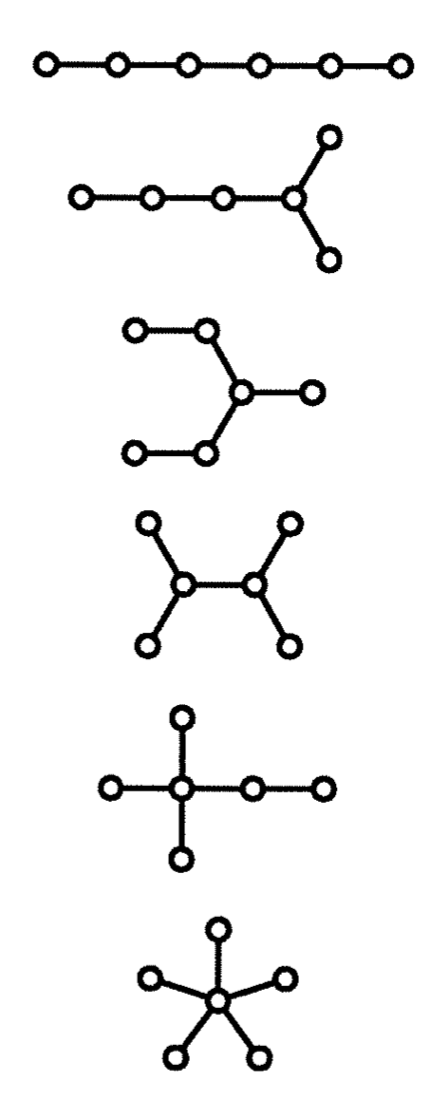}
	  \end{tabular}
  \end{center}
  \caption{Trees with 2, 3, 4, 5 and 6 vertices; they are numbered $T_{2,a}, \ldots, T_{6,f}$.}
  \label{fig:trees2-6}
\end{figure}

\medskip 
\subsection{Trees with 3 vertices}
\subsubsection{Tree $T_{3,a}$}

\begin{eqnarray*}
S(T)(\sqrt{t}) 
& =  & 
\frac{1- h_2}{2 t}
=
1+2 t +10 t^{2}+70 t^{3}+588 t^{4}+5544 t^{5}+56628 t^{6}+\cdots, \\
S(T)\left( \frac14\right)
& = & 
-\frac{64}{3 \pi} + 8
\approx 1.209389
.
\end{eqnarray*}

\medskip 
\subsection{Trees with 4 vertices}
\subsubsection{Tree $T_{4,a}$}
\begin{eqnarray*}
S(T)(\sqrt{t}) 
& = & 
\frac{h_1^{2}+2 h_1-16 t -3}{32 t^{2}} \\
& = & 
1+3 t +17 t^{2}+127 t^{3}+1111 t^{4}+10772 t^{5}+112305 t^{6}+\cdots,\\
S(T)\left( \frac14\right)
& = & 
\frac{128}{\pi^{2}}+\frac{64}{\pi}-32 
\approx 1.340944
.
\end{eqnarray*}

\subsubsection{Tree $T_{4,b}$}\label{sec:star_s=3}
\begin{eqnarray*}
S(T)(\sqrt{t}) 
& = & 
\frac{ h_1 - 
\left(6 t +1\right) \, h_2 }{20 t^{2}}
 \\
& = & 
1+3 t +18 t^{2}+140 t^{3}+1260 t^{4}+12474 t^{5}+132132 t^{6}+\cdots, \\
S(T)\left( \frac14\right)
& = & 
\frac{64}{15 \pi} 
\approx 1.358122
.
\end{eqnarray*}

\medskip 
\subsection{Trees with 5 vertices}
\subsubsection{Tree $T_{5,a}$}
\begin{eqnarray*}
S(T)(\sqrt{t}) 
& = & 
\frac{ h_1-4 t -1}{4 t^{2}}
 \\
& = & 
1+4 t +25 t^{2}+196 t^{3}+1764 t^{4}+17424 t^{5}+184041 t^{6}+\cdots, \\
S(T)\left( \frac14\right)
& =& 
\frac{256}{\pi}-80 
\approx 1.487331
.
\end{eqnarray*}

\subsubsection{Tree $T_{5,b}$}
\begin{eqnarray*}
S(T)(\sqrt{t}) 
& = & 
\frac{- h_1 h_2
- h_1
- h_2+ 4 t +3}{8 t^{2}} \\
& = & 
1+4 t +26 t^{2}+211 t^{3}+1952 t^{4}+19708 t^{5}+211880 t^{6}+\cdots,  \\
S(T)\left( \frac14\right)
& = & 
-\frac{1024}{3 \pi^{2}}-\frac{640}{3 \pi}+104 
\approx 1.509594
.
\end{eqnarray*}

\subsubsection{Tree $T_{5,c}$}\label{sec:star_s=4}
\begin{eqnarray*}
S(T)(\sqrt{t}) 
& = & 
\frac{
\left(2 t +1\right)
h_1 
- 
\left(32 t^{2}+8 t +1\right)
h_2}{140 t^{3}} \\
& = & 
1+4 t +28 t^{2}+240 t^{3}+2310 t^{4}+24024 t^{5}+264264 t^{6}+ \cdots, \\
S(T)\left( \frac14\right)
& = & 
\frac{512}{105 \pi} 
\approx 1.552140
.
\end{eqnarray*}

\medskip 
\subsection{Trees with 6 vertices}
\subsubsection{Tree $T_{6,a}$}
\begin{eqnarray*}
S(T)(\sqrt{t}) 
& = & 
\frac{
h_1^{3}
-3 h_1^{2}
- \left(112 t +13\right) h_1 
-256 t h_2  +432 t +15}{384 t^{3}}
 \\
& = & 
1+5 t +34 t^{2}+278 t^{3}+2563 t^{4}+25701 t^{5}+274210 t^{6}+ \cdots,\\
S(T)\left( \frac14\right)
& = & 
\frac{2048}{3 \pi^{3}}-\frac{512}{\pi^{2}}-\frac{11776}{9 \pi}+448
\approx 1.649799
 .
\end{eqnarray*}

\subsubsection{Tree $T_{6,b}$}
\begin{eqnarray*}
S(T)(\sqrt{t}) 
& = & 
-\frac{h_2+2 t -1}{4 t^{2}}
 \\
& = & 
1+5 t +35 t^{2}+294 t^{3}+2772 t^{4}+28314 t^{5}+306735 t^{6}+\cdots,\\
S(T)\left( \frac14\right)
& = & 
-\frac{512}{3 \pi} + 56
\approx 1.675113
 .
\end{eqnarray*}

\subsubsection{Tree $T_{6,c}$} \label{ssec:T6c}
\begin{eqnarray*}
S(T)(\sqrt{t}) 
& = & 
\frac{ h_1^{3}+3 h_1^{2}- \left(16 t +1\right) h_1 +32 t h_2 -(48 t +3)}{192 t^{3}}
 \\
& = & 
1+5 t +35 t^{2}+295 t^{3}+2794 t^{4}+28671 t^{5}+311963 t^{6}+\cdots, \\
S(T)\left( \frac14\right)
& = & 
\frac{4096}{3 \pi^{3}}+\frac{1024}{\pi^{2}}-\frac{512}{9 \pi}-128
\approx 1.678691
 .
\end{eqnarray*}

\subsubsection{Tree $T_{6,d}$}
\begin{eqnarray*}
S(T)(\sqrt{t}) 
& = & 
\frac{ h_1^{2}+8 t h_2^{2}  +  \left(8 t +2\right) h_1+16 t h_2  - (48 t +3)}{32 t^{3}}
 \\
& = & 
1+5 t +36 t^{2}+311 t^{3}+3004 t^{4}+31313 t^{5}+345064 t^{6}+\cdots, \\
S(T)\left( \frac14\right)
& = & 
\frac{22528}{9 \pi^{2}}+\frac{4864}{3 \pi}-768
\approx 1.704609
 .
\end{eqnarray*}

\subsubsection{Tree $T_{6,e}$}
\begin{eqnarray*}
S(T)(\sqrt{t}) 
& = & 
\frac{h_1^{2} -\left(6 t +1\right) h_1 h_2 + 10 t h_2^{2} + h_1+ \left(14 t-1 \right) h_2+ 40 t^{2}-30 t}{80 t^{3}}
 \\
& = & 
1+5 t +37 t^{2}+327 t^{3}+3214 t^{4}+33954 t^{5}+378130 t^{6}+\cdots, \\
S(T)\left( \frac14\right)
& = & 
\frac{13312}{45 \pi^{2}}+\frac{2816}{15 \pi}-88
\approx 1.730434
 .
\end{eqnarray*}

\subsubsection{Tree $T_{6,f}$}\label{sec:star_s=5}
\begin{eqnarray*}
S(T)(\sqrt{t}) 
& = & 
\frac{  \left(10 t^{2}+1\right) h_1-  \left(160 t^{3}+30 t^{2}+6 t +1\right) h_2}{840 t^{4}}
 \\
& = & 
1+5 t +40 t^{2}+375 t^{3}+3850 t^{4}+42042 t^{5}+480480 t^{6}+\cdots, \\
S(T)\left( \frac14\right)
& = & 
\frac{256}{45 \pi}
\approx 1.810830
 .
\end{eqnarray*}

\begin{figure}[t]
  \begin{center}
    \includegraphics[scale=0.5,angle=90]{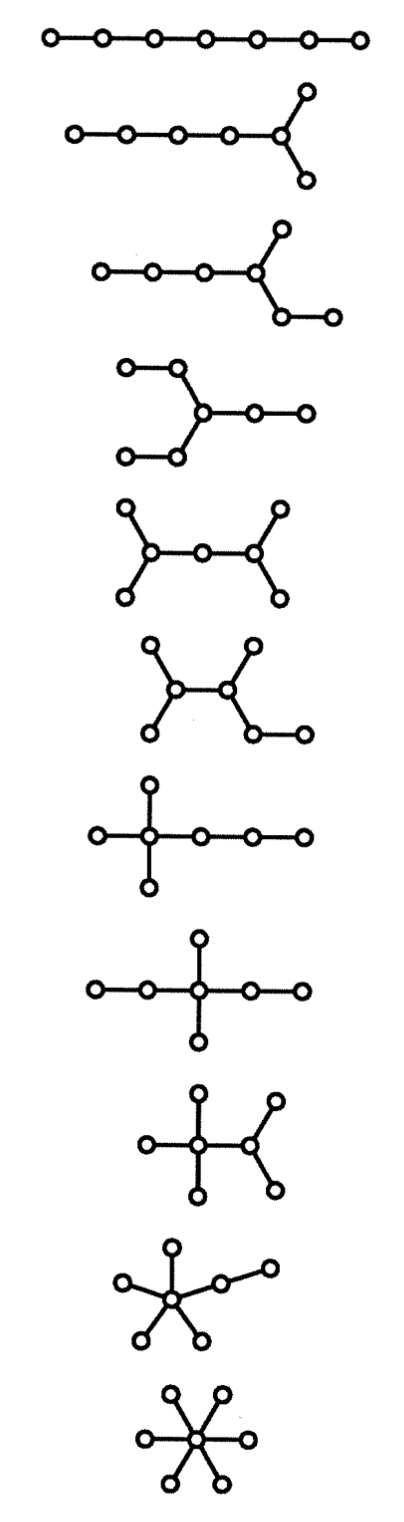}
  \end{center}
  \caption{Trees with 7 vertices, numbered $T_{7,a}, \ldots, T_{7,k}$.}
  \label{fig:trees7}
\end{figure}

\medskip 

\subsection{Trees with 7 vertices}
\subsubsection{Tree $T_{7,a}$}
\begin{eqnarray*}
S(T)(\sqrt{t}) 
& = & 
\frac{5 h_1^{2} + 6 h_1 + \left(64 t +4\right) h_2 - (120 t +15) }{80 t^{3}} \\
& = & 
1+6 t +44 t^{2}+374 t^{3}+3526 t^{4}+35850 t^{5}+385944 t^{6}+\cdots, \\
S(T)\left( \frac14\right)
& = & 
\frac{4096}{\pi^{2}} +\frac{34816}{15 \pi} -1152 
\approx 
1.830035
 .
\end{eqnarray*}

\subsubsection{Tree $T_{7,b}$}
\begin{eqnarray*}
S(T)(\sqrt{t}) 
& = & 
\frac{ \left(40 t +54\right) h_1 +  \left(176 t +11\right) h_2 + 5 h_1^{2}
+10 h_1 h_2 - 5 h_1^{2} h_2-(440 t +75) }{320 t^{3}}\\
& = & 
1+6 t +45 t^{2}+391 t^{3}+3756 t^{4}+38790 t^{5}+423086 t^{6}+\cdots, \\
S(T)\left( \frac14\right)
& = & 
-\frac{8192}{3 \pi^{3}}
+\frac{7168}{3 \pi^{2}}
+\frac{54656}{15 \pi}
-1312
\approx 1.858234
 .
\end{eqnarray*}

\subsubsection{Tree $T_{7,c}$}
\begin{eqnarray*}
S(T)(\sqrt{t}) 
& = & 
\frac{3- h_1-2 h_2}{4 t^{2}}
\\
& = & 
1+6 t +45 t^{2}+392 t^{3}+3780 t^{4}+39204 t^{5}+429429 t^{6}+\cdots, \\
S(T)\left( \frac14\right)
& = & 
-\frac{1792}{3 \pi} + 192
\approx 1.862895
 .
\end{eqnarray*}

\subsubsection{Tree $T_{7,d}$}
\begin{eqnarray*}
S(T)(\sqrt{t}) 
& = & 
\frac{3 h_1^{2} + 32 t h_2 -2 h_1  - (48 t +1) }{32 t^{3}}\\
& = & 
1+6 t +45 t^{2}+393 t^{3}+3804 t^{4}+39618 t^{5}+435773 t^{6}+\cdots, \\
S(T)\left( \frac14\right)
& = & 
\frac{6144}{\pi^{2}} -\frac{1024}{3 \pi} -512
\approx 1.867578
 .
\end{eqnarray*}

\subsubsection{Tree $T_{7,e}$}
\begin{eqnarray*}
S(T)(\sqrt{t}) 
& = & 
- \frac{ h_2^{2} + 4 t -1}{4 t^{2}}
\\
& = & 
1+6 t +46 t^{2}+408 t^{3}+3988 t^{4}+41788 t^{5}+461378 t^{6}+\cdots, \\
S(T)\left( \frac14\right)
& = & 
-\frac{4096}{9 \pi^{2}}+48
\approx 1.887604
 .
\end{eqnarray*}

\subsubsection{Tree $T_{7,f}$}
\begin{eqnarray*}
S(T)(\sqrt{t}) 
& = & 
\frac{ 6 h_1 - \left(16 t +1\right) h_2 - 5 h_1^{2}- 10 h_1 h_2- 5 h_1^{2} h_2 + 80 t + 15}{320 t^{3}}
\\
& = & 
1+6 t +46 t^{2}+410 t^{3}+4035 t^{4}+42589 t^{5}+473562 t^{6}+\cdots, \\
S(T)\left( \frac14\right)
& = & 
-\frac{8192}{3 \pi^{3}}
-\frac{7168}{3 \pi^{2}}
+\frac{3584}{15 \pi}
+256
\approx 1.896571
 .
\end{eqnarray*}

\subsubsection{Tree $T_{7,g}$}
\begin{eqnarray*}
S(T)(\sqrt{t}) 
& = & 
\frac{5 t h_2^{2} +  h_1-  \left(6 t +1\right)h_2  -5 t }{40 t^{3}}
\\
& = & 
1+6 t +47 t^{2}+426 t^{3}+4243 t^{4}+45172 t^{5}+505475 t^{6}+\cdots, \\
S(T)\left( \frac14\right)
& = & 
\frac{2048}{9 \pi^{2}}
+\frac{512}{15 \pi}
-32
\approx 1.921176
 .
\end{eqnarray*}

\subsubsection{Tree $T_{7,h}$}
\begin{eqnarray*}
S(T)(\sqrt{t}) 
& = & 
\frac{ \left(40 t +6\right) h_1 + \left(64 t -1\right) h_2 - 5 h_1^{2}- 10 h_1 h_2- 5 h_1^{2} h_2-(40 t - 15)}{160 t^{3}}
\\
& = & 
1+6 t +47 t^{2}+428 t^{3}+4290 t^{4}+45974 t^{5}+517695 t^{6}+\cdots, \\
S(T)\left( \frac14\right)
& = & 
-\frac{16384}{3 \pi^{3}}
-\frac{14336}{3 \pi^{2}}
+\frac{5376}{5 \pi}
+320
\approx 1.930247
 .
\end{eqnarray*}

\subsubsection{Tree $T_{7,i}$}
\begin{eqnarray*}
S(T)(\sqrt{t}) 
& = & 
\frac{\left(24 t +4\right) h_2^{2}  - 34 h_1 -  \left(56 t +16\right) h_2 - 5 h_1^{2} - 24 h_1 h_2 + 240 t + 75}{160 t^{3}}
\\
& = & 
1+6 t +48 t^{2}+445 t^{3}+4524 t^{4}+49033 t^{5}+557248 t^{6}+\cdots, \\
S(T)\left( \frac14\right)
& = & 
-\frac{342016}{45 \pi^{2}}
-\frac{24064}{5 \pi}
+2304
\approx 1.961159
 .
\end{eqnarray*}

\subsubsection{Tree $T_{7,j}$}
\begin{eqnarray*}
S(T)(\sqrt{t}) 
& = & 
\frac{  
 \left(4 t +2\right) h_1^{2} 
 -2 \left(32 t^{2}+22 t +1\right)  h_1 h_2 +  \left(168 t^{2}+28 t \right) h_2^{2}
 }{}\\
&  & 
\frac{- \left(24 t -2\right) h_1 + \left(104 t^{2}+12 t -2\right) h_2 + 560 t^{3}}{1120 t^{4}}\\
& = & 
1+6 t +50 t^{2}+480 t^{3}+5014 t^{4}+55504 t^{5}+641436 t^{6}+\cdots, \\
S(T)\left( \frac14\right)
& = & 
-\frac{4096}{315 \pi^{2}}
-\frac{512}{35 \pi}
+8
\approx 2.026084
 .
\end{eqnarray*}

\subsubsection{Tree $T_{7,k}$}\label{sec:star_s=6}
\begin{eqnarray*}
S(T)(\sqrt{t}) 
& = & 
\frac{  \left(48 t^{3}+7 t^{2}-3 t +1\right) h_1 -  \left(768 t^{4}+130 t^{3}+9 t^{2}+3 t +1\right) h_2}{4620 t^{5}}
\\
& = & 
1+6 t +54 t^{2}+550 t^{3}+6006 t^{4}+68796 t^{5}+816816 t^{6}+\cdots, \\
S(T)\left( \frac14\right)
& = & 
\frac{23552}{3465 \pi} 
\approx 2.163589
 .
\end{eqnarray*}

\section*{Acknowledgements}
The second author thanks Eugene Zima for useful exchanges at an early stage of this project. \revision{We are also grateful to the anonymous referees who read thoroughly this long paper, and provided helpful suggestions to improve its presentation.}

The collaboration that led to the current article started at the 
\href{https://www.mfo.de/occasion/2250}{\emph{Workshop on Enumerative Combinatorics}} held in December 2022 at the Mathematisches Forschungsinstitut Oberwolfach.
During the workshop, the last part of~\cref{thm:sum-pi}, as well as \cref{coro:proba_meanders}, were presented as open problems, see the \href{https://publications.mfo.de/bitstream/handle/mfo/4019/OWR_2022_57.pdf}{Oberwolfach Report No. 57/2022, pages 3256 and 3298}. We would like to thank the organizers of the workshop and the MFO's scientific and administrative staff for making all of this possible.

This work has been supported by the French--Austrian project 
\href{https://anr.fr/Project-ANR-22-CE91-0007}{EAGLES}
(ANR-22-CE91-0007 \& FWF I6130-N)
and the ANR project \href{https://anr-louccoum-anr-louccoum.apps.math.cnrs.fr/}{LOUCCOUM} (ANR-24-CD40-7809).

\bibliographystyle{BoFeTh25}

\newcommand{\etalchar}[1]{$^{#1}$}

\end{document}